\documentclass[reqno, 11pt]{amsart}
\usepackage{amsmath, amsthm, amssymb, mathrsfs}
\usepackage[a4paper, left=3cm, right=3cm, top=3cm, bottom = 2.5cm]{geometry}
\usepackage[dvipsnames]{xcolor}
\usepackage{multirow}
\usepackage{setspace}
\usepackage{tikz-cd}
\usetikzlibrary{decorations.pathreplacing}
\usepackage{enumitem}
\usepackage{hyperref}
\usepackage{aligned-overset}
\usepackage{subfiles}

\theoremstyle{plain}
\newtheorem{theo}{Theorem}[section]
\newtheorem{lemma}[theo]{Lemma}
\newtheorem{cor}[theo]{Corollary}

\newtheorem{prop}[theo]{Proposition}

\newtheorem{defi}[theo]{Definition}

\theoremstyle{remark}
\newtheorem{remark}[theo]{Remark}

\newcommand{\N}{\mathbb{N}}

\newcommand{\C}{\mathbb{C}}
\newcommand{\M}{\mathbb{M}}
\newcommand{\D}{\mathbb{D}}

\newcommand{\norm}[1]{\|#1 \|}
\newcommand{\tZ}[2]{\tilde{Z}_{#1,#2}}
\newcommand{\tD}[1]{\tilde{D}_{#1,#1+1}}
\newcommand{\piph}[1]{\pi_\varphi(e_{#1})}
\newcommand{\pips}[2]{\pi_\psi(e_{#1,#2})}
\newcommand{\Ad}[1]{\mathrm{Ad}(#1 )}
\newcommand{\cstarc}[1]{c_{#1}^*\hspace{0.05em} c_{#1}}
\newcommand{\ctil}[2]{\tilde{c}_{#1}^{(#2)}}
\newcommand{\stil}[1]{\tilde{s}^{(#1)}}
\newcommand{\tildeZ}{\tilde{\mathcal{Z}}}
\newcommand{\Id}[1]{\langle #1 \rangle}

\newcommand{\bc}{\bar{c}}
\newcommand{\bs}{\bar{s}}
\newcommand{\tinyfrac}[2]{{\text{\scriptsize$\frac{#1}{#2}$}}}

\newenvironment{nalign}{
    \begin{equation}
    \begin{aligned}
}{
    \end{aligned}
    \end{equation}
    \ignorespacesafterend
}

\allowdisplaybreaks

\begin{document}

\title{{\textbf{\sc A $\mathrm{C}^*$-diagonal in the Jiang--Su algebra via entangled matrix cones}}}
\author[L. Obermeyer]{Lukas Obermeyer}
\address{Lukas Obermeyer, Mathematical Institute, University of M\"unster, Ein\-stein\-strasse 62, 48149 M\"unster, Germany}
\email{lukas.obermeyer@uni-muenster.de}
\author[W. Winter]{Wilhelm Winter}
\address{Wilhelm Winter, Mathematical Institute, University of M\"unster, Ein\-stein\-strasse 62, 48149 M\"unster, Germany}
\email{wwinter@uni-muenster.de}

\maketitle

\begin{abstract}
 We construct the Jiang--Su algebra $\mathcal{Z}$ as an inductive limit of dimension drop algebras, describing the latter as entangled matrix cones to explicitly define the connecting $^*$-homomorphisms. This construction gives rise to a $\mathrm{C}^*$-diagonal in $\mathcal{Z}$ with one-dimensional spectrum which is not locally connected. Along the way, we give a new characterisation of normalisers in Cartan pairs in terms of state excision.
\end{abstract}

\section[]{Introduction}
\renewcommand{\thetheo}{\Alph{theo}}

The notion of $\mathrm{C}^*$-diagonals in $\mathrm{C}^*$-algebras---inspired by  Cartan subalgebras in von Neumann algebras studied by Feldman and Moore (\cite{RF1}, \cite{RF2}, \cite{RF3})---was introduced by Kumjian in \cite{Kum} and later generalised by Renault in \cite{Ren1} to the notion of Cartan subalgebras in $\mathrm{C}^*$-algebras. They showed that any separable Cartan pair, i.e.\@ a separable $\mathrm{C}^*$-algebra with a Cartan subalgebra, arises uniquely as the groupoid $\mathrm{C}^*$-algebra of a (possibly twisted) second countable, locally compact Hausdorff étale groupoid that is effective (principal in the case of $\mathrm{C}^*$-diagonals).  This yields a strong connection between $\mathrm{C}^*$-algebras and topological dynamics. In particular, Barlak and Li in \cite{BaLi1}, \cite{BaLi2}, building on seminal work by Tu in \cite{Tu}, used this connection to characterise Rosenberg's and Schochet's Universal Coefficient Theorem (UCT): A separable nuclear $\mathrm{C}^*$-algebra satisfies the UCT (i.e.\@ it is homotopy equivalent in a weak sense to an abelian $\mathrm{C}^*$-algebra) if and only if it contains a Cartan subalgebra. This intricate relation between Cartan subalgebras and the UCT sparked renewed interest in Cartan subalgebras from the Elliott classification programme, as the infamous UCT problem---asking whether any nuclear, separable $\mathrm{C}^*$-algebra satisfies the UCT---remains unsolved. \\

The study of Cartan subalgebras and $\mathrm{C}^*$-diagonals has led to numerous existence results, often by realising them as groupoid $\mathrm{C}^*$-algebras. Spielberg showed that any non-unital UCT Kirchberg algebra arises as a groupoid $\mathrm{C}^*$-algebra with groupoid models originating from graphs, and his techniques can be extended to the unital case (\cite{LiRe}, \cite{CFaH}). Combined with X.\@ Li's breakthrough result in \cite{Li1} that all classifiable stably finite $\mathrm{C}^*$-algebras admit a $\mathrm{C}^*$-diagonal, we have existence of Cartan subalgebras in all classifiable $\mathrm{C}^*$-algebras. There are further existence results in several cases beyond the classifiable setting (\cite{LiRa}, \cite{Oy}). However, even when we know existence of some Cartan subalgebra or $\mathrm{C}^*$-diagonal, there are only few instances in which we understand the entirety of the underlying dynamical structure of a $\mathrm{C}^*$-algebra, e.g.\@ in which we can achieve uniqueness or classification of Cartan subalgebras or $\mathrm{C}^*$-diagonals; see \cite{BaRa}, \cite{Li2}, \cite{LiRe} for results in rather restricted settings. Thus, the current goal is to find more (classes of) explicit examples of Cartan subalgebras and $\mathrm{C}^*$-diagonals and techniques to obtain these in prominent $\mathrm{C}^*$-algebras to add to the understanding of their underlying dynamical structure. This was recently addressed for the Cuntz algebras (\cite{SiWi}, \cite{EvSi1}, \cite{EvSi2}), and the CAR-algebra $M_{2^\infty}$ (\cite{KoWi}). In this article, we focus on the Jiang--Su algebra $\mathcal{Z}$.\\

The Jiang--Su algebra was first constructed in \cite{JiSu} as an inductive limit of dimension drop algebras. It plays a key role in the Elliott classification programme, which culminated after three decades of work in the classification of separable, simple, nuclear $\mathrm{C}^*$-algebras with UCT that are $\mathcal{Z}$-stable, i.e.\@ tensorially absorb $\mathcal{Z}$, by $K$-theoretic and tracial data (cf.\@ \cite{GongLinNiu}, \cite{CGSTW}, \cite{Elliott-Gong-Lin-Niu}). Naturally, the Jiang–Su algebra has been studied in great detail throughout the programme, with several characterisations and constructions developed over the years, each of them technically involved (e.g.\@  \cite{RWi}, \cite{JaWi}, \cite{Sch}).\\
A dynamical construction of $\mathcal{Z}$ was given by Deeley, Putnam, and Strung in \cite{DePuStr}, thereby providing the first example of a $\mathrm{C}^*$-diagonal in the Jiang--Su algebra. Their approach was to construct a minimal dynamical system on a homologically point-like infinite metric space and consider the orbit-breaking equivalence relation of this system over a point in the underlying space. The resulting groupoid $\mathrm{C}^*$-algebra matches the Elliott invariant of $\mathcal{Z}$ and falls in the scope of the classification programme, hence is isomorphic to the Jiang--Su algebra. However, the spectrum of the resulting $\mathrm{C}^*$-diagonal has dimension $>1$, whereas the original representation of $\mathcal{Z} $ as an inductive limit of dimension drop algebras suggests the existence of a $\mathrm{C}^*$-diagonal with one-dimensional spectrum. Later, a construction of a $\mathrm{C}^*$-diagonal in $\mathcal{Z}$ with spectrum homeomorphic to a one-dimensional Peano continuum, which can be arranged to be the Menger curve, was indeed exhibited by X.\@ Li; see \cite{Li1}, \cite{Li2}. This was an application of the sophisticated machinery developed in \cite{Li1} to give an inductive system of groupoids with limit yielding $\mathcal{Z}$ as its groupoid $\mathrm{C}^*$-algebra, applying the classification theorem mentioned above. An isomorphic model was obtained in \cite{AuMi}.\\
Our main result is a new construction of a $\mathrm{C}^*$-diagonal in the Jiang--Su algebra that is constructed explicitly with $\mathrm{C}^*$-algebraic generators and relations. Our approach yields an inductive limit of dimension drop algebras with techniques inspired by \cite{JaWi}; this limit will be isomorphic to the Jiang--Su algebra using the classification result from \cite{JiSu}. The construction is very explicit in the sense that we define the family of connecting $^*$-homomorphisms for the system by explicitly describing where a designated set of generating elements of the respective dimension drop algebra is mapped. Although technical, we can precisely point out all the required input (see Equations (\ref{eq:choiceofL'}), (\ref{eq:defictilde:a}), (\ref{eq:defictilde:b}), (\ref{eq:defi:tildec}), (\ref{eq:defi:stilde:a}), (\ref{eq:defi:stilde:b}), (\ref{eq:defi:tildes}), (\ref{eq:choiceofLn}) in Subsection \ref{subsec:Theconstruction}) and even estimate the size of the dimension drop algebras throughout the construction. This new presentation of the Jiang--Su algebra then naturally gives rise to a $\mathrm{C}^*$-diagonal in $\mathcal{Z}$.
This $\mathrm{C}^*$-diagonal has one-dimensional spectrum which is not locally connected, distinguishing it from the models in \cite{DePuStr}, \cite{Li2}.\\

Let us now outline the construction in more detail, beginning with the precise definition of Cartan subalgebras and $\mathrm{C}^*$-diagonals from \cite{Kum}, \cite{Ren1}.

\begin{defi}
\label{defi:diagonals}
Let $A$ be a $\mathrm{C}^*$-algebra and $ D\subset A$ an abelian $\mathrm{C}^*$-subalgebra. Then, $D$ is called Cartan subalgebra if the following hold:
\begin{enumerate}
\item[(0)] $D$ contains an approximate unit of $A$.
\item[(1)] $D$ is maximal abelian.
\item[(2)] $D$ is regular, i.e.\@ the set of normalisers $$N_A(D) := \{ x \in A: x^*Dx \subset D,\, x D x^* \subset D\}$$ generates $A$ as a $\mathrm{C}^*$-algebra.
\item[(3)] There exists a faithful conditional expectation $P:A \to D$.
\end{enumerate}
 $D$ is called $\mathrm{C}^*$-diagonal in $A$ if additionally
 \begin{enumerate}
     \item[(4)] $D$ satisfies the pure state extension property relative to $A$, i.e.\@ any pure state on $D$ extends uniquely to a (necessarily pure) state on $A$.
 \end{enumerate}
\end{defi}

To construct $\mathcal{Z}$ as an inductive limit of dimension drop algebras, which are given as
\begin{equation}
Z_{p,q}:= \{ f \in C( [0,1], \M_p \otimes \M_q) : f(0) \in \M_p \otimes 1_q,\, f(1) \in 1_p \otimes \M_q \}, \,\,\, p,q \geq 1 \label{eq:defidimdrop},
\end{equation}
we use the universal $\mathrm{C}^*$-algebra picture from \cite{RWi}, presenting dimension drop algebras of the form $Z_{L,L+1}$ as \emph{entangled matrix cones}. We use this picture to set up explicit connecting $^*$-homomorphisms $\Phi_{L_n,L_{n+1}}:Z_{L_n, L_{n}+1} \to Z_{L_{n+1},L_{n+1}+1}$ for a suitable increasing sequence of positive integers $(L_n)_n$ to achieve two things:\\
On the one hand, we build the connecting maps and choose the matrix sizes $(L_n)_n$ such that the inductive limit will be simple with unique tracial state. Then, the limit of this system will be isomorphic to the Jiang--Su algebra by the classification result in \cite{JiSu}.  \\
On the other hand, we identify a $\mathrm{C}^*$-diagonal $\tilde{D}_{L,L+1}\subset Z_{L,L+1}$ for arbitrary $L$ in terms of the defining generators and relations in these dimension drop algebras. The connecting $^*$-homomorphisms are constructed in a way that they preserve this diagonal and its structure. That is, we arrange our construction to fall in the scope of Theorem 3.6 in \cite{BaLi1} (see Theorem \ref{theo:XinLi}  for a $\mathrm{C}^*$-algebraic formulation), saying that Cartan subalgebras and $\mathrm{C}^*$-diagonals are well behaved under taking inductive limits when the connecting $^*$-homomorphisms preserve diagonal elements, normalisers, and intertwine the conditional expectations.\\
In our setting, the main challenge is to arrange that normalisers are preserved throughout the system. To control this, we give a new characterisation of normalisers in separable Cartan pairs in terms of pure state excision (see Theorem \ref{theocharnorm}), which will be of interest for upcoming work.
With all this in place, we get a $\mathrm{C}^*$-diagonal in the inductive limit we ensured to be $\mathcal{Z}$. Altogether, this yields our main result. 

\begin{theo}[see Theorem \ref{theo:constrindlim}, Theorem \ref{theo:diaginZ}]
There is an increasing sequence of integers $(L_n)_n$  and a sequence of $^*$-homomorphisms $\Phi_{L_n,L_{n+1}}: Z_{L_n,L_{n}+1} \to Z_{L_{n+1},L_{n+1}+1} $ described in terms of explicit generators and relations such that $$\mathcal{Z} \cong \varinjlim (Z_{L_n, L_{n}+1},\Phi_{L_n,L_{n+1}})$$  and
\begin{equation*}
\varinjlim (\tD{L_n}, \Phi_{L_n, L_{n+1}}\vert_{\tD{L_n}} ) \subset \mathcal{Z}
\end{equation*}
is a $\mathrm{C}^*$-diagonal.
\end{theo}

This article is structured as follows. In Section \ref{sec:construction}, we construct the inductive system yielding the Jiang--Su algebra as its limit. Here, our construction is explicitly described in the self-contained  Subsection \ref{subsec:Theconstruction}, while conceptual explanations and verifications are given in Subsections \ref{subsec:visual} to \ref{subsec:simplicityanduniquetrace}. In Section \ref{sec:charnorm}, we derive the new characterisation of normalisers in Cartan pairs. This is used in Section \ref{sec:diaginZ} to show that the new construction of $\mathcal{Z}$ indeed yields a $\mathrm{C}^*$-diagonal in $\mathcal{Z}$; an explicit description of the diagonal is contained in the tandem of Proposition \ref{prop:tD} and Subsection \ref{subsec:Theconstruction}. Section \ref{sec:spectrum} is dedicated to the analysis of the spectrum of the resulting diagonal, showing that it is not homeomorphic to the Menger curve.

\subsection*{Acknowledgements} The first-named author would like to thank Andrea Vaccaro and Aleksandra Kwiatkowska for helpful discussions on the spectrum of the constructed $\mathrm{C}^*$-diagonal.
This work was funded by the Deutsche Forschungsgemeinschaft
(DFG, German Research Foundation) under Germany's Excellence Strategy EXC 2044-390685587,
Mathematics M{\"u}nster: Dynamics--Geometry--Structure, by the SFB 1442 of the DFG, and by ERC Advanced Grant 834267 -- AMAREC.

\section[]{Construction of the Jiang--Su algebra $\mathcal{Z}$}
\label{sec:construction}
\renewcommand{\thetheo}{\thesection.\arabic{theo}}
\setcounter{theo}{0}
\numberwithin{theo}{section}
In \cite{JiSu}, the Jiang--Su algebra is constructed as the inductive limit of prime dimension drop algebras with unital $^*$-homomorphisms as connecting maps. Moreover, it follows from the classification result therein that the limit of such a system is isomorphic to $\mathcal{Z}$ if and only if it is simple with unique tracial state (see also Theorem 2.2 in \cite{RWi}). Here, we construct $\mathcal{Z}$ in a new way as an inductive limit of prime dimension drop algebras $Z_{L,L+1}$ that already encodes a $\mathrm{C}^*$-diagonal in $\mathcal{Z}$.

\subsection{{\sc The construction} \nopunct}
\label{subsec:Theconstruction}
\quad \vspace{1ex}\\
In order to define the connecting maps in the inductive system, we will use the following universal $\mathrm{C}^*$-algebra picture from \cite{RWi}; see also Section 2 in \cite{Sato}. We think of this as a presentation using \emph{entangled matrix cones}, as opposed to the one from \cite{JiSu} which uses \emph{commuting matrix cones}.
\begin{prop}
\label{prop:sato}
Let $L \in \N$ and $ X = \{ c_1, \dots, c_L, s \}$. Further, let $\tZ{L}{L+1}$ denote the universal $\mathrm{C}^*$-algebra generated by $X$ with relations $(\mathcal{\tilde{R}}_L)$ given by
\begin{equation}
\label{eq:Therelations}
\tag{$\mathcal{\tilde{R}}_{L}$}
c_1 \geq 0,\hspace{0.5em} c_i c_i^* = c_1^2, \hspace{0.5em} c_i^*c_i \perp c_j^*c_j, \hspace{0.5em}  c_1 s = s, \hspace{0.5em} \sum\limits_{l=1}^L c_l^*c_l + s^*s = 1 \;\;\; \text{ for } 1 \leq i\neq j \leq L. 
\end{equation}
 Then, $\tZ{L}{L+1}$ is isomorphic to the dimension drop algebra $Z_{L,L+1}$ from (\ref{eq:defidimdrop}).
\end{prop}
Observe that the relations (\ref{eq:Therelations}) immediately give us 
\begin{equation}
\label{eq:basicCT}
\norm{c_i}\leq 1,\hspace{0.7em}  \norm{s} \leq 1, \hspace{0.7em} s^2 = 0, \hspace{0.7em}  c_j c_i = 0, \hspace{0.7em}  s^*s \in \mathrm{C}^*(c_1, \dots, c_L)' 
\end{equation}
for all $1 \leq i \leq L$, $2 \leq j \leq L$. Now, recall, for example from \cite{Lor}, that $C_0((0,1], \M_L)$ can be written as the universal $\mathrm{C}^*$-algebra with generators $\{ x_i  :  1 \leq i \leq L \}$ and relations
\begin{equation}
\label{eq:relationscone}
\norm{x_i} \leq 1, \hspace{0.5em} x_1 \geq 0 , \hspace{0.5em} x_{i} x_{i}^* = x_1^2, \hspace{0.5em} x_i^*x_i \perp x_j^*x_j \;\; \; \text{ for } 1\leq i \neq j \leq L.
\end{equation} 
Here, identifying  $C_0((0,1]) \otimes \mathbb{M}_L$ with  $C_0((0,1], \M_L)$, the element $\mathrm{id}^{\frac{1}{2}} \otimes e_{1,i} $ corresponds to $x_i$. In particular, $C_0((0,1], \mathbb{M}_2)$ can be identified with the universal $\mathrm{C}^*$-algebra given as $ \mathrm{C}^*(x : \norm{x}\leq 1,\, x^2=0)$. That is why we view $\tZ{L}{L+1}$ as generated by an $L$-dimensional and a $2$-dimensional cone, which are entangled. In contrast, the description of $Z_{L,L+1}$ as a universal $\mathrm{C}^*$-algebra from \cite{JiSu} arises from the standard definition of dimension drop algebras; see (\ref{eq:defidimdrop}). The generating set is $\{a_1, \ldots, a_L, b_1, \ldots , b_{L+1}\}$ subject to the relations $\mathcal{R}_L$ ensuring that the $a_i$'s and $b_j$'s generate commuting copies of an $L$-dimensional and an $L+1$-dimensional cone, such that the identity functions on the respective intervals add up to the unit. However, the relations (\ref{eq:Therelations}) have the conceptual advantage of featuring no commutation relations, which tend not to be weakly stable in the sense of Loring (\cite{Lor}). \\

Now, take $L \geq 2$ and let us set up $^*$-homomorphisms $\Phi_{L,L'}: \tZ{L}{L+1} \to \tZ{L'}{L'+1}$ for suitable $L' \geq L$ which will serve as the connecting maps in our construction. Choose $K', L'$ depending on $L$ and parameters $M,K \geq 2$ as
\begin{equation}
\label{eq:choiceofL'}
L'=KML^2+ML, \,\,\, K'=KML.
\end{equation}
In the following, when we write $L'(L,M,K)$, we mean that $L'$ is chosen as in (\ref{eq:choiceofL'}) with respect to $L,M,K$.
Write $c_{1,L}, \dots , c_{L,L}, s_L \in \tZ{L}{L+1}$ and $c_{1,L'}, \dots, c_{L',L'}, s_{L'} \in \tZ{L'}{L'+1}$ for the respective generators from Proposition \ref{prop:sato}. The following table gives an overview of the situation.
\vspace{2ex}
\begin{center}
    \begin{tabular}{c|c|c|c}
        Dimension drop & \multirow{2}{6em}{\centering Matrix size}& \multirow{2}{14em}{\centering Canonical generators} & \multirow{2}{4em}{Relations}  \\ 
        algebra &&& \\
         \hline
     \multirow{2}{4em}{ $Z_{L,L+1}$}   & \multirow{2}{4em}{\centering $L$}& \small$(1-\mathrm{id})^{1/2} \otimes e_{1,i} \otimes 1_{L+1} \text{ for } 1 \leq i \leq L $, & \multirow{2}{4em}{\centering$(\mathcal{R}_L)$} \\
      & & \small $\mathrm{id}^{1/2} \otimes 1_L \otimes e_{1,j} \text{ for } 1 \leq j \leq L+1$ & \\
      \hline
       \multirow{2}{4em}{\centering $\tZ{L}{L+1}$} & \multirow{2}{4em}{\centering$L$}& \multirow{2}{4em}{\centering $c_{1,L}, \ldots, c_{L,L}, s_L$} & \multirow{2}{4em}{\centering (\ref{eq:Therelations})} \\
       &&& \\
      \hline
        \multirow{2}{4em}{\centering $\tZ{L'}{L'+1}$} & \multirow{2}{8em}{\centering $L' = KML^2+ML$}& \multirow{2}{4em}{\centering $c_{1,L'}, \ldots, c_{L',L'}, s_{L'}$} & \multirow{2}{4em}{\centering (\hyperref[eq:Therelations]{$\mathcal{\tilde{R}}_{L'}$})} \\
       &&& 
          \end{tabular}
\end{center}

\vspace{2ex}
To define $^*$-homomorphisms between dimension drop algebras, we now use the second and third row of the table above instead of the standard picture from the first row. To simplify notation, we henceforth write 
\begin{align*}
&\bar{c}_1, \ldots, \bar{c}_L, \bar{s} \,\text{ for }\, c_{1,L}, \ldots, c_{L,L}, s_L \, \text{ in } \tZ{L}{L+1},\\
&c_1, \ldots, c_{L'},s \,\text{ for } \,c_{1, L'}, \ldots, c_{L',L'}, s_{L'} \,\text{ in } \tZ{L'}{L'+1} .
\end{align*}

The goal now is to find $\tilde{c}_1, \dots, \tilde{c}_{L}, \tilde{s} \in \tZ{L'}{L'+1}$ expressed in terms of $c_1,\ldots, c_{L'}, s \in \tZ{L'}{L'+1}$ satisfying the relations  (\ref{eq:Therelations}), since then we obtain a $^*$-homomorphism 
\begin{equation*}
\label{eq:PhiL,L'}
    \Phi_{L,L'}: \tZ{L}{L+1} \to \tZ{L'}{L'+1},\;\; \bar{c}_i \mapsto \tilde{c}_i, \,\bar{s} \mapsto \tilde{s}.
\end{equation*}
For the explicit definition of $\tilde{c}_1, \dots , \tilde{c}_L, \tilde{s} $, let us establish some auxiliary c.p.c.\@ order zero maps. First, note that there is a $^*$-homomorphism $C_0((0,1], \M_2) \to \tZ{L'}{L'+1}$ mapping $\mathrm{id}^{\frac{1}{2}} \otimes e_{1,2}$ to $s$ since $s^2=0$. This induces a c.p.c.\@ order zero map $\varphi:\M_2 \to \tZ{L'}{L'+1}$ with supporting $^*$-homomorphism denoted by $\pi_\varphi$. Similarly, the elements $c_1, \dots, c_{L'} \in \tZ{L'}{L'+1}$ satisfy the relations of the $L'$-dimensional cone from (\ref{eq:relationscone}). Hence, we get a $^*$-homomorphism $C_0((0,1], \M_{L'}) \to \tZ{L'}{L'+1}$ with $\mathrm{id}^{\frac{1}{2}} \otimes e_{1,i} \mapsto c_i$, which induces a c.p.c.\@ order zero map $\psi:\mathbb{M}_{L'} \to \tZ{L'}{L'+1}$. The respective supporting $^*$-homomorphism is denoted by $\pi_\psi$.\\

With this, we can define the elements $\tilde{c}_1, \dots , \tilde{c}_L, \tilde{s}$ we aim for.
    Let $g,h \in C_0((0,1])$ be piecewise linear functions defined as follows:
\begin{equation*}
    g(t) = \begin{cases}
        4t & \text{if } t \in [0, \frac{1}{4}]\\
        1 & \text{if } t \in [ \frac{1}{4}, 1],
    \end{cases} 
   \,\,\,\,\,
   h(t) = \begin{cases}
        0 & \text{if } t \in [0, \frac{3}{4}],\\
        4(t-\frac{3}{4}) & \text{if } t \in [\frac{3}{4},1].     
    \end{cases} 
\end{equation*}

Then, we define for $1 \leq l \leq L$ with $\pi_\varphi, \pi_\psi$ as above
\begin{align}
    \label{eq:defictilde:a}
    \ctil{l}{1}&= \sum\limits_{i=1}^{K'} \pips{i}{(l-1)(K'+M)+i} g^{\frac{1}{2}}(\cstarc{(l-1)(K'+M)+i}),\\
    \label{eq:defictilde:b}
    \ctil{l}{2} & = \sum\limits_{i=1}^M \pips{K'+i}{ lK'+(l-1)M+i}(g-\frac{i}{M}h)^{\frac{1}{2}}(\cstarc{lK'+(l-1)M+i}) ,
\end{align}
and put 
\begin{equation}
\label{eq:defi:tildec}
   \tilde{c}_l= \ctil{l}{1} + \ctil{l}{2}.
\end{equation} 
Moreover, we set
\begin{align}
\label{eq:defi:stilde:a}
    \stil{1}&= \piph{1,2}h^{\frac{1}{2}}(s^*s),\\
\label{eq:defi:stilde:b}
    \stil{2}&= \sum\limits_{l=1}^{L}\sum\limits_{i=1}^M \pips{(l-1)M+i+1}{lK'+(l-1)M+i}(\frac{i}{M}h)^{\frac{1}{2}}(\cstarc{lK'+(l-1)M+i}),
\end{align}
and define 
\begin{equation}
\label{eq:defi:tildes}
\tilde{s}= \stil{1} + \stil{2}.
\end{equation}
These elements should be compared to the images of the generators of the building blocks $\tZ{L}{L+1}$ in the construction in \cite{JaWi}, where generators and relations are expressed in terms of c.p.c.\@ order zero maps. The alterations we made here ensure that the $\mathrm{C}^*$-diagonal in $\tZ{L}{L+1}$ we will specify in Proposition \ref{prop:tD} and its structure are preserved throughout the inductive system we construct.\\

By the universal property of $\tZ{L}{L+1}$, the elements defined in  (\ref{eq:defi:tildec}) and (\ref{eq:defi:tildes}) now implement a $^*$-homomorphism $\Phi_{L,L'}:\tZ{L}{L+1} \to  \tZ{L'}{L'+1}$. For that, the only thing to check is that  $\tilde{c}_1, \dots, \tilde{c}_{L}, \tilde{s}$ satisfy the relations (\ref{eq:Therelations}), which will be addressed in Subsection \ref{subsec:Checkingtherelations} below. As $L'$ only depends on $L$ and the parameters $M,K$ as in (\ref{eq:choiceofL'}), this gives us a whole family of $^*$-homomorphisms. Using these as connecting maps, we can then construct an inductive system of dimension drop algebras, still written as universal $\mathrm{C}^*$-algebras, as follows:\\
Take some $L_1\geq 2$ and then inductively choose sequences of positive integers $(M_n)_n, (K_n)_n$ determining $(L_n)_n$ by  
\begin{equation}
\label{eq:choiceofLn}
L_n := L'(L_{n-1}, M_{n-1}, K_{n-1}), \,n \geq 2,
\end{equation}
as in (\ref{eq:choiceofL'}) in such a way that
\begin{equation*}
\tZ{L_1}{L_1+1} \overset{\Phi_{L_1,L_2}}{\longrightarrow} \tZ{L_2}{L_2+1} \overset{\Phi_{L_2, L_3}}{\longrightarrow} \tZ{L_3}{L_3+1} \overset{\Phi_{L_3, L_4}}{\longrightarrow} \dots
\end{equation*}
yields a simple limit with unique tracial state, which is thus isomorphic to $\mathcal{Z}$. Essentially, both $(K_n)_n$, $(M_n)_n$ have to increase sufficiently fast to achieve that. More precisely, choosing $(K_n)_n$ with Proposition \ref{prop:tooltraces} will ensure that the inductive limit has a unique tracial state, whereas choosing $(M_n)_n$ with Proposition \ref{prop:maintoolsimple} will ensure that the inductive limit is simple. This results in the following theorem. 

\begin{theo}
\label{theo:constrindlim}
There are sequences of positive integers $(K_n)_n$, $(M_n)_n$, such that the following holds: Take $L_1\geq2$ and $L_n := L'(L_{n-1}, M_{n-1}, K_{n-1})$ for $ n \geq 2$, as in (\ref{eq:choiceofL'}). Let 
\begin{equation*}
    \Phi_{L_n,L_{n+1}}: \tZ{L_n}{L_{n+1}} \to \tZ{L_{n+1}}{L_{n+1}+1}
\end{equation*}
denote the $^*$-homomorphism induced by the elements defined as in (\ref{eq:defi:tildec}), (\ref{eq:defi:tildes}) (with $L= L_n, L'=L_{n+1}$). Further, let $\tildeZ$ denote  the  inductive limit of
\begin{equation*}
\tZ{L_1}{L_1+1} \overset{\Phi_{L_1,L_2}}{\longrightarrow} \tZ{L_2}{L_2+1} \overset{\Phi_{L_2, L_3}}{\longrightarrow} \tZ{L_3}{L_3+1} \overset{\Phi_{L_3, L_4}}{\longrightarrow} \dots
\end{equation*}
Then, $\tildeZ$ is simple and has a unique tracial state. In particular, we have $\tilde{\mathcal{Z}} \cong \mathcal{Z}$.
\end{theo}

At this point, we have an inductive limit construction which works for pairs of sequences $(K_n)_n$, $(M_n)_n$ in $\N$ increasing sufficiently fast. One possible configuration is $M_n:=n4^n$, $K_n:= 2^{n+3}$; we will discuss this in Remark \ref{remark:Explicitchoice}.  Albeit technical, the description is very explicit, requiring only the definitions given in (\ref{eq:choiceofL'}), (\ref{eq:defictilde:a}), (\ref{eq:defictilde:b}), (\ref{eq:defi:tildec}), (\ref{eq:defi:stilde:a}), (\ref{eq:defi:stilde:b}), (\ref{eq:defi:tildes}), (\ref{eq:choiceofLn}). In the next subsection, we attempt to explain the structure of the relations (\ref{eq:defi:tildec}), (\ref{eq:defi:tildes}) a little more conceptually. Afterwards we will prove the theorem; to this end, we need to check that the elements from (\ref{eq:defi:tildec}), (\ref{eq:defi:tildes}) satisfy the relations (\ref{eq:Therelations})---we do this in Subsection \ref{subsec:Checkingtherelations} below---and we need to describe how the sequences $(K_n)_n, (M_n)_n$ can be chosen to indeed obtain the Jiang--Su algebra; this will be done in Subsection \ref{subsec:simplicityanduniquetrace}.

\subsection{{\sc A visualisation}\nopunct}
\label{subsec:visual}
\quad \vspace{1ex} \\
The definition in (\ref{eq:defi:tildec}), (\ref{eq:defi:tildes}) may seem rather ad hoc, so let us explain what the underlying idea is by visualising the situation. With (\hyperref[eq:Therelations]{$\mathcal{\tilde{R}}_{L'}$}) and (\ref{eq:basicCT}),  we observe that for $f \in C_0((0,1))$ and $1 \leq i,j \leq L'$,
\begin{itemize}
    \item the elements $c_1^2, \ldots, c_{L'}^*c_{L'}$ are mutually orthogonal but not orthogonal to $s^*s$,
    \item $\sum_{l=1}^{L'}c_l^*c_l +s^*s=1$, 
    \item the elements $c_i^*sf(c_j^*c_j)s^*c_i$ are positive and mutually orthogonal for varying $j$,
    \item $c_i^*c_i$ acts as a unit on $c_i^*sf(c_j^*c_j)s^*c_i$,
    \item the elements $c_i^*c_is^*s$ are mutually orthogonal, positive,  and $c_i^*c_i +s^*s$ acts as a unit on $c_i^*c_is^*s$,
    \item $c_i^*sf(c_j^*c_j)s^*c_i$ and elements of the form  $c_i^*c_is^*s$ are orthogonal.
\end{itemize}
Thus, we may think of the elements $c_1^2, \dots, c_{L'}^*c_{L'},s^*s$ as arranged in $\tZ{L'}{L'+1}$ as follows.

\begin{figure}[h!]
\centering
\begin{tikzpicture}[scale=0.7, transform shape]
 \foreach \i in {0,1,2,3,4,5,6,7,8,9} {
        \draw[red, thick] (2,-\i/3 + 9/6) -- (5,-\i/3+9/6) ;
            }
  \foreach \i in {1,2,3,4,5,6,7,8} {
        \draw[blue,thick] (2,-\i/3 +1.5 - 0.04) -- (5,-\i/3 +1.5 - 0.04) ;
     }

\foreach \i in {0,1,2} {
\draw[blue, thick] (0.5,  -\i*0.1+ 0.1- \i/3+1/3) -- (1.5, -\i/3 + 1/3);

}
\draw[blue, thick, dotted] (1.,1) -- (1, 1.5);
\draw[blue, thick, dotted] (1,-1) -- (1, -1.5);

\draw[blue,thick] (2,-0.04 + 1.5) -- (0,3);
\draw[blue, thick]  (2, -0.04 + 1.5) -- (5, -0.04+ 1.5) ;

\foreach \i in {0,1,2,3,4,5,6,7,8,9} {
        \draw[blue,thick] (-3,4.5-\i/3) -- (0,4.5-\i/3) ;
     
        }  
    
\draw[blue,thick] (-3,4.5) -- (-3,1.5) node[left, midway]{\LARGE $c_1^2$ \hspace{1em}};
\draw[blue,thick] (0,4.5) -- (0,1.5);

\draw[blue, very thick, dotted] (-1.5, 0.75) -- (-1.5,-0.75);

\draw[blue,thick] (2, -1.5 -0.04) -- (0,-3);
\draw[blue,thick]  (2, -1.5 -0.04) -- (5, -1.5 -0.04) ;

\foreach \i in {0,1,2,3,4,5,6,7,8,9} {
        \draw[blue,thick] (-3,-1.5-\i/3) -- (0,-1.5 -\i/3) ;
     }      
    
\draw[blue,thick] (-3,-1.5) -- (-3,-4.5) node[left, midway]{\LARGE $c_{L'}^*c_{L'}$ \hspace{1em}};
\draw[blue,thick] (0,-1.5) -- (0,-4.5);

\draw[red,very thick] (5,+9/6) -- (5, -3+9/6) node[right, midway] {\LARGE \hspace{0.5em} $s^*s$};

\end{tikzpicture}
\caption{Visualisation of $ c_1^2, \ldots, c_{L'}^*c_{L'},s^*s \in \tilde{Z}_{L',L'+1}$.}

\end{figure}
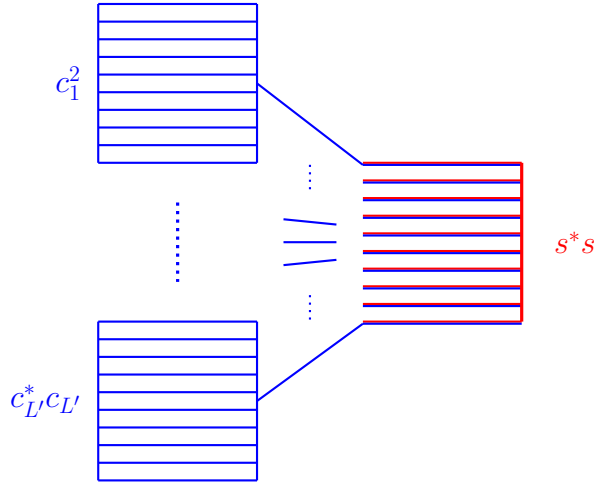

Here, each block on the left-hand side illustrates the $L'$ intervals coming from the positive, mutually orthogonal elements of the form $c_i^*s f(c_j^*c_j)s^*c_i$, glued together in the endpoints illustrated by the vertical line connecting them. Thus, we view $c_i^*c_i$ as being constant on the $i$-th block on the left. The right block indicates the $L'$ intervals coming from the elements $c_i^*c_is^*s$ that lie under $c_i^*c_i+s^*s$, glued together in the right endpoint. Here, we think of $c_i^*c_i$ as being $1$ in the left endpoint  of the $i$-th interval (identified with the right endpoint of the intervals in the $i$-th block on the left, hence connected by an auxiliary line) and going to $0$ in the right endpoint, whereas $s^*s$ is $1$ in the right endpoint and goes to $0$ in the left endpoints along the $L'$ intervals.\\
With this in mind, we can now visualise the elements $\tilde{c}_1^*\tilde{c}_1, \dots, \tilde{c}_{L}^*\tilde{c}_{L}, \tilde{s}^*\tilde{s} \in \tZ{L'}{L'+1}$ which are sums of functions with values in $[0,1]$ applied to the elements $c_1^2, \dots, c_{L'}^*c_{L'}, s^*s \in \tZ{L'}{L'+1}$. For $L=2$ and $L'=L'(2,M,K)$ for some $M,K \in \N$, the elements look as depicted in Figure \ref{fig:defielements}. \\

\begin{figure}[h]
\centering
\begin{minipage}{0.475\textwidth}
\centering
\begin{tikzpicture}[scale=0.9, transform shape]
\begin{scope}[shift={(-1.25,0)}]

\draw[decorate,decoration={brace,amplitude=10pt}] (-2.65,6.1)
-- (-2.65,8)  node[black,left, midway]{$K'$ \,\,\quad}
;

\draw[decorate,decoration={brace,amplitude=10pt}] (-2.65,0.25)
-- (-2.65,6)  node[black,left, midway]{$M$ \,\,\quad}
;
\draw[decorate,decoration={brace,amplitude=10pt}] (-2.65,-1.75)
-- (-2.65,0.1)  node[black,left, midway]{$K'$ \,\,\quad}
;
\draw[decorate,decoration={brace,amplitude=10pt}] (-2.65,-7.6)
-- (-2.65,-1.85)  node[black,left, midway]{$M$ \,\,\quad}
;

\foreach \i in {0,1,...,15} {
        \draw[Gray,thick] (1.5,-\i/15 *3 +1.5) -- (4.5,-\i/15*3 + 1.5) ;
}

\draw[Gray, thick, dotted] (0.75, 3.85) -- (0.75, 3.6);
\draw[Gray, thick, dotted] (0.75, 2.55) -- (0.75, 2.3);
\draw[Gray, thick, dotted] (0.75, 1.35) -- (0.75, 1.1);
\draw[Gray, thick, dotted] (0.75, -0.95) -- (0.75, -1.2);
\draw[Gray, thick, dotted] (0.75, -2.1) -- (0.75, -2.35);
\draw[Gray, thick, dotted] (0.75, -3.35) -- (0.75, -3.6);

\draw[Gray,thick] (1.5,1.5) -- (0,7.25 ) ;
\draw[Green!35,thick] (1.5,-3/15 *3 +1.5) -- (0,5.25) ;
\draw[Green!35,thick, dashed] (1.5,-3/15 *3 +1.5) -- (2.25,-3/15 *3 +1.5) ;

\draw[Green!45,thick, dashed] (1.5,-4/15 *3 +1.5) -- (2.25,-4/15 *3 +1.5) ;
\draw[Green!55,thick] (1.5,-5/15 *3 +1.5) -- (0,3.25);
\draw[Green!55,thick, dashed] (1.5,-5/15 *3 +1.5) -- (2.25,-5/15 *3 +1.5);

\draw[Green!80,thick, dashed] (1.5,-6/15 *3 +1.5) -- (2.25,-6/15 *3 +1.5);
\draw[Green,thick] (1.5,-7/15 *3 +1.5) -- (0,1.25);
\draw[Green,thick, dashed] (1.5,-7/15 *3 +1.5) -- (2.25,-7/15 *3 +1.5);

\draw[Gray,thick] (1.5,-8/15 *3 +1.5) -- (0,-0.75);
\draw[Green!35,thick] (1.5,-11/15 *3 +1.5) -- (0,-2.75) ;
\draw[Green!35,thick, dashed] (1.5,-11/15 *3 +1.5) -- (2.25,-11/15 *3 +1.5) ;
\draw[Green!45,thick, dashed] (1.5,-12/15 *3 +1.5) -- (2.25,-12/15 *3 +1.5) ;
\draw[Green!55,thick] (1.5,-13/15 *3 +1.5) -- (0,-4.75);
\draw[Green!55,thick, dashed] (1.5,-13/15 *3 +1.5) -- (2.25,-13/15 *3 +1.5);
\draw[Green!80,thick, dashed] (1.5,-14/15 *3 +1.5) -- (2.25,-14/15 *3 +1.5);
\draw[Green,thick] (1.5,-15/15 *3 +1.5) -- (0,-6.75);
\draw[Green,thick, dashed] (1.5,-15/15 *3 +1.5) -- (2.25,-15/15 *3 +1.5);

\foreach \x in { 0, 8}{
 \foreach \i in {0,1,2,3,...,12,13,14,15} {
        \draw[Gray, thick] (0,-\i/15 *1.5 + \x) -- (-2.5,-\i/15*1.5 + \x) ;
        
}

\draw[Gray, thick,  dotted] (-1.25, \x-1.6) -- (-1.25, \x-1.9);
\draw[Gray,  thick] (-2.5,\x) -- (-2.5,\x-1.5);
\draw[Gray,  thick] (0,\x) -- (0,\x-1.5);

}

\draw[Green,  thick] (4.5,1.5) -- (4.5, -1.5);

\foreach \i in {0,1,...,15} {
        \draw[Green,  thick, dashed] (3.75,-\i/15 *3 +1.5) -- (4.5,-\i/15*3 + 1.5) ;
    
}

\foreach \x in {2,-6}{
\foreach \i in {0,1,2,3,...,12,13,14,15} {
        \draw[Green,thick] (0,-\i/15 *1.5+ \x) -- (-2.5,-\i/15*1.5 +\x ) ;}
\draw[Green,  thick] (-2.5,\x) -- (-2.5,\x-1.5);
\draw[Green, thick] (0, \x) -- (0,\x-1.5);

}
\foreach \x in {-4,4} {
\foreach \i in {0,1,2,3,...,12,13,14,15} {
        \draw[Green!55, thick] (0,-\i/15 *1.5 +\x) -- (-2.5,-\i/15*1.5 +\x ) ;}
\draw[Green!55, thick] (-2.5,\x) -- (-2.5,\x-1.5);
\draw[Green!55, thick] (0,\x) -- (0,\x-1.5);
\draw[Gray, thick,  dotted] (-1.25, \x-1.6) -- (-1.25, \x-1.9);
}
 
\foreach \x in { -2,6}{
 \foreach \i in {0,1,2,3,...,12,13,14,15} {
        \draw[Green!35, thick] (0,-\i/15 *1.5 + \x) -- (-2.5,-\i/15*1.5 + \x) ;

}
\draw[Gray, thick,  dotted] (-1.25, \x-1.6) -- (-1.25, \x-1.9);
\draw[Green!35,  thick] (-2.5,\x) -- (-2.5,\x-1.5);
\draw[Green!35,  thick] (0,\x) -- (0,\x-1.5);
}
\end{scope}
\end{tikzpicture}
\end{minipage}
\hfill
\begin{minipage}{0.475\textwidth}
\centering
\begin{tikzpicture}[scale=0.9, transform shape]
\begin{scope}[shift={(-1.25,0)}]

\draw[Gray] (3,5.9) rectangle (4.5,8.1);
\filldraw[Magenta] (3.3,7.8) circle(2pt) node[right]{\,\,$\tilde{c}_1^2$};
\filldraw[Blue]  (3.3 , 7) circle(2pt) node[right]{\,\,$\tilde{c}_2^*\tilde{c}_2$};
\filldraw[Green] (3.3 , 6.2) circle(2pt) node[right]{\,\,$\tilde{s}^*\tilde{s}$};

\draw[Gray, thick, dotted] (0.75, 3.85) -- (0.75, 3.6);
\draw[Gray, thick, dotted] (0.75, 2.55) -- (0.75, 2.3);
\draw[Gray, thick, dotted] (0.75, 1.35) -- (0.75, 1.1);
\draw[Gray, thick, dotted] (0.75, -0.95) -- (0.75, -1.2);
\draw[Gray, thick, dotted] (0.75, -2.1) -- (0.75, -2.35);
\draw[Gray, thick, dotted] (0.75, -3.35) -- (0.75, -3.6);

\foreach \i in {0,1,...,15} {
        \draw[Gray,thick] (1.5,-\i/15 *3 +1.5) -- (4.5,-\i/15*3 + 1.5) ;
}
\draw[Gray, thick] (1.5,+1.5) -- (1.5,- 1.5);
\draw[Gray, thick] (4.5, 1.5) -- (4.5, -1.5);

\draw[Magenta,thick] (1.5,1.5) -- (0,7.25 ) ;
\draw[Magenta!75,thick] (1.5,-3/15*3+1.5) -- (0,5.25 ) ;
\draw[Magenta!45,thick] (1.5,-5/15 *3 +1.5) -- (0,3.25);
\draw[Gray,thick] (1.5,-7/15 *3 +1.5) -- (0,1.25);

\draw[Blue,thick] (1.5,-8/15 *3 +1.5) -- (0,-0.75);
\draw[Blue!45,thick] (1.5,-11/15 *3 +1.5) -- (0,-2.75) ;
\draw[Blue!25,thick] (1.5,-13/15 *3 +1.5) -- (0,-4.75);
\draw[Gray,thick] (1.5,-15/15 *3 +1.5) -- (0,-6.75);

\foreach \x in { -6, 2}{
 \foreach \i in {0,1,2,3,...,12,13,14,15} {
        \draw[Gray, thick] (0,-\i/15 *1.5 + \x) -- (-2.5,-\i/15*1.5 + \x) ;
        
}

\draw[Gray,  thick] (-2.5,\x) -- (-2.5,\x-1.5);
\draw[Gray,  thick] (0,\x) -- (0,\x-1.5);

}

\draw[Gray,  thick] (4.5,1.5) -- (4.5, -1.5);

\foreach \i in {0,1,...,7} {
        \draw[Magenta,  thick, dashed] (3.75,-\i/15 *3 +1.5) -- (4.5,-\i/15*3 + 1.5) ;
        \draw[Magenta, thick](3.75,-\i/15 *3 +1.5) -- (2.25,-\i/15*3 + 1.5);
}
\foreach \i in {0,1,2}{
        \draw[Magenta, thick](1.5,-\i/15 *3 +1.5) -- (2.25,-\i/15*3 + 1.5);
}
\draw[Magenta!90, thick, dashed](1.95, - 3/15*3+1.5) -- ( 2.25, -3/15*3 + 1.5);
\draw[Magenta!80, thick, dashed](1.85, - 3/15*3+1.5) -- ( 1.75, -3/15*3 + 1.5);
\draw[Magenta!70, thick, dashed](1.55, - 3/15*3+1.5) -- ( 1.65, -3/15*3 + 1.5);

\draw[Magenta!85, thick, dashed](1.95, - 4/15*3+1.5) -- ( 2.25, -4/15*3 + 1.5);
\draw[Magenta!70, thick, dashed](1.85, - 4/15*3+1.5) -- ( 1.75, -4/15*3 + 1.5);
\draw[Magenta!55, thick, dashed](1.55, - 4/15*3+1.5) -- ( 1.65, -4/15*3 + 1.5);

\draw[Magenta!80, thick, dashed](1.95, - 5/15*3+1.5) -- ( 2.25, -5/15*3 + 1.5);
\draw[Magenta!65, thick, dashed](1.85, - 5/15*3+1.5) -- ( 1.75, -5/15*3 + 1.5);
\draw[Magenta!50, thick, dashed](1.55, - 5/15*3+1.5) -- ( 1.65, -5/15*3 + 1.5);

\draw[Magenta!75, thick, dashed](1.95, - 6/15*3+1.5) -- ( 2.25, -6/15*3 + 1.5);
\draw[Magenta!57.5, thick, dashed](1.85, -6/15*3+1.5) -- ( 1.75, -6/15*3 + 1.5);
\draw[Magenta!45, thick, dashed](1.55, - 6/15*3+1.5) -- ( 1.65, -6/15*3 + 1.5);

\draw[Magenta!70, thick, dashed](1.95, - 7/15*3+1.5) -- ( 2.25, -7/15*3 + 1.5);
\draw[Magenta!40, thick, dashed](1.85, - 7/15*3+1.5) -- ( 1.75, -7/15*3 + 1.5);
\draw[Magenta!10, thick, dashed](1.55, - 7/15*3+1.5) -- ( 1.65, -7/15*3 + 1.5);

\foreach \i in {8,9,10} {
        \draw[Blue,  thick, dashed] (3.75,-\i/15 *3 +1.5) -- (4.5,-\i/15*3 + 1.5) ;
        \draw[Blue, thick](3.75,-\i/15 *3 +1.5) -- (1.5,-\i/15*3 + 1.5);
}

\foreach \i in {11,12,13,14,15} {
        \draw[Blue,  thick, dashed] (3.75,-\i/15 *3 +1.5) -- (4.5,-\i/15*3 + 1.5) ;
        \draw[Blue, thick](3.75,-\i/15 *3 +1.5) -- (2.25,-\i/15*3 + 1.5);
}

\draw[Blue!90, thick, dashed](1.95, -11/15*3+1.5) -- ( 2.25, -11/15*3 + 1.5);
\draw[Blue!80, thick, dashed](1.85, - 11/15*3+1.5) -- ( 1.75, -11/15*3 + 1.5);
\draw[Blue!70, thick, dashed](1.55, - 11/15*3+1.5) -- ( 1.65, -11/15*3 + 1.5);

\draw[Blue!85, thick, dashed](1.95, - 12/15*3+1.5) -- ( 2.25, -12/15*3 + 1.5);
\draw[Blue!70, thick, dashed](1.85, - 12/15*3+1.5) -- ( 1.75, -12/15*3 + 1.5);
\draw[Blue!55, thick, dashed](1.55, - 12/15*3+1.5) -- ( 1.65, -12/15*3 + 1.5);

\draw[Blue!80, thick, dashed](1.95, - 13/15*3+1.5) -- ( 2.25, -13/15*3 + 1.5);
\draw[Blue!65, thick, dashed](1.85, - 13/15*3+1.5) -- ( 1.75, -13/15*3 + 1.5);
\draw[Blue!50, thick, dashed](1.55, - 13/15*3+1.5) -- ( 1.65, -13/15*3 + 1.5);

\draw[Blue!75, thick, dashed](1.95, -14/15*3+1.5) -- ( 2.25, -14/15*3 + 1.5);
\draw[Blue!57.5, thick, dashed](1.85, -14/15*3+1.5) -- ( 1.75, -14/15*3 + 1.5);
\draw[Blue!45, thick, dashed](1.55, - 14/15*3+1.5) -- ( 1.65, -14/15*3 + 1.5);

\draw[Blue!70, thick, dashed](1.95, - 15/15*3+1.5) -- ( 2.25, -15/15*3 + 1.5);
\draw[Blue!40, thick, dashed](1.85, - 15/15*3+1.5) -- ( 1.75, -15/15*3 + 1.5);
\draw[Blue!10, thick, dashed](1.55, - 15/15*3+1.5) -- ( 1.65, -15/15*3 + 1.5);

\foreach \x in {8}{
\foreach \i in {0,1,2,3,...,12,13,14,15} {
        \draw[Magenta,thick] (0,-\i/15 *1.5+ \x) -- (-2.5,-\i/15*1.5 +\x ) ;}
\draw[Magenta,  thick] (-2.5,\x) -- (-2.5,\x-1.5);
\draw[Magenta, thick] (0, \x) -- (0,\x-1.5);
\draw[Gray, thick,  dotted] (-1.25, \x-1.6) -- (-1.25, \x-1.9);
}
\foreach \x in {0}{
\foreach \i in {0,1,2,3,...,12,13,14,15} {
        \draw[Blue,thick] (0,-\i/15 *1.5+ \x) -- (-2.5,-\i/15*1.5 +\x ) ;}
\draw[Blue,  thick] (-2.5,\x) -- (-2.5,\x-1.5);
\draw[Blue, thick] (0, \x) -- (0,\x-1.5);
\draw[Gray, thick,  dotted] (-1.25, \x-1.6) -- (-1.25, \x-1.9);
}
\foreach \x in {4} {
\foreach \i in {0,1,2,3,...,12,13,14,15} {
        \draw[Magenta!45, thick] (0,-\i/15 *1.5 +\x) -- (-2.5,-\i/15*1.5 +\x ) ;}
\draw[Magenta!45, thick] (-2.5,\x) -- (-2.5,\x-1.5);
\draw[Magenta!45, thick] (0,\x) -- (0,\x-1.5);
\draw[Gray, thick,  dotted] (-1.25, \x-1.6) -- (-1.25, \x-1.9);

}
\foreach \x in {-4} {
\foreach \i in {0,1,2,3,...,12,13,14,15} {
        \draw[Blue!45, thick] (0,-\i/15 *1.5 +\x) -- (-2.5,-\i/15*1.5 +\x ) ;}
\draw[Blue!35, thick] (-2.5,\x) -- (-2.5,\x-1.5);
\draw[Blue!35, thick] (0,\x) -- (0,\x-1.5);
\draw[Blue, thick,  dotted] (-1.25, \x-1.6) -- (-1.25, \x-1.9);

}
 
\foreach \x in { 6}{
 \foreach \i in {0,1,2,3,...,12,13,14,15} {
        \draw[Magenta!75, thick] (0,-\i/15 *1.5 + \x) -- (-2.5,-\i/15*1.5 + \x) ;

}
\draw[Gray, thick,  dotted] (-1.25, \x-1.6) -- (-1.25, \x-1.9);

\draw[Magenta!75,  thick] (-2.5,\x) -- (-2.5,\x-1.5);
\draw[Magenta!75,  thick] (0,\x) -- (0,\x-1.5);
}
\foreach \x in { -2}{
 \foreach \i in {0,1,2,3,...,12,13,14,15} {
        \draw[Blue!75, thick] (0,-\i/15 *1.5 + \x) -- (-2.5,-\i/15*1.5 + \x) ;

}
\draw[Gray, thick,  dotted] (-1.25, \x-1.6) -- (-1.25, \x-1.9);
\draw[Blue!65,  thick] (-2.5,\x) -- (-2.5,\x-1.5);
\draw[Blue!65,  thick] (0,\x) -- (0,\x-1.5);
}
\end{scope}
\end{tikzpicture}
\end{minipage}
\caption{Illustration of the elements $\tilde{c}_1^2, \tilde{c}_2^*\tilde{c}_2, \tilde{s}^*\tilde{s}\in \tilde{Z}_{L',L'+1}$ for $L=2$ and $L'=L'(2,M,K)$ with $M,K\geq2$.\label{fig:defielements}}
\end{figure}
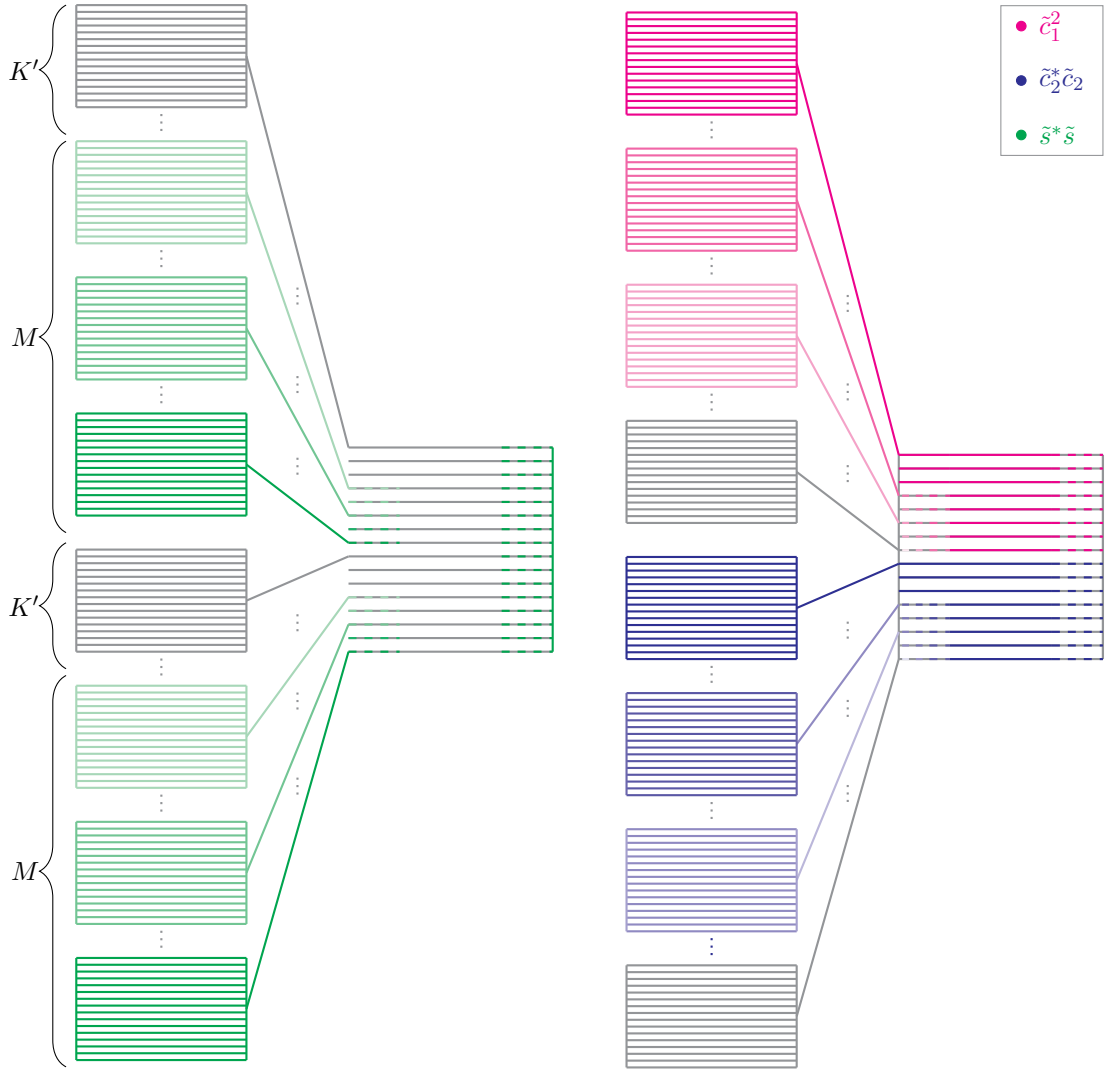

Here, the dashed lines indicate that the respective applied function is linearly increasing or decreasing, the non-dashed parts indicate that the applied function is constant, and the shade of the colour indicates the value of the function in $[0,1]$. The picture shows that the element $\tilde{s}^*\tilde{s}$ is $0$ on the first $K'$ blocks, then has non-zero, constant, growing values on the following $M$
blocks, being $1$ on the last one. The behaviour repeats itself for the next $K'+M$ blocks. On the other hand, the element $\tilde{c}_1^2$ is constant $1$ on the first $K'$ blocks and attains non-zero, constant, decreasing values on the following $M$ blocks, being $0$ on the last one. $\tilde{c}_2^*\tilde{c}_2$ looks the same, but on the second $K'+M$ blocks.\\
In the depiction, we already see that the elements $\tilde{c}_1^2$ and $ \tilde{c}_2^*\tilde{c}_2$ are mutually orthogonal and that $\tilde{c}_1^2 + \tilde{c}_2^*\tilde{c}_2 + \tilde{s}^*\tilde{s}=1$. Heuristically speaking, $\tilde{c}_2$ is chosen such that it transports the support of $\tilde{c}_2^*\tilde{c}_2$  to the support of $\tilde{c}_1^2$, thus ensuring $\tilde{c}_2\tilde{c}_2^*= \tilde{c}_1^2$. Similarly, $\tilde{s}$ is defined to transport the element $\tilde{s}^*\tilde{s}$ under the element $\tilde{c}_1^2$, so we get $\tilde{c}_1\tilde{s}=\tilde{s}$. Hence, the illustration already indicates that the relations (\ref{eq:Therelations}) will be satisfied.

\subsection{{\sc Checking the relations}\nopunct}
\label{subsec:Checkingtherelations}
\quad \vspace{1ex} \\
Let us now check that the elements defined in (\ref{eq:defi:tildec}), (\ref{eq:defi:tildes}) satisfy the relations (\ref{eq:Therelations}). 
\begin{prop}
    \label{theo:starhom}
Let  $L,M,K\geq 2$  determining $K',L'$ as in (\ref{eq:choiceofL'}). Then,  $\tilde{c}_1, \ldots, \tilde{c}_L, \tilde{s}$ defined in (\ref{eq:defi:tildec}), (\ref{eq:defi:tildes}) satisfy the relations (\ref{eq:Therelations}). Hence, they implement a $^*$-homomorphism
\begin{equation*}
\Phi_{L,L'} : \tZ{L}{L+1} \to \tZ{L'}{L'+1}
\end{equation*}
satisfying
\begin{equation*}
\Phi_{L,L'}(\bar{c}_l) = \tilde{c}_l, \;\; \Phi_{L,L'}(\bar{s}) = \tilde{s},\;\; \text{ for } 1 \leq l \leq L.
\end{equation*}
\end{prop}
 Note that the induced $^*$-homomorphism will automatically be unital and injective as prime dimension drop algebras do not contain non-trivial projections. Before we prove the proposition, let us observe some auxiliary facts. First, note that 
\begin{gather}
f(c_j^*c_j)\perp f'(c_k^*c_k),\label{eq:polynomialtrick1} \\
 f(1-s^*s) =f \big(\sum\limits_{i=1}^{L'} c_i^*c_i \big) = \sum\limits_{i=1}^{L'} f(c_i^*c_i),\label{eq:polynomialtrick2}\\
f(c_k^*c_k) a(1-s^*s) = f(c^*_kc_k)a(c_k^*c_k) =  a(1-s^*s)  f(c_k^*c_k), \label{eq:polynomialtrick3}
\end{gather}
for $f, f' \in C_0((0,1])$, $a \in C([0,1])$, and $1 \leq j \neq k \leq L'$. This is straightforward to check for polynomials and hence holds for arbitrary continuous functions. Moreover, observe the following general facts on c.p.c.\@ order zero maps.

\begin{lemma}
\label{lemma:cpctricks}
Let $A$ be a $\mathrm{C}^*$-algebra, $n \in \N$ and $\sigma: \mathbb{M}_n \to A$ a c.p.c.\@ order zero map with supporting $^*$-homomorphism denoted by $\pi_\sigma$. Then, we have for $f \in C_0((0,1])$:
\begin{align}
&\pi_\sigma(e_{i,i}) f(\sigma(e_{k,k})) = \delta_{ik}  f(\sigma(e_{k,k})) = f(\sigma(e_{k,k})) \pi_\sigma(e_{i,i}), \label{eq:lemma:cpctricks1}\\[3pt]
&\pi_\sigma(e_{i,j}) f(\sigma(e_{k,k})) = \delta_{jk} f(\sigma(e_{i,i}))  \pi_\sigma(e_{i,j})
\label{eq:lemma:cpctricks2}.
\end{align}
\end{lemma}
\begin{proof}
    It suffices to check both identities for polynomials, which is straightforward using $\sigma(1) \pi_\sigma(e_{i,j}) = \sigma(e_{i,j})$ and the fact that $\sigma $ preserves orthogonality. 
\end{proof}
Now, we can check (\ref{eq:Therelations}) for the elements $\tilde{c}_1, \ldots, \tilde{c}_L, \tilde{s} \in \tZ{L'}{L'+1}$.

\begin{proof}[Proof of Proposition \ref{theo:starhom}]
First, observe that 
\begin{equation*}
    \tilde{c}_1^2 \overset{(\ref{eq:lemma:cpctricks1})}{=} \sum\limits_{i=1}^{K'} g(\cstarc{i}) + \sum\limits_{i=1}^{M} (g-\frac{i}{M}h)(\cstarc{K'+i})\geq 0.
\end{equation*}
Moreover, we compute for $ 1 \leq l \leq L$ that
\begin{align*}
    \tilde{c}_l^*\tilde{c}_l  = (\ctil{l}{1})^*\ctil{l}{1}+ (\ctil{l}{2})^*\ctil{l}{2} 
    \overset{(\ref{eq:lemma:cpctricks1})}&{=} \sum\limits_{i=1}^{K'} g(\cstarc{(l-1)(K'+M)+i})\\
     &  \hspace{0.5cm} + \sum\limits_{i=1}^{M} (g-\frac{i}{M}h)(\cstarc{lK'+(l-1)M+i}).
\end{align*}
In particular, this implies $\tilde{c}_i^* \tilde{c}_i \perp \tilde{c}_j^*\tilde{c}_j$ for $1 \leq i\neq j \leq L$ with (\ref{eq:polynomialtrick1}). Analogously, we obtain
\begin{equation*}
    \tilde{c}_l\tilde{c}_l^* \overset{(\ref{eq:polynomialtrick1})}{=} \ctil{l}{1} (\ctil{l}{1})^* + \ctil{l}{2} ( \ctil{l}{2})^* \overset{(\ref{eq:lemma:cpctricks1}),(\ref{eq:lemma:cpctricks2})}{=} \tilde{c}_1^2.
\end{equation*}
To check $\sum_{l=1}^{L} \tilde{c}_l^*\tilde{c}_l + \tilde{s}^* \tilde{s} =1 $, we note that 
\begin{equation*}
c_1\stil{1} \overset{(\ref{eq:lemma:cpctricks2})}{=} c_1 h^{\frac{1}{2}}(ss^*) \piph{1,2} =  h^{\frac{1}{2}}(ss^*) \piph{1,2} \overset{(\ref{eq:lemma:cpctricks2})}{=} \stil{1}
\end{equation*}
since $c_1s=s$ by the relations of $\tZ{L'}{L'+1}$.
With that, we compute
\begin{align*}
    \tilde{s}^* \tilde{s} 
    \overset{(\ref{eq:polynomialtrick1}),(\ref{eq:lemma:cpctricks2})}&{=} (\stil{1})^* \stil{1} + (\stil{2})^* \stil{2} \\
    \overset{(\ref{eq:lemma:cpctricks1})}&{=} h(s^*s) + \sum\limits_{l=1}^L \sum\limits_{i=1}^M \frac{i}{M} h (\cstarc{lK'+(l-1)M+i}).
\end{align*}
This yields
\begin{align*}
\tilde{s}^*\tilde{s} + \sum\limits_{l=1}^{L} \tilde{c}_l^*\tilde{c}_l  & = h(s^*s) + \sum\limits_{l=1}^L \sum\limits_{i=1}^M \frac{i}{M} h (\cstarc{lK'+(l-1)M+i})\\
& \hspace{0.5cm} + \sum\limits_{l=1}^{L} \sum\limits_{i=1}^{K'} g(\cstarc{(l-1)(K'+M)+i})\\
&  \hspace{0.5cm} +\sum\limits_{l=1}^{L} \sum\limits_{i=1}^{M} (g-\frac{i}{M}h)(\cstarc{lK'+(l-1)M+i}) \\
& = h(s^*s) + \sum\limits_{i=1}^{L'} g(\cstarc{i}) \\
& \overset{(\ref{eq:polynomialtrick2})}{=} h(s^*s) + g\circ(1-\mathrm{id})(s^*s)\\
&=1.
\end{align*}
It remains to show $\tilde{c}_1 \tilde{s} = \tilde{s}$. As $c_1s=s$ and $gh=h$, we get for $1 \leq l \leq L, 1 \leq i \leq M$, that
\begin{align*}
    g^\frac{1}{2}(c_1^2) \piph{1,2} h^{\frac{1}{2}}(s^*s) \overset{(\ref{eq:lemma:cpctricks2})}&{=} \piph{1,2} h^{\frac{1}{2}}(s^*s),\\
    \big(\sum\limits_{j=1}^{K'} g^\frac{1}{2}(c_j^*c_j)\big) (\frac{i}{M}h)^{\frac{1}{2}}(\cstarc{(l-1)M+i+1})\overset{(\ref{eq:polynomialtrick1})}&{=}   (\frac{i}{M}h)^{\frac{1}{2}}(\cstarc{(l-1)M+i+1}).
\end{align*}
For the second equation, we in particular used that $ K'=KML \geq LM +1$. This implies $\ctil{1}{1} \stil{1} = \stil{1}$ and $\ctil{1}{1} \stil{2} = \stil{2}$, and thus $\tilde{c}_1 \tilde{s} = \tilde{s}$ since $\ctil{1}{1} \perp \ctil{1}{2}$.
\end{proof}

\subsection{{\sc Simplicity and monotraciality} \nopunct}
\label{subsec:simplicityanduniquetrace}
\quad \vspace{1ex}
\quad \\
In the following, we prepare and conduct the construction of the inductive system of dimension drop algebras in  Theorem \ref{theo:constrindlim}, such that the inductive limit has a unique tracial state and is simple, hence is isomorphic to the Jiang--Su algebra. 
Recall that the idea is to find sequences $(M_n)_n, (K_n)_n$ of positive integers that determine $(L_n)_n$ by $L_n :=L'(L_{n-1}, M_{n-1}, K_{n-1})$, $n \geq 2$, as in (\ref{eq:choiceofL'}). Here, the parameters $(M_n)_n$ will make sure that the inductive limit will be simple, while the choice of $(K_n)_n $ will ensure the existence of a unique tracial state on the limit.\\
We start with analysing the behaviour of tracial states under the $^*$-homomorphisms from Proposition \ref{theo:starhom} in order to derive the parameters $(K_n)_n$.
\begin{lemma}
\label{lemma:starstrace}
Let $L, M,K \geq 2$ and $L':=L'(L,M,K)$ as in (\ref{eq:choiceofL'}). Then,
\begin{equation*}
\tau \circ \Phi_{L,L'}(\bs^*\bs)\leq \frac{2}{K}
\end{equation*}
holds for any tracial state $\tau \in T(\tZ{L'}{L'+1})$.
\end{lemma}
\begin{proof}
    Take a tracial state $\tau$ on $\tZ{L'}{L'+1}$. Recall $\Phi_{L,L'}(\bar{s}^*\bar{s}) = (\stil{1})^*\stil{1} + (\stil{2})^*\stil{2}$ from the proof of Proposition \ref{theo:starhom} with $\stil{1}, \stil{2}$ from (\ref{eq:defi:stilde:a}), (\ref{eq:defi:stilde:b}), and observe that 
\begin{equation}
   \label{eq:proof:lemma:starstrace:1}
       \tau((\stil{1})^*\stil{1}) = \tau (h(s^*s)) = \tau(h(ss^*)) \leq \tau(c_1^2) = \frac{1}{L'} \tau (\sum\limits_{l=1}^{L'} c_l^*c_l) \leq \frac{1}{K}
\end{equation}
by the trace property and the relations in $\tZ{L'}{L'+1}$.
Now, let us define elements $v_j$ for $1\leq j \leq K$ by 
\begin{equation*}
v_j := \sum\limits_{l=1}^{L} \sum\limits_{i=1}^M  \pips{(j-1)ML + (l-1)M + i +1}{(l-1)M + i +1}(\frac{i}{M} h)^\frac{1}{2} ( \cstarc{(l-1)M + i +1}) .
\end{equation*}
Then, we have 
\begin{align*}
v_j^ *v_j \overset{(\ref{eq:lemma:cpctricks1}), (\ref{eq:lemma:cpctricks2})}&{=} \sum\limits_{l=1}^{L} \sum\limits_{i=1}^M \frac{i}{M} h ( \cstarc{(l-1)M + i +1}) \overset{(\ref{eq:lemma:cpctricks1}), (\ref{eq:lemma:cpctricks2})}{=} \stil{2} (\stil{2})^*,\\
v_j v_j^* \overset{(\ref{eq:lemma:cpctricks1}), (\ref{eq:lemma:cpctricks2})}&{=} \sum\limits_{l=1}^{L} \sum\limits_{i=1}^M  \frac{i}{M} h ( \cstarc{(j-1)ML + (l-1)M + i +1}),
\end{align*}
for each $ 1 \leq j \leq K$. In particular, observe that
\begin{equation*}
\sum\limits_{j=1}^{K} v_j v_j^*  
=\sum\limits_{j=1}^{K} \sum\limits_{l=1}^{L} \sum\limits_{i=1}^M  \frac{i}{M} h ( \cstarc{(j-1)ML + (l-1)M + i +1})
\leq  \sum\limits_{l=1}^{L'} c_l^*c_l \leq 1.
\end{equation*}
Thus, we get
\begin{equation}
\label{eq:proof:lemma:starstrace:2}
\tau ( (\stil{2})^* \stil{2})  = \frac{1}{K} \tau ( \sum \limits_{j=1}^K v_j^*v_j) 
\leq \frac{1}{K} 
\end{equation}
using the trace property. 
Combining (\ref{eq:proof:lemma:starstrace:1}) and  (\ref{eq:proof:lemma:starstrace:2}) now yields the statement.
\end{proof}
With this, we can define a distinguished tracial state on $\tZ{L}{L+1}$ which will give rise to the unique tracial state on the inductive limit later.

\begin{prop}
\label{prop:defitau_L}
Let $L\geq 2$ and let $(M_n)_n$, $(K_n)_n$ be sequences of natural numbers with $K_n \to \infty$. For each $n$, we set $L_n:= L'(L,M_n, K_n)$ as in (\ref{eq:choiceofL'}). Further, let $(\tau_n)_n$ be any sequence of tracial states with  $\tau_n \in T(\tZ{L_n}{L_n+1})$ for $n \in \N$.  Then, we have a well-defined tracial state given by
\begin{equation*}
\tau^{(L)}: \tZ{L}{L+1} \to \C,\;\; \tau^{(L)}(x):= \lim\limits_{n \to \infty} \tau_n \circ \Phi_{L,L_n} (x)
\end{equation*}
that is independent of the choice of $(M_n)_n$, $(K_n)_n$, and $(\tau_n)_n$.
\end{prop}

\begin{proof}
    For well-definedness of $\tau^{(L)}$, we first consider $x \in \mathcal{W}( \bar{c}_2, \dots, \bar{c}_L,\bar{s})$, which denotes the set of words in the generators $\bar{c}_2, \dots, \bar{c}_L, \bar{s}$ of $\tZ{L}{L+1}$ and their adjoints. If \@ $x= y\bar{s}z$ for some $y,z \in \mathcal{W}( \bar{c}_2, \dots, \bar{c}_L,\bar{s}) \cup \{1\}$, then
\begin{equation}
\label{eq:proof:prop:defitau_L:1}
\vert \tau_n \circ \Phi_{L,L_n}(x) \vert \leq \vert \tau_n \circ \Phi_{L,L_n}(\bar{s}^*\bar{s}) \vert^\frac{1}{2}\; \vert \tau_n \circ \Phi_{L,L_n}((zy)^*zy) \vert^\frac{1}{2} 
\overset{n }{\longrightarrow} 0,
\end{equation}
by the Cauchy--Schwarz inequality, Lemma \ref{lemma:starstrace} and the trace property. Therefore,  $\tau^{(L)}(x)$ is well-defined in that case and hence also if $\bar{s}^*$ appears in $x$. Thus, it remains to consider $x \in \mathcal{W}( \bar{c}_2, \dots, \bar{c}_L)$.
With the relations in (\ref{eq:Therelations}), (\ref{eq:basicCT}) and by replacing $\bar{c}_i\bar{c}_i^*$ with $\bar{c}_1^2$ for $2 \leq i \leq L$ whenever possible, we get that $x$ is either $0$ or we have
\begin{align*}
    x\in \{\bar{c}_i, \bar{c}_i^*\bar{c}_i, \bar{c}_i^*\bar{c}_j, (\bar{c}_1^2)^k, (\bar{c}_1^2)^k\bar{c}_i, \bar{c}_i^*(\bar{c}_1^2)^k \bar{c}_j \, : \, 2 \leq i,j \leq L, \; k \in \N \}.
\end{align*}
In each of these cases, the limit $\lim_{n \to \infty} \tau_n \circ \Phi_{L,L_n}(x)$ exists by the relations in (\ref{eq:Therelations}) and (\ref{eq:basicCT}), the trace property and (\ref{eq:proof:prop:defitau_L:1}). That is, $\tau^{(L)}(x) $ is well-defined for any $x \in \mathcal{W}(\bar{c}_2, \dots,\bar{c}_L,\bar{s})$ and thus for all $x \in \mathrm{C}^*(\bar{c}_2, \dots, \bar{c}_L,\bar{s}) = \tZ{L}{L+1}$. Then $\tau^{(L)}$ clearly is
a tracial state on $\tZ{L}{L+1}$.\vspace{1pt} 
In particular, the arguments did not depend on the choice of $(K_n)_n$, $(M_n)_n$, or $(\tau_n)_n$, hence $\tau^{(L)}$ does not.
\end{proof}
This tracial state $\tau^{(L)}$ on $\tZ{L}{L+1}$ is the key ingredient to make sure that the inductive limit constructed later has a unique tracial state. Heuristically speaking, throughout the inductive system, the traces on the later stages will get closer and closer to $\tau^{(L_n)}$ when restricted to $\tZ{L_n}{L_n+1}$. We will use the following proposition to ensure this in the construction.
\begin{prop}
\label{prop:tooltraces}
Let $L\geq 2$ and $\tau^{(L)}$ be defined as in Proposition \ref{prop:defitau_L}. Further, let $x \in \tZ{L}{L+1}$ and $\varepsilon > 0$. Then, there exists  $\bar{K} \in \N$ such that for any $K\in \N $ with $K \geq \bar{K}$ and any $M \in \N$,
\begin{equation*}
\vert \tau \circ \Phi_{L,L'} (x) - \tau^{(L)}(x) \vert < \varepsilon
\end{equation*}
holds for any tracial state $\tau$ on $\tZ{L'}{L'+1}$, where $L':= L'(L,M,K)$ as in (\ref{eq:choiceofL'}). 
\end{prop}

\begin{proof}
Assume that there are $x \in \tZ{L}{L+1}$ and $\varepsilon>0$ such that for any $n \in \N$ there is $K_n \geq n$, $M_n \in \N$, and a tracial state $\tau_n$ on $\tZ{L_n}{L_n+1}$ with 
\begin{equation*}
\vert \tau_n \circ \Phi_{L,L_n} (x) - \tau^{(L)}(x) \vert \geq \varepsilon,
\end{equation*}
where $L_n:= L'(L,M_n,K_n)$ as in (\ref{eq:choiceofL'}). 
Then, for the sequences $(K_n)_n$, $(M_n)_n $, and $(\tau_n)_n$ 
\begin{equation*}
\lim\limits_{n\to \infty}\tau_n \circ \Phi_{L,L_n} (x) \neq  \tau^{(L)}(x)
\end{equation*}
holds. This contradicts the definition of $\tau^{(L)}$ from Proposition \ref{prop:defitau_L}, which was independent of the choices made.
\end{proof}

Now, let us derive a tool to choose the sequence $(M_n)_n$ that ensures that the constructed inductive limit is simple.
The idea is to arrange the inductive system in such a way that, for each stage, there is a dense subset of elements that will generate the whole building block as an ideal at some later stage in the inductive system.\\
The main conceptual obstacle is the following: Although the description of dimension drop algebras as entangled matrix cones we work with is useful for defining $^*$-homomorphisms or handling tracial states, the ideal structure is not as easily accessible as in the standard picture. In order to get a better understanding of the ideals, let us introduce some auxiliary c.p.c.\@ order zero maps. For now, take $L \geq 2$ and recall that $\bar{c}_1, \dots, \bar{c}_L,\bar{s}$ denote the generators of $\tZ{L}{L+1}$. Then, observe that the elements $(\bar{s}^*\bar{s})^\tinyfrac{1}{2}, \bar{s}^*\bar{c}_1, \dots, \bar{s}^*\bar{c}_L$ satisfy the relations of $C_0((0,1], \M_{L+1})$ from (\ref{eq:relationscone}). Thus, we get a $^*$-homomorphism
\begin{equation*}
 C_0((0,1], \M_{L+1}) \to \tZ{L}{L+1},\;\; \mathrm{id}^\frac{1}{2} \otimes e_{0,i} \mapsto \begin{cases} (\bar{s}^*\bar{s})^\frac{1}{2} & \text{ if } i = 0, \\
\bar{s}^*\bar{c}_i & \text{ if } i \geq 1, \end{cases}
\end{equation*}
where we denote the matrix units of $ \M_{L+1}$ by $\{ e_{i,j} \hspace{0.1em} : \hspace{0.1em} 0 \leq i,j \leq L\}$.
This induces a c.p.c.\@ order zero map
$\hat{\sigma}_{L+1}:\M_{L+1} \to \tZ{L}{L+1}$ with supporting $^*$-homomorphism $\pi_{\hat{\sigma}_{L+1}}$.
Similarly, we get a $^*$-homomorphism 
\begin{equation*}
C_0((0,1], \M_{L}) \to \tZ{L}{L+1},\;\; \mathrm{id}^\frac{1}{2} \otimes e_{1,i} \mapsto (1-\bar{s}\bar{s}^*)^\frac{1}{2} \bar{c}_i
\end{equation*}
inducing a c.p.c.\@ order zero map $\hat{\sigma}_L:\M_{L} \to \tZ{L}{L+1}$ with supporting $^*$-homomorphism $\pi_{\hat{\sigma}_{L}}$. In the same way, we define yet another c.p.c.\@ order zero map  $\hat{\theta}: \mathbb{M}_{L} \otimes \mathbb{M}_{L+1} \to \tZ{L}{L+1}$ induced by
\begin{equation*}
\{ \bar{c}_i \hat{\sigma}_{L+1}(1_{L+1})^\frac{1}{2 } \pi_{\hat{\sigma}_{L+1}}(e_{0,j})\, : \; 1 \leq i \leq L, \, 0 \leq j \leq L\}
\end{equation*}
which satisfies the relations  of $C_0((0,1], \M_L \otimes \M_{L+1})$ using $\mathbb{M}_L \otimes \mathbb{M}_{L+1} \cong \mathbb{M}_{L(L+1)}$. Now, take piecewise linear functions  $a,b \in C_0((0,1])$ defined as in Figure \ref{fig:fg}.
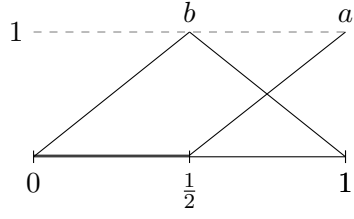
\begin{figure}[h!]
    \centering
	\begin{tikzpicture}[scale=0.825]
	\draw[gray, dashed] (1,2) -- (6,2) node[black, left, at start] {$1$};
    \draw[black] (1,0) -- (6,0) ;
    \draw[black]  (1,0.1) -- (1,-0.1) node[below]{$0$};
    \draw[black]  (6,0.1) -- (6,-0.1) node[below]{$1$};
    \draw[black]  (6,0.1) -- (6,-0.1) node[below]{$1$};
     \draw[black]  (3.5,0.1) -- (3.5,-0.1) node[below]{$\frac{1}{2}$};

\draw[black] (1,0) -- (3.5,2) node[above, at end]{$b$};
  \draw[black] (6,0) -- (3.5,2); 
  \draw[black] (6,2) -- (3.5,0.025) node[above, at start]{$a$}; 
  \draw[black] (3.5,0.025) -- (1,0.025);

      \end{tikzpicture}
	\caption{ \label{fig:fg} Graphs of $a,b \in C_0((0,1])$.}
	
\end{figure}

We put 
$\sigma_L := a(\hat{\sigma}_L)$ and $\sigma_{L+1} := a(\hat{\sigma}_{L+1})$ using functional calculus for order zero maps from \cite{WiZa}.
Moreover, let $b' \in C_0((0,1])$ with $b' \circ (\mathrm{id} - \mathrm{id}^2)=b$. Then, $\theta:= b'(\hat{\theta})$ satisfies
\begin{align*}
&\theta( e_{i,i} \otimes e_{j,j} ) = \begin{cases} b(\bar{c}_i^*\bar{c}_i)& \text{ if } j =0,\\
\pi_{\hat{\sigma}_{L+1}} (e_{j,0}) b(\bar{c}_i^*\bar{c}_i) \pi_{\hat{\sigma}_{L+1}} (e_{0,j}) & \text{ if } j \geq 1.
\end{cases} 
\end{align*}

With this setup, we now identify some prominent ideals in $\tZ{L}{L+1}$, in particular, an essential ideal. Throughout, for a subset $\{x_\gamma : \gamma \in \Gamma\} $ in some $\mathrm{C}^*$-algebra $A$, we write $\Id{x_\gamma : \gamma \in \Gamma }$ for the closed two-sided ideal in $A$ generated by the subset.

\begin{prop}
    \label{prop:essideal}
Define ideals  in $\tZ{L}{L+1}$ by 
\begin{align*}
&I_1:= \Id{ \sigma_{L}(\M_{L}), \theta(\M_L \otimes \M_{L+1})} , \\
&I_2 :=\Id{ \sigma_{L+1}(\M_{L+1}), \theta(\M_{L} \otimes \M_{L+1})}, \\
&J:= \Id{\theta(\M_L \otimes \M_{L+1})}. 
\end{align*}
Then, the following hold:
\begin{itemize}
    \item[\upshape{(i)}] We have
\begin{equation}
\label{eq:prop:essideal}
J  \cong C_0((0,1)) \otimes \M_L \otimes \M_{L+1}, \;\; \tZ{L}{L+1} /I_1 \cong \M_{L+1},\; \; \tZ{L}{L+1} / I_2 \cong \M_{L}.
\end{equation}
\item[\upshape{(ii)}] $J \trianglelefteq \tZ{L}{L+1}$ is essential, i.e.\@ $xJ =0$ implies $x=0$ for any positive $x \in \tZ{L}{L+1}$.
\item[\upshape{(iii)}] \vspace{2pt} We have $J = \Id{q_i: 1 \leq i\leq L }$ where 
$q_i:= \bar{c}_i^*\bar{c}_i \bar{s}^*\bar{s}$  for $1 \leq i \leq L.$
\end{itemize}
\end{prop}

\begin{proof}
(i): First, observe that the following hold:
\begin{spacing}{1.1}
\begin{enumerate}
\item[\upshape{(a)}] $ \sigma_L(1), \sigma_{L+1}(1), \theta(1) \in Z(\tZ{L}{L+1})$, where $Z(\tZ{L}{L+1})$ denotes the centre, i.e.\@ all elements in $\tZ{L}{L+1}$ which commute with any other element, 
\item[\upshape{(b)}] $  \sigma_L(1) + \sigma_{L+1}(1) +  \theta(1) = 1$, 
\item[\upshape{(c)}] $ \sigma_{L} \perp \sigma_{L+1}$, 
\item[\upshape{(d)}] $\mathrm{C}^*( \sigma_L(\M_L), \sigma_{L+1}(\M_{L+1}), \theta(\M_L \otimes \M_{L+1})) = \tZ{L}{L+1}$.
\end{enumerate}
\end{spacing}
To check these, let us compute
\begin{nalign}
\label{eqproofprop:essideal0}
\sigma_L(1)&= a( \sum\limits_{i=1}^L \bc_i^*(1-\bs\bs^*)\bc_i) = a\circ(1 - \mathrm{id})(\hat{\sigma}_{L+1}(1)),\\
\sigma_{L+1}(1) & =a( \hat{\sigma}_{L+1}(1)) = a(\bs^*\bs) +\sum\limits_{i=1}^La(\bc_i^*\bs\bs^*\bc_i),\\
\theta(1) &= b(\bs^*\bs) + \sum\limits_{j=1}^{L} b(\bc_j^*\bs\bs^*\bc_j)
 = b(\hat{\sigma}_{L+1}(1)).
\end{nalign}
Now, for (a) it suffices to show that $\hat{\sigma}_{L+1}(1)$ lies in the centre, which is straightforward. Statements (b) and (c) are immediate from the definition of $a,b$. To see (d), we first note that $ \mathrm{C}^*(\hat{\sigma}_{L}(\mathbb{M}_L), \hat{\sigma}_{L+1}(\mathbb{M}_{L+1})) = \tZ{L}{L+1}$ since each generator of $\tZ{L}{L+1}$ is contained  in the left hand side which is elementary using (\ref{eq:Therelations}). Thus, it suffices to show that
\begin{equation*}
    \hat{\sigma}_L(\M_L), \hat{\sigma}_{L+1}(\M_{L+1}) \subset \mathrm{C}^*( \sigma_L(\M_L), \sigma_{L+1}(\M_{L+1}), \theta(\M_L \otimes \M_{L+1}))=: A.
\end{equation*}
Indeed, we have that $\hat{\theta}(\M_L \otimes \M_{L+1}) \subset A$ since $\hat{\theta}(1)= f(\theta(1))$ for $f \in C_0((0,1])$ with $ f \circ b =\mathrm{id}- \mathrm{id}^2 $. Thus $ \hat{\theta}(1) \pi_{\hat{\sigma}_{L+1}}(e_{i,j}) \in A, $ for $1 \leq i \leq L,$ $ 0 \leq j \leq L$ as
\begin{equation*}
 \hat{\theta}(1) \pi_{\hat{\sigma}_{L+1}}(e_{i,j}) = \sum\limits_{l=1}^L \pi_{\hat{\sigma}_{L+1}} (e_{i,0}) \bc_l^*\hat{\sigma}_{L+1}(1) \bc_l   \pi_{\hat{\sigma}_{L+1}}  (e_{0,j}) \in \hat{\theta}(\mathbb{M}_L \otimes \mathbb{M}_{L+1}) \subset A. 
\end{equation*}
As $\theta(1)$ is a function in $\hat{\theta}(1)$, this implies
\begin{equation*}
    b(\hat{\sigma}_{L+1}(1))  \pi_{\hat{\sigma}_{L+1}}(e_{i,j}) =\theta(1) \pi_{\hat{\sigma}_{L+1}}(e_{i,j})  \in A.
\end{equation*}
\begin{spacing}{1,1}
Since $a(\hat{\sigma}_{L+1}(1)) \pi_{\hat{\sigma}_{L+1}}(e_{i,j}) \in A $ and $\mathrm{C}^*(a,b) = C_0((0,1])$, we get $\hat{\sigma}_{L+1}(1) \pi_{\hat{\sigma}_{L+1}}(e_{i,j}) \in A$ for any $i,j$, hence $\hat{\sigma}_{L+1}(\M_{L+1}) \subset A.$ The proof of $\hat{\sigma}_L(\mathbb{M}_L) \subset A$ goes analogously. 
    
\end{spacing}
Now, to show the first identification in (\ref{eq:prop:essideal}), we observe that 
\begin{equation}
\label{eq:proof:prop:essideal:1}
J = \mathrm{C}^* (\theta(\M_L \otimes \M_{L+1}), \theta(1) \sigma_{L}(1), \theta(1 ) \sigma_{L+1}(1) )=:B . 
\end{equation} 
Indeed, we have $\bc_i \theta(1), \bs \theta(1)\in B$ due to basic computations. Thus, $B$ is a closed ideal as $\theta(1)\in Z(\tZ{L}{L+1})$, which suffices to show $J=B$.
Next, let us consider the $^*$-homomorphism
\begin{nalign}
\label{eq:definalpha}
\alpha: C_0((0,1)) \otimes \M_L \otimes \M_{L+1} &\to \mathrm{C}^* (\theta(\M_L \otimes \M_{L+1}),\hspace{0.1em} \theta(1) \sigma_{L}(1),\hspace{0.1em} \theta(1 ) \sigma_{L+1}(1) ),\\
f \otimes x \otimes y & \mapsto f(\hat{\sigma}_{L+1}(1)) \pi_{\theta}(x \otimes y),
\end{nalign}
where $\pi_\theta$ is the supporting $^*$-homomorphism of $\theta$.
Since $C_0((0,1)) = \mathrm{C}^*(b, ab, a \circ (1 - \mathrm{id}) b)$, $\alpha$ is fully determined by the assignments
\begin{align*}
a \circ (1 - \mathrm{id})b \otimes x \otimes y & \mapsto \sigma_{L}(1) \theta(1) \pi_{\theta}(x \otimes y)= \sigma_{L}(1) \hspace{0.1em}\theta(x \otimes y), \\
ab \otimes x \otimes y & \mapsto \sigma_{L+1}(1) \theta(1) \pi_{\theta}(x \otimes y) =  \sigma_{L+1}(1)\hspace{0.1em} \theta(x \otimes y), \\
b \otimes x \otimes y &\mapsto \theta(x \otimes y).
\end{align*}
This also shows that $\alpha$ is surjective. For injectivity, assume that $f \otimes x\otimes y \in \ker (\alpha)$ with $f\neq0$. Then, $f \otimes 1 \otimes e_{0,0} \in \ker (\alpha) $ as $\ker (\alpha)$ is a closed ideal and we have
\begin{equation*}
0= f(\hat{\sigma}_{L+1}(1)) \theta(1 \otimes e_{0,0}) = f(\hat{\sigma}_{L+1}(e_{0,0})) \theta(1 \otimes e_{0,0})= fb(\bs^*\bs),
\end{equation*}
hence the spectrum of $\bs^*\bs$ has a gap. This contradicts the fact that $\tZ{L}{L+1}$ contains no non-trivial projection.

In order to show $\tZ{L}{L+1} /I_1 \cong \M_{L+1}$, consider the c.p.c.\@ order zero map
\begin{equation*}
Q_{I_1} \circ \sigma_{L+1} : \M_{L+1} \to \tZ{L}{L+1} / I_1,
\end{equation*}
 where $Q_{I_1}$ denotes the quotient map. This map is non-trivial as $\sigma_{L+1}(1) \notin I_1$. Indeed, $\sigma_{L+1}(1) \in I_1$ would imply $\sigma_{L+1}(1) \in J $ as $\sigma_L \perp \sigma_{L+1}$ and $\sigma_{L+1}(1) \in Z(\tZ{L}{L+1})$.  Let $(m_n)_n$ be an approximate unit of $C_0((0,1])$. Then, $(m_n(\theta(1)))_n$ is an approximate unit of $J$, but $ m_n(\theta(1)) \sigma_{L+1}(1) \nrightarrow \sigma_{L+1}(1)$ by (\ref{eqproofprop:essideal0}).\\ 
 Moreover,  $Q_{I_1} \circ \sigma_{L+1}$ is a $^*$-homomorphism since it maps the unit to a projection by (b). By simplicity of $\M_{L+1}$  and (d), the map is even a $^*$-isomorphism. The third isomorphism in (\ref{eq:prop:essideal}) can be defined analogously with $Q_{I_1} \circ \sigma_{L+1}$ replaced with $Q_{I_2} \circ \sigma_{L}$.\\

(ii): Let $Q_J: \tZ{L}{L+1} \to \tZ{L}{L+1} / J$ denote the quotient map. Then, $Q_J \circ \sigma_{L+1}$ is a $^*$-homomorphism as the unit is sent to a projection by (b) and (c). The map is non-trivial since $\sigma_{L+1}(1) \notin J$, as observed above, and hence injective. Similarly, one checks that $Q_J \circ \sigma_L$ is an injective $^*$-homomorphism. With these observations and (c), we now get a $^*$-homomorphism
\begin{equation*}
      \mathbb{M}_L \oplus \mathbb{M}_{L+1} \to \tZ{L}{L+1}/J,\;\; (x,y) \mapsto Q_J \circ \sigma_L(x) + Q_J\circ \sigma_{L+1} (y)
\end{equation*}
which is injective as $Q_J \circ \sigma_L$ and $ Q_J \circ \sigma_{L+1}$ are and by (c). It is also surjective by (c) and (d),
hence $\tZ{L}{L+1}/ J \cong \mathbb{M}_L \oplus \mathbb{M}_{L+1}$. Composing $Q_{J}$ with the projection onto $\mathbb{M}_L$, respectively $\mathbb{M}_{L+1}$, yields $^*$-homomorphisms
\begin{align*}
\pi_0: \tZ{L}{L+1} \twoheadrightarrow \M_L \oplus \M_{L+1} \twoheadrightarrow \M_L,\;\;\;
\pi_1: \tZ{L}{L+1} \twoheadrightarrow \M_L \oplus \M_{L+1}\twoheadrightarrow \M_{L+1}.
\end{align*}
Now, let us assume that there is $x \geq 0$ in $\tZ{L}{L+1}$ with $xJ =0$.
Without loss of generality, we may suppose that $\pi_0(x)$ is a projection and $\pi_1(x) = 0 $. Take $(u_n)_n$ an approximate unit of $J$ and note that $x^2-x = (x^2-x)(1-u_n)$ for any $n \in \N$. We get
\begin{equation*}
\norm{x^2-x} = \lim\limits_{n \to \infty} \norm{(x^2-x)(1-u_n)} =  \norm{Q_J(x^2-x)} = \norm{\pi_0(x^2-x)} =0,
\end{equation*}
i.e.\@ $x \in \tZ{L}{L+1}$ is a projection, which implies $x=0$.\\

(iii): To show $J \subset \Id{q_i:  1 \leq i\leq L}=: I$, it suffices to check that $ \theta(1) \in I$. Recall $\theta(1)$ from (\ref{eqproofprop:essideal0}). Since  $\bc_i^*\bc_i-(\bc_i^*\bc_i)^2 \in I$ for any $1 \leq i \leq L$, we have that $b(\bc_i^*\bc_i) \in I$ and thus $b(\bs^*\bs) \in I$ by (\ref{eq:polynomialtrick2}). Further, we have $\bc_j^*\bs(\bc_i^*\bc_i-(\bc_i^*\bc_i)^2)\bs^*\bc_j \in I$ for $1 \leq i,j \leq L$ which implies $\bc_j^*\bs f(\bc_i^*\bc_i)\bs^*\bc_j \in I$
for any $f \in C_0((0,1))$. Hence, 
\begin{equation*}
\bc_j^*\bs f(\bs^*\bs)\bs^*\bc_j= \bc_j^*\bs\bs^*f(\bs\bs^*)\bc_j \in I
\end{equation*}
and thus $\bc_j^*f(\bs\bs^*)\bc_j \in I$ for any such $f$ and in particular for $f=b$. For the reverse inclusion, note that $b(\bc_i^*\bc_i) \in J$ for $1 \leq i \leq L$ yields $q_i = \bc_i^*\bc_i -(\bc_i^*\bc_i)^2  \in J$. 
\end{proof}
From the statements above, we will mainly need that $J$ is an essential ideal for constructing the inductive limit. The other observations will play a crucial role when we describe a diagonal in $\tZ{L}{L+1}$ later. The following proposition is the last ingredient required for the construction.
\begin{prop}
\label{prop:maintoolsimple}
Let $L \geq 2$ and $x \in \tZ{L}{L+1}$ be positive and non-zero. Then, there exists $M\in \N$ such that $\Id{\Phi_{L,L'}(x)}= \tZ{L'}{L'+1}$, where $L':= L'(L,M,K)$ as in (\ref{eq:choiceofL'}) for $K \geq 2$ arbitrary.
\end{prop}
\begin{proof}
First, let us observe the following.\\
\textbf{Claim}: There exists $0 \neq f \in C_0((0,1))$ such that $f(\bs^*\bs) \in \Id{x}$.\\
    Indeed, note that $q_iyq_i \in \overline{\mathrm{C}^*(\bs^*\bs) \mathrm{C}^*(\bc_i^*\bc_i) }$ for $1 \leq i \leq L$ for any $y \in \tZ{L}{L+1}$, where $q_i$ is defined as in Proposition \ref{prop:essideal}. To see that, it suffices to show $q_iyq_i \in \mathrm{C}^*(\bs^*\bs) \mathrm{C}^*(\bc_i^*\bc_i) $ in the case $y \in \mathcal{W}(\bc_2, \dots, \bc_L, \bs)$, the set of words in $\bc_2, \ldots, \bc_L, \bs$ and their adjoints (cf.\@ proof of \ref{prop:defitau_L}). By the relations (\ref{eq:Therelations}) and (\ref{eq:basicCT}), we either get $q_iyq_i=0$ or
    \begin{equation*}
    q_iyq_i = (\bs^*\bs)^n q_i y' q_i = (\bs^*\bs)^{n+2} (\bc_i^* \bc_i)^m \in \mathrm{C}^*(\bs^*\bs)\mathrm{C}^*(\bc_i^*\bc_i)
    \end{equation*}
    for some $n\geq0, m\geq 2$ and $y'\in \mathcal{W}(\bc_2, \ldots, \bc_L)$. We thus have for any $y \in \tZ{L}{L+1}$ that
    \begin{equation}
    \label{eq:proof:lemma:simpleintermediategoal}
    q_iyq_i \in \overline{\mathrm{C}^*(\bs^*\bs) \mathrm{C}^*(\bc_i^*\bc_i) } \overset{(\ref{eq:polynomialtrick3})}{=} \{ f(\bc_i^*\bc_i) : f \in C_0((0,1))\}.
    \end{equation}
    
Now, since the ideal $J = \Id{q_i : 1 \leq i \leq L }\trianglelefteq \tZ{L}{L+1}$ from Proposition \ref{prop:essideal} is essential and $x \neq 0$ by assertion, we have $\{ 0 \} \neq J \Id{x}$. Therefore, there is a positive, non-zero $y \in \Id{x}$ and $1 \leq i \leq L$ such that $q_i y q_i\neq 0$. Hence, there exists  $f \in C_0((0,1))$ with $f \neq 0$ such that $f(\bc_i^*\bc_i)= q_i y q_i \in \Id{x}$ by (\ref{eq:proof:lemma:simpleintermediategoal}). This implies $f(\bc_l^*\bc_l)\in \Id{x}$ for all $1 \leq l \leq L$, and thus $f\circ ( 1 -\mathrm{id)}(\bs^*\bs)\in \Id{x}$  by (\ref{eq:polynomialtrick2}). This concludes the proof of the claim.\\
Now, let us show the statement of the proposition. Let $0 \neq f\in C_0((0,1))$ such that $f(\bs^*\bs) \in \Id{x}$, which exists by the claim. Since $f \neq 0$, there is $M \in \N$ and $1 \leq i \leq M$ such that $f(\frac{i}{M})\neq 0$. We may assume that $f(\frac{i}{M}) =1$, otherwise scale $f $ appropriately. Take $K',L'$ depending on $L,K,M$ as in (\ref{eq:choiceofL'}) and recall from the proof of Proposition \ref{theo:starhom} that $\Phi_{L,L'}(\bar{s}^*\bar{s}) = (\stil{1})^*\stil{1} + (\stil{2})^* \stil{2}$ with $\stil{1}, \stil{2}$ defined as in (\ref{eq:defi:stilde:a}), (\ref{eq:defi:stilde:b}) in terms of the generators $c_1, \ldots, c_L, s \in \tZ{L'}{L'+1}$. We observe
    \begin{equation*}
        \Phi_{L,L'}(f(\bar{s}^*\bar{s}))= f(\Phi_{L,L'}(\bar{s}^*\bar{s})) \geq f((\stil{2})^*\stil{2})) \geq f \circ (\frac{i}{M} h)(\cstarc{K'+i})
    \end{equation*}
    and put $f'=f \circ (\frac{i}{M} h)$. In particular, we have $f' \in C_0((0,1])$ with $f'(1)=1$. We set 
    \begin{equation*}
    I:= \Id{f'(\cstarc{K'+i})} \subset \Id{\Phi_{L,L'}(f(\bs^*\bs))}\subset \Id{\Phi_{L,L'}(x)}
    \end{equation*}
    and note that $f'(c_l^*c_l) \in I$
    for any $ 1 \leq l \leq L'$ by (\ref{eq:lemma:cpctricks1}),(\ref{eq:lemma:cpctricks2}). This implies
    \begin{equation*}
         f'\circ (1-\mathrm{id})(s^*s) \overset{(\ref{eq:polynomialtrick2})}{=} \sum \limits_{l=1}^{L'} f'(c_l^*c_l) \in I.
    \end{equation*}
    Now, we have for any $F \in C_0((0,1])$ that $F(ss^*)=F(ss^*)f'(c_1^2) \in I$ and hence $F(s^*s) \in I$ by Lemma \ref{lemma:cpctricks}. Taking $F= 1-  f'\circ (1- \mathrm{id})$ thus yields $1 \in I$. We conclude that $ \tZ{L'}{L'+1}= I \subset \Id{\Phi_{L,L'}(x)}$.
    \end{proof}
We now finalise our construction of the Jiang--Su algebra proving Theorem \ref{theo:constrindlim}. In the following, for  $a\in A_{\mathrm{sa}}$, we write $a_+$ for the positive part of $a$.
\begin{proof}[Proof of Theorem \ref{theo:constrindlim}]
    Let $(\varepsilon_n)_n$ be a summable, decreasing sequence of positive reals.
    Put $L_1:=2$ and fix an increasing sequence of finite subsets   $ F_1^{(1)} \subset F_1^{(2)} \subset \dots \subset (\tZ{L_1}{L_1+1})_+^1$  with dense union.\\ 
\textbf{Claim 1}: There are sequences of positive integers $(K_n)_n$, $(M_n)_n$, such that the following holds: Put $L_n:= L'(L_{n-1}, M_{n-1}, K_{n-1})$, $n \geq 2$, as in (\ref{eq:choiceofL'}), i.e.\@ we have $^*$-homomorphisms $\Phi_{L_n,L_{n+1}}:\tZ{L_n}{L_n+1} \to \tZ{L_{n+1}}{L_{n+1}+1}$ as in Proposition \ref{theo:starhom} determined by the elements in (\ref{eq:defi:tildec}), (\ref{eq:defi:tildes}). Then, we have for each $n \in \N$ that
\begin{enumerate}
\item[a)] \label{enum:proof:theo:constrindlim:a}
there is a sequence of finite subsets $ F_{n+1}^{(1)} \vspace{2pt} \subset F_{n+1}^{(2)} \subset F_{n+1}^{(3)}\subset \dots \subset (\tZ{L_{n+1}}{L_{n+1}+1})_+^1$ with dense union, such that $\Phi_{L_{n},L_{n+1}}(F_{n}^{(i+1)}) \subset F_{n+1}^{(i)}$ for $i \in \N$,
\item[b)] for each  $x \in F_{n}^{(1)}$ and any tracial state $\tau$ on $\tZ{L_{n+1}}{L_{n+1}+1}$, we have
\begin{equation*}
\vert \tau \circ \Phi_{L_{n},L_{n+1}} (x) - \tau^{(L_{n})}(x) \vert < \varepsilon_{n},
\end{equation*} where $\tau^{(L_n)}$ is defined as in Proposition \ref{prop:defitau_L}, \label{enum:proof:theo:constrindlim:b} \vspace{1pt}
\item[c)] 
 $\Id{\Phi_{L_{n},L_{n+1}}((x - \frac{\norm{x}}{2}\cdot 1)_+)} =\tZ{L_{n+1}}{L_{n+1}+1}$ for any non-zero $x \in F_{n}^{(1)}$.  \label{enum:proof:theo:constrindlim:c}
\end{enumerate}
We prove the claim by choosing $(K_n)_n,(M_n)_n$ inductively, starting with $K_1$ and $M_1$. 
Apply Proposition \ref{prop:tooltraces} to each element of the finite set $F_{1}^{(1)}$ to find $K_1$ such that 
\begin{equation*}
\vert \tau \circ \Phi_{L_1,L'} (x) - \tau^{(L_1)}(x) \vert < \varepsilon_1
\end{equation*}
holds for any $M \in \N$, each $x \in F_{1}^{(1)}$, and any  $\tau \in T(\tZ{L'}{L'+1})$, where $L'$ is chosen as in (\ref{eq:choiceofL'}). Then, apply Proposition \ref{prop:maintoolsimple} for the finite set of positive, non-zero elements
\begin{equation*}
\{ (x - \frac{\norm{x}}{2}\cdot 1)_+ : x \in F_{1}^{(1)}, x\neq 0 \} \subset (\tZ{L_1}{L_1+1})_+^1
\end{equation*}
to find $M_1\in \N$ such that $\Phi_{L_1,L_2}((x - \frac{\norm{x}}{2}\cdot 1)_+)$ generates $\tZ{L_2}{L_2+1}$ as an ideal for any non-zero $x \in F_{1}^{(1)}$, where $L_2:= L'(L_1, M_1,K_1)$. In particular, \hyperref[enum:proof:theo:constrindlim:b]{b)} and \hyperref[enum:proof:theo:constrindlim:c]{c)} are satisfied for $n=1$. Moreover, there is a sequence $ \tilde{F}_{2}^{(1)}\vspace{2pt} \subset \tilde{F}_{2}^{(2)} \subset \tilde{F}_{2}^{(3)}\subset \dots \subset (\tZ{L_{2}}{L_{2}+1})_+^1$  of finite subsets with dense union. We define $F_{2}^{(i)}:= \tilde{F}_{2}^{(i)} \cup  \Phi_{L_{1}, L_{2}}( F_{1}^{(i+1)})$, $i \in \N$. Note that $(F_{2}^{(i)})_i$ is still an increasing sequence of finite sets with dense union in $ (\tZ{L_{2}}{L_{2}+1})_+^1$. Thus,  \hyperref[enum:proof:theo:constrindlim:a]{a)} is satisfied for $n=1$.\\
Now, assume that there is $n \in \N$ such that $K_1,\dots, K_{n-1}$ and $M_1,\dots, M_{n-1}$, and thereby $L_1, \dots, L_n$, are already defined and \hyperref[enum:proof:theo:constrindlim:a]{a)}, \hyperref[enum:proof:theo:constrindlim:b]{b)}, and \hyperref[enum:proof:theo:constrindlim:c]{c)} hold for all $1 \leq m \leq n-1$. 
We apply Proposition \ref{prop:tooltraces} 
to get $K_n$ such that
\begin{equation*}
\vert \tau \circ \Phi_{L_n,L'} (x) - \tau^{(L_n)}(x) \vert < \varepsilon_n
\end{equation*}
holds for any $M \in \N$, each $x \in F_{n}^{(1)}$, and any tracial state $\tau$ on $\tZ{L'}{L'+1}$, where $L'$ is chosen as $L'(L_n, M, K_n)$ as in (\ref{eq:choiceofL'}).
As before, we now apply Proposition \ref{prop:maintoolsimple} to find $M_n \in \N$ such that $\Phi_{L_{n},L_{n+1}}((x - \frac{\norm{x}}{2}\cdot 1)_+)$ generates $\tZ{L_{n+1}}{L_{n+1}+1}$ as an ideal for any non-zero $x \in F_{n}^{(1)}$, where $L_{n+1}:= L'(L_n, M_n,K_n)$. Hence,  \hyperref[enum:proof:theo:constrindlim:b]{b)} and \hyperref[enum:proof:theo:constrindlim:c]{c)} hold. \vspace{1pt} For  \hyperref[enum:proof:theo:constrindlim:a]{a)}, take a sequence  $ \tilde{F}_{n+1}^{(1)}\vspace{2pt} \subset \tilde{F}_{n+1}^{(2)} \subset \tilde{F}_{n+1}^{(3)}\subset \dots \subset (\tZ{L_{n+1}}{L_{n+1}+1})_+^1$ of finite subsets with dense union and put $F_{n+1}^{(i)}:= \tilde{F}_{n+1}^{(i)} \cup \vspace{2pt} \Phi_{L_{n}, L_{n+1}}( F_{n}^{(i+1)})$, $i \in \N$. These are still finite and nested with dense union, and hence \hyperref[enum:proof:theo:constrindlim:a]{a)} is satisfied for $n$. This concludes the proof of Claim 1.\\
Now, we define $\tildeZ := \varinjlim( \tZ{L_n}{L_n+1}, \Phi_{L_n,L_{n+1}})$ and show that $\tildeZ$ is simple and admits a unique tracial state using \hyperref[enum:proof:theo:constrindlim:a]{a)}, \hyperref[enum:proof:theo:constrindlim:b]{b)}, and \hyperref[enum:proof:theo:constrindlim:c]{c)}. Let $\Phi_{n,\infty}:\tZ{L_n}{L_n+1} \to \tildeZ $, $n \in \N,$ denote the $^*$-homomorphism induced by the inductive limit structure and write
\begin{equation*}
\Phi_{n,m} := \Phi_{L_{m-1},L_m} \circ \dots \circ \Phi_{L_{n+1},L_{n+2}} \circ \Phi_{L_n, L_{n+1}} : \tZ{L_n}{L_{n}+1} \to \tZ{L_m}{L_m+1}
\end{equation*}
for $m \geq n+1$ to shorten notation. Note that $\Phi_{n,\infty} $ and $\Phi_{n,m}$ are injective and unital.\\ 
\textbf{Claim 2}: $\tildeZ$ is simple.\\
Let $x \in \tildeZ_+^1$ be non-zero. Then, there exists $n \in \N$ and $y' \in (\tZ{L_n}{L_n+1})_+^1$ non-zero with 
\begin{equation*}
\norm{x - \Phi_{n, \infty} (y') } \leq \frac{\norm{x}}{6}.
\end{equation*}
By \hyperref[enum:proof:theo:constrindlim:a]{a)}, there is $m \in \N$ and $y \in F_{n}^{(m)}$ non-zero such that $\norm{y - y'} \leq \frac{\norm{x}}{6}$. We get that
\begin{equation}
\label{eqtheo:constrindlimproof3}
\norm{x - \Phi_{n, \infty}(y) } \leq \frac{\norm{x}}{3},
\end{equation}
which implies $\norm{x} - \norm{y}  \leq \frac{\norm{x}}{3}$, and hence $\norm{x} \leq \frac{3}{2}\norm{y}$.  Inserting this in (\ref{eqtheo:constrindlimproof3}) yields
\begin{equation*}
\norm{x - \Phi_{n, \infty} (y) } \leq \frac{\norm{y}}{2}.
\end{equation*}
Thus, \vspace{1pt} $(\Phi_{n, \infty} (y)- \frac{\norm{y}}{2} \cdot 1)_+ \in \Id{x}$ by Lemma 2.2 in \cite{RKir}. Note that  $\Phi_{n,k}(y) \in F_{k}^{(1)}$ for any $k \geq n+m$ by construction and consequently
\begin{equation*}
\Phi_{k,k+1} ( (\Phi_{n,k}(y) - \frac{\norm{\Phi_{n,k}(y)}}{2} \cdot 1)_+) = (\Phi_{n,k+1}(y) - \frac{\norm{y}}{2} \cdot 1)_+
\end{equation*} 
generates $\tZ{L_{k+1}}{L_{k+1}+1}$ as an ideal by \hyperref[enum:proof:theo:constrindlim:c]{c)}. Therefore, we have that
\begin{align*}
\bigcup_{i \in \N} \Phi_{i, \infty}(\tZ{L_i}{ L_{i}+1}) 
 \subset \hspace{-0.75em} \bigcup_{ k \geq n+m+1} \hspace{-1em} \Phi_{k, \infty}(\Id{(\Phi_{n,k}(y) - \frac{\norm{y}}{2} \cdot 1)_+} ) 
\subset \Id{(\Phi_{n, \infty}(y)- \frac{\norm{y}}{2} \cdot 1)_+} \subset \Id{x}.
\end{align*}
This implies $ \Id{x} = \tildeZ $ since it contains a dense subset of $\tildeZ$.\\ 
\textbf{Claim 3}:\vspace{0.5pt} There is a unique tracial state on $\tildeZ$.\\
For the existence, take $n  \in \N$ and $x \in  (\tZ{L_n}{L_n+1})_+^1$. Then, the sequence  $(\tau^{(L_i)} \circ \Phi_{n,i}(x))_{i\hspace{0.1em}>\hspace{0.1em} n} \vspace{1pt}$ is Cauchy, where the tracial states $\tau^{(L_i)}$ are defined as in Proposition \ref{prop:defitau_L}. Indeed, take $\varepsilon>0$ and $N \in \N$ such that $\sum_{k \geq N} \varepsilon_k < \frac{\varepsilon}{3}$. Note that there are $m \in \N$ and $y \in F_{n}^{(m)}$ such that $\norm{x-y} < \frac{\varepsilon}{3}$ by  \hyperref[enum:proof:theo:constrindlim:a]{a)}.
As in the proof of Claim 2, we have $\Phi_{n,k}(y) \in F_{k}^{(1)}$ for any $k \geq n+m$. Thus, we get for $j > i\geq \max \{ n+m, N \}$ that
\begin{align*}
 \norm{\tau^{(L_j)} \circ \Phi_{n,j}(x) - \tau^{(L_i)} \circ \Phi_{n,i}(x)} 
& \leq  \frac{2\varepsilon}{3} + \sum\limits_{k=i}^{j-1}\norm{\tau^{(L_{k+1})} \circ \Phi_{n,k+1}(y) - \tau^{(L_k)} \circ \Phi_{n,k}(y)}\\
\overset{\hyperref[enum:proof:theo:constrindlim:b]{b)}}&{\leq} \frac{2\varepsilon}{3} + \sum\limits_{k=i}^{j-1} \varepsilon_k < \varepsilon,
\end{align*}
hence  $(\tau^{(L_i)} \circ \Phi_{n,i}(x))_{i\hspace{0.1em}> \hspace{0.1em}n}$ is Cauchy, thus convergent. With that, define 
\begin{equation*}
\tau: \bigcup\limits_{n \in \N} \Phi_{n,\infty} ( \tZ{L_n}{L_n+1}) \to \C, \;\; \tau (\Phi_{n,\infty}(x)) := \lim\limits_{i \to \infty} \tau^{(L_i)} \circ \Phi_{n,i}(x),
\end{equation*}
which is well-defined by injectivity of $\Phi_{n, \infty}$. 
Further, $\tau$ is clearly linear, positive, bounded and satisfies  $\tau(1_{\tildeZ}) =1$ and the trace property. Thus, $\tau $ extends uniquely to a tracial state on $\tildeZ$, also denoted by $\tau$.\\ 
It remains to prove that $\tau$ is unique. So let $\tau'$ be an arbitrary tracial state on $\tildeZ$. Take  $x \in \tZ{L_n}{L_n+1}$ for some $n \in \N$ and let $\varepsilon > 0$. We may assume that $x \in F_{n}^{(m)}$ for some $m \in \N$ by  \hyperref[enum:proof:theo:constrindlim:a]{a)}. Then, choose $i_0 \geq n$ such that $\varepsilon_{i} < \varepsilon$ and $\Phi_{n,i}(x) \in F_{i}^{(1)}$ for any $i \geq i_0$,  again using  \hyperref[enum:proof:theo:constrindlim:a]{a)}. With that, we compute for $i \geq i_0$ that
\begin{align*}
\vert \tau' \circ \Phi_{n,\infty}(x) - \tau^{(L_i)}\circ \Phi_{n,i}(x) \vert = \vert  (\tau' \circ \Phi_{i+1,\infty}) \circ \Phi_{i,i+1} (\Phi_{n,i}(x)) - \tau^{(L_i)}(\Phi_{n,i}(x)) \vert \overset{\hyperref[enum:proof:theo:constrindlim:b]{b)}}{<} \varepsilon.
\end{align*}
Thus, $\tau'$ and $\tau$ coincide on $\bigcup_{n \in \N} \Phi_{n,\infty} ( \tZ{L_n}{L_n+1})$ and hence on all of $\tildeZ$. 
\end{proof}
 
\begin{remark}
\label{remark:Explicitchoice}
In the construction above, we can even choose the parameters $(M_n)_n, (K_n)_n$ and thereby $(L_n)_n$ explicitly. In the proof of Proposition \ref{prop:maintoolsimple}, it suffices to ensure that the support of $f(\tilde{s}^*\tilde{s})$ in the interval induced by $s^*s$ contains $\frac{i}{M}$ for some $1 \leq i \leq M$. By the choice of $\tilde{s}^{(1)}$, this support appears shrunk by factor $4$ after each application of the connecting $^*$-homomorphism. Hence, we require $(M_n)_n$ to increase faster than $n \mapsto 4^n$, e.g.\@ we may take $M_n:=n4^n$. The sequence $(K_n)_n$ is required to have a summable sequence of inverses. The reason is the following: Using the ideal structure determined in Proposition \ref{prop:essideal}, each trace on $\tZ{L}{L+1}$ is given by a convex combination of two traces factorising through $\mathbb{M}_L, \mathbb{M}_{L+1}$, respectively, and a trace supported on $J$. The trace $\tau^{(L)}$ is exactly the trace factorising through $\mathbb{M}_L$. By the same proof as for Lemma \ref{lemma:starstrace}, we get that $\norm{\tau\circ \Phi_{L,L'}\vert_{I_2}}\leq \frac{4}{K}$, where $\tau \in T(\tZ{L'}{L'+1})$ arbitrary and $L'$ is chosen as in (\ref{eq:choiceofL'}). Hence, up to error $\frac{4}{K}$, the trace $\tau \circ \Phi_{L,L'}$ factorises through $\mathbb{M}_L$. Therefore, we have $\tau\circ \Phi_{L,L'} \approx_{\frac{8}{K}} \tau^{(L)}$. Following that argument,
if the sequence $(\varepsilon_n)_n \in \ell^1(\N)_+$ is given by $\varepsilon_n=2^{-n}$, we may choose $K_n:= 2^{n+3}$, for instance. With $L_1:=2$ and these choices, one checks that $L_n\leq 2^{3^n}$ for all $n \geq 4$, which was the size of the dimension drop algebras in the construction in \cite{JaWi}.
\end{remark}

\section{Characterising normalisers}
\label{sec:charnorm}
In the following, we establish a new characterisation of normalisers in terms of state excision, as introduced by Akemann, Anderson, and Pedersen in \cite{AAP}. For a state $\rho \in S(B)$ on the separable $\mathrm{C}^*$-algebra $B$, a sequence $(k_n)_n$ of norm one elements in $B_+$ is said to excise $\rho$ if we have $\norm{k_nbk_n - \rho(b)k_n^2} \to 0$ for any $b \in B$. If $\rho$ is pure, such a sequence exists and can be chosen such that $\rho(k_n)=1$ for all $n$. 

For this section, we consider a separable Cartan pair $D \subset A$ and write $\rho_x$ for the pure state on $D$ given by evaluation in $x \in \hat{D}$, where $\hat{D}$ is the spectrum of $D$. We will characterise the normalisers $N_A(D)$ using the following normaliser excision property (NEP). 
\begin{defi}
\label{defi:propertyNEP}
Let $(D\subset A)$ be a separable Cartan pair. We say that $v \in A$ satisfies the NEP if the following is satisfied: \\ For each $x \in \hat{D}$ there is at most one $y \in \hat{D}$ such that there exist sequences $(k_n)_n$ in $D$ excising $\rho_x \in PS(D)$ and $(l_n)_n$ in $D$ excising $\rho_y  \in PS(D)$ with 
    \begin{equation}
        \norm{l_nvk_n} \nrightarrow 0,
    \end{equation}
    i.e.\@ the sequence $(\norm{l_nvk_n})_n$ in $ [0, \infty)$ does not converge to $0$.
\end{defi}

Heuristically, if $(k_n)_n, (l_n)_n$ are excising sequences for some $\rho_x, \rho_y \in PS(D)$, respectively, and $v \in A$ satisfies $l_nvk_n \nrightarrow 0$, we think of $v$ as moving $x$ to $y$ in $\hat{D}$. That is, $v$ satisfies the NEP if $v$  moves any point $ x \in \hat{D}$ to at most one point $y$ in this sense.
This turns out to characterise normalisers of $(D \subset A)$ in the following way.

\begin{theo}
\label{theocharnorm}
Let $(D\subset A)$ be a separable Cartan pair. Then, 
$v \in A$ is a normaliser of $D$ if and only if  $v$ and $v^*$ satisfy the NEP.
\end{theo}

The rest of this section is dedicated to proving this theorem. First, recall that we have $v^*v,vv^* \in D$  for $v \in N_A(D)$ and we write 
\begin{equation*}
\mathrm{dom}(v)= \{x \in \hat{D} \, : \, v^*v(x)>0\}, \;\;\; \mathrm{ran}(v)= \{x \in \hat{D} \, : \, vv^*(x)>0\}.
\end{equation*}
By Lemma 4.6 in \cite{Ren1}, there is a unique homeomorphism $\alpha_v: \mathrm{dom}(v) \to \mathrm{ran}(v)$ such that for all $d \in D$ and $x \in \mathrm{dom}(v)$
    \begin{equation}
    \label{eq:homeoalphav}
    v^*dv(x) =d(\alpha_v(x)) v^*v(x).
    \end{equation}
We have $\alpha_{v^*} = \alpha_v^{-1}$, and $\alpha_v = \mathrm{id}_{\mathrm{dom}(v)}$ if $v \in D$.
Next, let us characterise pure state excision in the abelian subalgebra $D$. The following actually holds for an arbitrary abelian $\mathrm{C}^*$-algebra, regardless of being a Cartan subalgebra.

\begin{lemma}
\label{lemma:charexcseq}
     Let $x_0 \in \hat{D}$ and $(k_n)_n$ be a sequence of positive elements of norm $1$. Then, $(k_n)_n$ excises $\rho_{x_0}$ if and only if for every open neighbourhood $U$ of $x_0$, we have 
    \begin{equation*}
        \sup\limits_{x \in \hat{D}\backslash U} \vert k_n(x) \vert \overset{n }{\longrightarrow} 0.
    \end{equation*} 
\end{lemma}

\begin{proof}
 First, suppose that $(k_n)_n$ excises $\rho_{x_0}$. Assume there are $\varepsilon > 0$ and $U$ an open neighbourhood of  $x_0$ such that there exists a subsequence $(k_{n_i})_{i \in \N}$ and  $y_i \in \hat{D}\backslash U$, $i \in \N$, with $\vert k_{n_i} (y_i) \vert \geq \varepsilon$. Take $K \subset U$ a compact neighbourhood of $x_0$.
 Then, by Urysohn's lemma for locally compact spaces (cf.\@ 2.12 in \cite{Rud}), there is $d \in D$ with $d\equiv 1$ on $K$ and $0$ on $\hat{D}\backslash U$. Consequently, we have
   \begin{equation*}
        0 \overset{i}{\longleftarrow} \norm{k_{n_i}d k_{n_i} - \rho_{x_0}(d) k_{n_i}^2} \geq \sup\limits_{x \in \hat{D}\backslash U} \vert 
        k_{n_i}^2(x) \vert \geq  \vert k_{n_i}^2(y_i) \vert \geq \varepsilon^2,
    \end{equation*}
    a contradiction. For the reverse implication, note that 
    $\norm{k_nd} \to 0$ for all $d \in D$ with $\rho_{x_0}(d)=0$ by the assumption. That is, $(1-k_n)_n$ is an approximate unit of $\ker(\rho_{x_0})$. The claim then follows by Proposition 2.2 in \cite{AAP}.
\end{proof}
In particular, if sequences $(k_n)_n, (m_n)_n$ excise $\rho_x\in PS(D)$ for some $x \in \hat{D}$ with $m_n(x) \to 1$, there is a subsequence $(k_n')_n$ of $(k_n)_n$ still excising $\rho_x$, with
\begin{equation*}
    \lim\limits_{n \to \infty} k_n'(1-m_n) =0.
\end{equation*}
Moreover, the previous lemma has a useful consequence if we restrict ourselves to diagonal pairs.

\begin{cor}
\label{corexcisingextensions}
    Let $(D \subset A)$ be a diagonal pair, $\rho_{x_0} \in PS(D)$ for some $x_0 \in \hat{D}$ and $\rho\in PS(A)$ its unique extension. Further, let $(k_n)_n$ be a sequence in $D$ excising $\rho_{x_0}$. Then,  $(k_n)_n$ excises $\rho$. 
\end{cor}

\begin{proof}
    Let $\varepsilon >0$, $a \in A$, and denote by $E: A \twoheadrightarrow D$  the faithful conditional expectation. We have $\rho(a) = \rho_{x_0}(E(a))$ as observed in \cite{Exel}. Moreover, by the equivalent definition of $\mathrm{C}^*$-diagonals in \cite{Kum}, there are $v_1, \ldots, v_N \in N_A(D)$ with $v_i^2=0$ such that 
    \begin{equation*}
        \norm{a-E(a) - \sum\limits_{i=1}^Nv_i} < \varepsilon.
    \end{equation*}
    In particular, we have $\mathrm{dom}(v_i) \cap \mathrm{ran}(v_i)=\emptyset$ for each $1 \leq i \leq N$. This implies
    \begin{equation*}
        \norm{k_nv_ik_n}^2 \leq  \norm{k_nv_i^*k_n^2v_i} \overset{(\ref{eq:homeoalphav})}{=}  \sup\limits_{x \in \mathrm{dom}(v_i)} \vert k_n(x) k_n^2(\alpha_{v_i}(x)) v_i^*v_i(x) \vert \overset{n}{\longrightarrow} 0,
    \end{equation*}
    using Lemma \ref{lemma:charexcseq}. We conclude that
    \begin{equation*}
        \limsup\limits_{n \to \infty} \norm{k_nak_n- \rho(a)k_n^2}\leq \varepsilon+\limsup\limits_{n \to \infty} \norm{k_nE(a)k_n- \rho_{x_0}(E(a))k_n^2}= \varepsilon.
    \end{equation*}
    As $\varepsilon>0$ was arbitrary, this finishes the proof.
\end{proof}
A further key observation is that by conjugating an excising sequence with a suitable normaliser, we get a sequence excising a pure state at some other point in the spectrum. 

\begin{lemma}
\label{lemmaexcisingalphavx0}
Let $(D\subset A)$ be a separable Cartan pair and let $v \in N_A(D)$. Further, let $x_0 \in \mathrm{dom}(v) \subset \hat{D}$ and $(k_n)_n$ be a sequence excising $\rho_{x_0} \in PS (D)$ with $k_n(x_0)= 1$ for all $n$. Then, the sequence $(\lambda_n vk_n^2v^*)_n$ excises $\rho_{\alpha_v(x_0)}$ and satisfies $\lambda_n vk_n^2v^*(\alpha_v(x_0)) \to 1$, where $\lambda_n= \norm{vk_n^2v^*}^{-1}$.
\end{lemma}

\begin{proof}
    Let $d \in D$ and write $y_0:= \alpha_v(x_0)$. Then, we have 
    \begin{equation}
    \label{eq:proof:lemma:excsingalphavx0}
        \lim\limits_{n \to \infty} \lambda_n^{-1}= \lim\limits_{n \to \infty} \sup_{x \in \mathrm{ran}(v)} \vert k_n^2(\alpha_{v^*}(x)) vv^*(x) \vert = vv^*(y_0)
    \end{equation}
    by (\ref{eq:homeoalphav}) and Lemma \ref{lemma:charexcseq}. We thus get
    that $(l_n)_n:=(\lambda_n vk_n^2v^*)_n$ satisfies
  \begin{equation*}
  l_n(y_0) \overset{(\ref{eq:homeoalphav})}{=}\norm{vk_n^2v^*}^{-1} k_n^2(x_0) vv^*(y_0)\overset{n}{\longrightarrow} 1.
   \end{equation*}
   Moreover, we have
    \begin{equation*}
    \rho_{y_0}(d)= d(y_0) \overset{(\ref{eq:homeoalphav})}{=} \frac{1}{v^*v(x_0)} (v^*dv)(x_0) = \frac{1}{v^*v(x_0)} \rho_{x_0}(v^*dv).
    \end{equation*}
    With that, we compute
    \begin{equation*}
       \limsup\limits_{n \to \infty} \norm{l_n d l_n - \rho_{y_0}(d) l_n^2} 
      \leq \norm{v}^2   \limsup\limits_{n \to \infty} \vert\lambda_n \vert^2\norm{k_nv^* d v k_n - \rho_{x_0}(v^*dv) k_n^2 } 
      = 0,
    \end{equation*}
    as $(k_n)_n$ excises $\rho_{x_0}$. Thereby, we showed that  $(l_n)_n$ excises $\rho_{y_0}.$
\end{proof}

We need one more technical lemma before we can start proving Theorem \ref{theocharnorm}.
\begin{lemma}
\label{lemma:existenceef}
Let $(D\subset A)$ be a separable Cartan pair and $s \in A_+^1$ such that $\mathrm{C}^* (D,s)$ is not abelian. Then, we have:
\begin{itemize}
    \item[ \upshape (i)] There are $e,f \in D_+^1$ with $e\perp f$ and $esf \neq 0$.
    \item[\upshape (ii)] There are $x_0 \neq y_0 \in \hat{D}$ with sequences $(k_n)_n, (l_n)_n$ excising $\rho_{x_0}, \rho_{y_0} \in PS(D)$, respectively, such that $$l_n s k_n \nrightarrow 0.$$
\end{itemize}
\end{lemma}

\begin{proof}
(i) We may assume $A:=\mathrm{C}^*(D,s)$. Since $A$ is non-abelian, there is $\rho \in PS(A)$ which is not a character. Let $(\pi_\rho, H_\rho)$ be the associated irreducible GNS representation. 
Now, there are $e',f' \in D_+^1$ with $e' \perp f'$ and $\norm{\pi_\rho(e')} = \norm{\pi_\rho(f')} =1$ (cf.\@ the proof of 1.8 in \cite{LLW}). Take $\xi' \in \overline{\pi_\rho(e')H_\rho}$ and $\eta' \in \overline{\pi_\rho(f')H_\rho}$ with $ \norm{\xi'} = \norm{\eta'} = 1$. Then, Kadison's transitivity theorem yields $z \in A$ with $\pi_\rho(z)\xi'= \eta'$. Let $p \in B(H_\rho)$ be given as the limit
\begin{equation*}
    p :=\text{s.o.\@-} \lim\limits_{n \to \infty} h_n(\pi_\rho(e'))
\end{equation*}
in the strong operator topology. Here, $h_n \in C_0((0,1])$, $n \in \N$, are  taken to satisfy
\begin{equation*}
    h_n(t) = \begin{cases}
        1 & \text{ if }t \in [\frac{1}{n+1},1], \\
        0 & \text { if } t \in [0, \frac{1}{n+2}],
    \end{cases}
\end{equation*}
and to be linear on $[\frac{1}{n+2}, \frac{1}{n+1}].$ One checks that $p$ is the projection onto $\overline{\pi_\rho(e')H_\rho}$. We observe that  $ p \pi_\rho(s)(1-p) \neq 0$.
Indeed, $ p \pi_\rho(s)(1-p) = 0$ implies 
\begin{equation*}
    p \pi_\rho(s) = p \pi_\rho(s) p = ( p \pi_\rho(s) p )^* = \pi_\rho(s) p,
\end{equation*}
i.e.\@ $p \in \pi_\rho(A)'.$ 
With that, we have
\begin{equation*}
    \eta' = \pi_\rho(z) \xi' = \pi_\rho(z) p \xi' = p \pi_\rho(z) \xi' = p \eta' = 0,
\end{equation*}
 which contradicts $\norm{\eta'}=1$. \\
 Consequently, there are elements $\xi \in p H_\rho$ and $\eta \in (1-p) H_\rho$ with $\langle \xi, \pi_\rho(s)\eta \rangle \neq 0.$ In particular, we can write
 \begin{equation*}
     \xi = \lim\limits_{n \to \infty} h_n(\pi_\rho(e')) \xi, \;\;\; \eta = \lim\limits_{n \to \infty} (1 - h_{n+1})(\pi_\rho(e')) \eta.
 \end{equation*}
Thus, there is $n_0 \in \N$ such that
\begin{equation*}
    \pi_\rho(h_{n_0}(e')s(1-h_{n_0+1})(e')) = \pi_\rho(h_{n_0}(e')) \pi_\rho (s) \pi_\rho((1-h_{n_0+1})(e')) \neq 0.
\end{equation*}
As $D$ is a Cartan subalgebra, it contains an approximate unit $(u_m)_m$ of $A$ and hence we get for $m_0 \in\N$ sufficiently large that
 \begin{equation*}
     h_{n_0}(e') s (u_{m_0} - u_{m_0}h_{n_0+1}(e')) \neq 0.
 \end{equation*}
Then, $e:= h_{n_0}(e')$ and $f:=(u_{m_0}-u_{m_0}h_{n_0+1}(e'))$ are elements in $D_+^1$ as desired.\\
(ii):
Take $e,f \in D_+^1$ with  $e \perp f$ but $esf \neq 0$ from (i). Let $E: A \twoheadrightarrow D$ be the faithful conditional expectation. Then, there is $x_0 \in \hat{D}$ and $\delta >0$  such that 
\begin{equation}
\label{eq:E(esf)large}
\rho_{x_0} \circ E ((esf)^*esf) \geq \delta > 0.
\end{equation}
With 2.2 in \cite{AAP}, one can find $(k_n)_n$ excising $\rho_{x_0}$ with $k_{n+1}k_n=k_n$ and $k_n(x_0) =1$ for each $n \in \N$. Now, take $\varepsilon >0$. As $D \subset A$ is regular, there are $u_1, \dots, u_N \in N_A(D)$ with 
\begin{equation}
\label{eqlemmacharelementsofdiagproof0}
\norm{esf - \sum_{i=1}^N u_i} < \varepsilon.
\end{equation}
We may assume that
$ \mathrm{dom}(u_i) \subset \mathrm{supp}(f)$ and $ \mathrm{ran}(u_i) \subset \mathrm{supp}(e).$ For $i \in \{1, \dots, N\}$ with $x_0 \in \mathrm{dom}(u_i)$, write $y_i:=\alpha_{u_i}(x_0) $, and note that $y_i \neq x_0$ as $e \perp f$. Define the sequence $(l_{n,i})_n :=(\lambda_{n,i}u_ik_n^2u_i^*)_n$ with $\lambda_{n,i} = \norm{u_ik_n^2u_i^*}^{-1}$. This excises the pure state $\rho_{y_i} \in PS(D)$ and satisfies $l_{n,i}(y_i) \to 1$
by Lemma \ref{lemmaexcisingalphavx0}. If on the other hand $x_0 \notin \mathrm{dom}(u_i)$, we observe that
 \begin{equation}
 \label{eqlemmacharelementsofdiagproof2}
     \lim\limits_{n \to \infty} \norm{u_ik_n^2u_i^* } = 0
 \end{equation}
 since $(k_n)_n $ excises $\rho_{x_0}$. Define $J:=\{ i \, : \, x_0 \in \mathrm{dom}(u_i) \} $ and
 choose  $I \subset J $ such that $\{y_i \, : \, i \in I \} = \{y_j \, : \, j \in J\}$ but $y_i \neq y_j$ for $i \neq j \in I$. By Lemma \ref{lemma:charexcseq}, we may assume that $l_{n,i} \perp l_{n,j}$ for $n \in \N $ and $i \neq j \in I$. Moreover, upon passing to suitable subsequences, we may assume that for $i \in I$, $j \in J$ with $y_i = y_j$, we have
\begin{equation}
\label{eqlemmacharelementsofdiagproof3}
    \norm{l_{n-1,i} l_{n,j} - l_{n,j}} \overset{n}{\longrightarrow} 0,
\end{equation}
since $l_{n,i}(y_i) \to 1$. Next, let us note that for $i \in J$
\begin{equation*}
    \lim\limits_{n \to \infty} \norm{(u_iu_i^*-u_iu_i^*(y_i))u_ik_{n+1} } = 0
\end{equation*}
using that $(\lambda_{n+1,i} u_ik_{n+1}^2 u_i^*)_n$ excises $\rho_{y_i}$ and $\lambda_{n,i} \to u_iu_i^*(y_i)^{-1} >0 $ as in (\ref{eq:proof:lemma:excsingalphavx0}). With that, we get for $i \in J$ that
\begin{equation}
\label{eqlemmacharelementsofdiagproof4}
     \lim\limits_{n \to \infty} \norm{l_{n,i}u_ik_{n+1} - u_ik_{n+1} } \overset{(\ref{eq:proof:lemma:excsingalphavx0})}{=}  \lim\limits_{n \to \infty} \vert \lambda_{n,i}\vert \, \norm{(u_iu_i^*- u_iu_i^*(y_i)) u_ik_{n+1}}=0
\end{equation}
 in particular since $k_{n+1}k_n=k_n$. Altogether, we get for $n$ sufficiently large that 
\begin{align*}
    \sum\limits_{i \in I} l_{n-1,i} esf k_{n+1}     &\overset{\hspace{-0.75em}(\ref{eqlemmacharelementsofdiagproof0}), (\ref{eqlemmacharelementsofdiagproof2})}{\approx_{ \varepsilon} } \sum\limits_{i \in I} l_{n-1,i} \sum\limits_{j \in J} u_j k_{n+1} \\
    & \overset{ \hspace{-0.75em}  (\ref{eqlemmacharelementsofdiagproof4})}{\approx_\varepsilon}  \sum\limits_{i \in I} l_{n-1,i} \sum\limits_{j \in J} l_{n,j} u_j k_{n+1} \\
    & \overset{\hspace{-0.75em} (\ref{eqlemmacharelementsofdiagproof3})}{\approx_\varepsilon} \sum\limits_{j \in J} l_{n,j}u_j k_{n+1} \\
    & \overset{\hspace{-0.75em} (\ref{eqlemmacharelementsofdiagproof4})}{\approx _\varepsilon} \sum\limits_{j \in J} u_j k_{n+1} \\
    & \overset{\hspace{-0.75em}(\ref{eqlemmacharelementsofdiagproof0}),  (\ref{eqlemmacharelementsofdiagproof2})}{\approx_{\varepsilon}}  esf k_{n+1},
\end{align*}
where $a\approx_\varepsilon b$ abbreviates $\norm{a-b}\leq \varepsilon$.
As $\varepsilon$ was chosen arbitrarily, we showed that $$
\lim\limits_{n \to \infty} \norm{\sum\limits_{i \in I} l_{n-1,i} esf k_{n+1} - esf k_{n+1}}=0.
$$
This implies for sufficiently large $n$ that
\begin{equation*}
\rho_{x_0} \circ E((\sum\limits_{i \in I} l_{n-1,i} esf k_{n+1} )^* \sum\limits_{i \in I} l_{n-1,i} esf k_{n+1}) \overset{(\ref{eq:E(esf)large})}{\geq} \frac{\delta}{2}.
\end{equation*}
Thus, there is at least one $i \in I$ with $l_{n-1,i} esf k_{n+1} \nrightarrow 0$, which shows the statement.
\end{proof}

Now, let us finally prove Theorem \ref{theocharnorm}.

\begin{proof}[Proof of Theorem \ref{theocharnorm}]
First, we suppose that $v \in N_A(D)$. Let $x_0, y_0 \in \hat{D}$ and $(k_n)_n, (l_n)_n$ be sequences excising $\rho_{x_0}, \rho_{y_0} \in PS(D)$, respectively.
If $x_0 \notin \mathrm{dom}(v)$, we get $\norm{vk_n} \to 0$ by the excision property. If $y_0 \notin \mathrm{ran}(v)$, we get $\norm{v^*l_n} \to 0$ in the same way. Next, let us consider the case that $x_0 \in \mathrm{dom}(v), y_0 \in \mathrm{ran}(v)$, but $\alpha_v(x_0 )\neq y_0$. Since $\alpha_v$ from (\ref{eq:homeoalphav}) is a homeomorphism, there is $z \neq x_0$ in $\mathrm{dom}(v)$ with $\alpha_v(z) = y_0$. Take open neighbourhoods $U,W \subset \mathrm{dom}(v)$ of $x_0,z$, respectively, with $U\cap W = \emptyset$. Now, applying Lemma \ref{lemma:charexcseq} for the open neighbourhoods $U$ of $x_0$ and $\alpha_v(W)$ of $y_0$ and using (\ref{eq:homeoalphav}) yields
\begin{equation*}
        \lim\limits_{n \to \infty} \norm{k_nv^*l_n^2vk_n} = \lim\limits_{n \to \infty} \sup\limits_{x \in \hat{D}\backslash (U \cup W)} \vert k_n^2(x) l_n^2(\alpha_v(x)) v^*v(x) \vert = 0,
    \end{equation*}
i.\@e.\@ $\lim_{n \to \infty} \norm{l_n v k_n}=0$. Thus, $\norm{l_nvk_n} \nrightarrow 0$ can only occur if $x_0 \in \mathrm{dom}(v), y_0 \in \mathrm{ran}(v)$ and $\alpha_v(x_0) = y_0$. Therefore, $v$ satisfies the NEP.\\
For the converse, let $v \in A^1$ such that $v$ and $v^*$ satisfy the NEP. To show $v\in N_A(D)$ via contradiction, we assume there is $d \in D_+^1$ with $v^*dv \notin D$. Apply Lemma \ref{lemma:existenceef} to get $\varepsilon>0$ and sequences $(k_n)_n, (l_n)_n$ excising $\rho_{x_0}, \rho_{y_0}\in PS(D)$, respectively, for $x_0 \neq y_0 \in \hat{D}$, with 
\begin{equation}
\label{eq:proof:theocharnorm:1}
    \norm{l_n v^*dvk_n} \geq \varepsilon, \; n \in \N.
\end{equation}
By Lemma \ref{lemma:charexcseq}, we may suppose that $(k_n)_n$ and $(l_n)_n$ are decreasing and satisfy $k_n(x_0)=1$, $l_n(y_0)=1$ for each $n$.
By (\ref{eq:proof:theocharnorm:1}) and since $D \subset A$ is regular, there is $u \in N_A(D)$ with 
\begin{equation*}
l_nu^*dvk_n \nrightarrow 0.
\end{equation*}
With Lemma \ref{lemmaexcisingalphavx0}, the sequence $(m_n)_n:=(\norm{ul_n^2u^*}^{-1} ul_n^2u^*)_n$ excises $\rho_{\alpha_u(y_0)}$ and satisfies $m_n(\alpha_u(y_0)) \to 1$. Moreover, we have $m_n v k_n \nrightarrow 0$ by choice of $u$. We put $z_0:= \alpha_u(y_0)$.\\
Now, let $\varepsilon\geq \delta > 0$. As $D \subset A$ is regular, there are  $u_1, \dots, u_N \in N_A(D)$ with 
\begin{equation}
\label{eq:regularityforv}
\norm{\sum_{i=1}^N u_i - v} < \delta.
\end{equation}
Consider  $i \in \{1, \dots, N \}$ with $y_0 \in \mathrm{dom}(u_i)$ and put $\lambda_{n,i}:= \norm{u_il_n^2u_i^*}^{-1}$. Then, by Lemma \ref{lemmaexcisingalphavx0}, the sequence $(\lambda_{n,i} u_il_n^2u_i^*)_n$ excises pure states $\rho_i \in PS(D)$. We note that 
\begin{equation}
\label{eqtheocharnormproof1}
    \norm{l_nu_i^*dvk_n}^2 \leq \lambda_{n,i}^{-1}  \; \norm{ (\lambda_{n,i} u_i l_n^2u_i^*) v k_n} \overset{n}{\longrightarrow} 0
\end{equation}
if $\rho_i \neq \rho_{z_0}$ by the NEP for $v$. Put $$I := \{ i \in \{ 1, \dots, N \} \, : \, y_0 \in \mathrm{dom}(u_i), \, \rho_i = \rho_{z_0} \}.$$ In particular, we have $l_nu_i^*dvk_n \to 0$ for all $i \notin I$.
By the comment after Lemma \ref{lemma:charexcseq}, we may assume, after possibly passing to a subsequence of $(l_n)_n$ still excising $\rho_{y_0}$, that
\begin{equation*}
\lim\limits_{n \to \infty} \norm{\lambda_{n,i} u_il_{n}^2u_i^*(1-m_n)}= 0,\; i \in I. 
\end{equation*}
This implies
\begin{equation}
\label{eqtheocharnormproof2}
   \lim\limits_{n \to \infty}  \norm{l_nu_i^*(1-m_n)}^2= \lim\limits_{n \to \infty} \lambda_{n,i}^{-1} \norm{(1-m_n) (\lambda_{n,i}u_i l_n^2u_i^*)(1-m_n)} =0, \; i \in I.
\end{equation}
We thus obtain for sufficiently large $n$
\begin{align*}
   l_nv^*dvk_n \overset{(\ref{eq:regularityforv})}{\approx}_{\hspace{-0.3em}\delta} \sum\limits_{i=1}^N l_nu_i^*dvk_n \overset{(\ref{eqtheocharnormproof1})}{\approx}_{\hspace{-0.3em}\delta} \sum\limits_{i\in I} l_nu_i^*dvk_n \overset{(\ref{eqtheocharnormproof2})}{\approx}_{\hspace{-0.3em}\delta} \sum\limits_{i\in I} l_nu_i^*m_ndvk_n \overset{(\ref{eq:regularityforv}),(\ref{eqtheocharnormproof1})}{\approx}_{\hspace{-1.1em}\delta \hspace{0.5em}} l_nv^*m_ndvk_n.
\end{align*}
If $\delta$ was chosen sufficiently small depending on $\varepsilon$, we get with  (\ref{eq:proof:theocharnorm:1}) that $l_nv^*m_n \nrightarrow 0$ and $k_nv^*m_n \nrightarrow 0$. This contradicts the NEP for $v^*$.
\end{proof}

\section[]{A $\mathrm{C}^*$-diagonal in $\mathcal{Z}$}
\label{sec:diaginZ}
The plan for detecting a $\mathrm{C}^*$-diagonal in $\mathcal{Z}$ now is to determine $\mathrm{C}^*$-diagonals in each of the building blocks in the inductive system from Theorem \ref{theo:constrindlim} and then show that the inductive limit of these is a $\mathrm{C}^*$-diagonal in $\mathcal{Z}$. In the following, we write $\Ad{a}(b):=aba^*$ for $a,b$ in a $\mathrm{C}^*$-algebra $A$.

\begin{prop}
\label{prop:tD}
Let $L \geq 2$ and $\bc_1, \dots, \bc_L, \bs$ denote the generators of $\tZ{L}{L+1}$. We define 
\begin{equation}
\tD{L}:= \mathrm{C}^*(\bc_i^*\bc_i, \, \bs^*\bs, \, \bc_i^*\bs\bs^* \bc_i,\,  \bc_j^*\bs \bc_i^*\bc_i\bs^*\bc_j \,:\, 1 \leq i,j \leq L ) \subset \tZ{L}{L+1}.
\end{equation}
Then, $\tD{L}$ is a $\mathrm{C}^*$-diagonal in $\tZ{L}{L+1}. $
\end{prop}

\begin{proof}
With (\ref{eq:Therelations}) and (\ref{eq:basicCT}), it is elementary to check that $\tD{L} $ is abelian.
Moreover, $\tD{L}$ is regular in $\tZ{L}{L+1}$ since each generator $\bc_1, \dots, \bc_L,\bs$ normalises the set 
\begin{equation*}
\{\bc_i^*(\bs\bs^*)^k\bc_i, (\bc_i^*\bc_i)^n (\bs^*\bs)^k, \bc_j^*\bs f(\bc_i^*\bc_i) \bs^*\bc_j \, : \,  1 \leq i,j \leq L, \, k,n \in \N \cup \{0\},  f \in C_0((0,1]) \}  ,
\end{equation*}
which is a generating set of $\tD{L}$ closed under multiplication. To show that the pure state extension property is satisfied, recall the ideals $I_1, I_2, J$ in $\tZ{L}{L+1}$ from Proposition \ref{prop:essideal}, the functions $a,b \in C_0((0,1])$ and the c.p.c.\@ order zero maps $\hat{\sigma}_L, \hat{\sigma}_{L+1}, \hat{\theta}, \sigma_{L}, \sigma_{L+1}$ and $\theta$ used therein.
Further, recall Properties (a)--(d) from the proof of Proposition \ref{prop:essideal} and observe that, additionally, the following hold:
\begin{spacing}{1,1}
\begin{enumerate}
\item[(e)] $\sigma_L(\D_{L}), \sigma_{L+1}(\D_{L+1}),  \theta(\D_L \otimes \D_{L+1}) \subset \tD{L}$, 
\item[(f)] $\mathrm{C}^*( \sigma_L(\D_L), \sigma_{L+1}(\D_{L+1}), \theta(\D_L \otimes \D_{L+1})) = \tD{L}$.
\end{enumerate}
\end{spacing}
Here,  $\D_L \subset \M_L$ and $\D_{L+1} \subset \M_{L+1}$ denote the diagonal matrices. For (e), note that $\sigma_L(\D_{L}), \sigma_{L+1}(\D_{L+1}) \subset \tD{L}$ and $\theta(e_{i,i} \otimes e_{j,j}) \in \tD{L}$ for $j =0$ are immediate from the respective definitions. For $j\geq1$, we compute that
\begin{equation*}
\hat{\theta}(e_{i,i} \otimes e_{j,j} ) = \bc_j^*\bs\bc_i^*\bc_i\bs^*\bc_j \in \tD{L}.
\end{equation*}
Since $\theta(e_{i,i} \otimes e_{j,j}) $ is a function in $\hat{\theta}(e_{i,i} \otimes e_{j,j})$, this implies $\theta(e_{i,i} \otimes e_{j,j}) \in \tD{L}$. 
\vspace{0.1ex} To see (f), write $D:= \mathrm{C}^*(\sigma_L(\mathbb{D}_L), \sigma_{L+1}(\mathbb{D}_{L+1}), \theta ( \mathbb{D}_L \otimes \mathbb{D}_{L+1}))$ and observe that
\begin{equation*}
    \bc_i^*a(\bs\bs^*)\bc_i, a(\bs^*\bs), b(\bs^*\bs), b(\bc_i^*\bc_i), a(\bc_i^*(1-\bs\bs^*)\bc_i), b(\bc_i^*\bs\bs^*\bc_i), \Ad{\pi_{\hat{\sigma}_{L+1}}(e_{j,0})}(b(\bc_i^*\bc_i))\in D
\end{equation*}
for any $1 \leq i \leq L,$ $0\leq j\leq L$. Appropriately combining these using Lemma \ref{lemma:cpctricks} and (\ref{eq:Therelations}), (\ref{eq:basicCT}), one checks that $D$ contains each generator of $\tD{L}$. Now, let us recall the isomorphisms from (\ref{eq:prop:essideal}). Restricting these to $\tD{L}$, we get
\begin{align}
Q_{I_1} \circ \sigma_{L+1}\vert_{\D_{L+1}}: & \, \mathbb{D}_{L+1} \overset{\cong}{\longrightarrow} \tD{L} / (I_1\cap \tD{L}) \nonumber\\
Q_{I_2} \circ \sigma_{L}\vert_{\D_{L}}:& \, \mathbb{D}_{L} \overset{\cong}{\longrightarrow}  \tD{L}/ (I_2 \cap \tD{L})  , \label{eq:isosdiagonals}\\
\alpha': & \, C_0((0,1)) \otimes \D_L \otimes \D_{L+1} \overset{\cong}{\longrightarrow}  J \cap \tD{L}  \nonumber,
\end{align}
using facts (e) and (f) from above. Here, $Q_{I_1}$ and $Q_{I_2}$ denote the respective quotient maps. With this, we can show the pure state extension property for $\tD{L}$ relative to $\tZ{L}{L+1}$. Take a pure state $\rho$ on $\tD{L}$. Then, the following cases may occur.\\
\textbf{Case 1}: $\rho$ annihilates $I_1 \cap \tD{L}$. \\
There is a pure state $\rho'$ on $ \D_{L+1}$  such that $\rho' \circ Q_{I_1}\vert_{\tD{L}} = \rho$, using the first identification in (\ref{eq:isosdiagonals}). Now, let $\rho_1, \rho_2 $ be two state extensions of $\rho$ on $\tZ{L}{L+1}$. Then, $\rho_1, \rho_2 $ both annihilate $I_1$ and there are states $\rho_1', \rho_2'$ on $\M_{L+1}$ such that
\begin{equation*}
\rho_1' \circ Q_{I_1}= \rho_1 , \;\;\; \rho_2' \circ Q_{I_1}= \rho_2.
\end{equation*}
Thereby, $\rho_1'$ and $ \rho_2'$ are both extensions of $\rho'$. However, since $\D_{L+1} \subset \M_{L+1}$ satisfies the pure state extension property, $\rho_1'$ and $ \rho_2'$ coincide. Hence, $\rho$ has a unique state extension.\\
\textbf{Case 2}: $\rho$ annihilates $I_2 \cap \tD{L}$. This case can be dealt \vspace{0.5pt} with analogously.\\ 
\textbf{Case 3}\vspace{1.5pt}: $\rho(\theta(1)) >0$.\\
First note that, by (\ref{eq:prop:essideal}), (\ref{eq:isosdiagonals}), and since $J \trianglelefteq \tZ{L}{L+1}$ is an ideal, $J\cap \tD{L} \subset  \tZ{L}{L+1}$ has the pure state extension property. If now $\rho(\theta(1)) >0$, then $\rho\, \vert_{J \cap \tD{L}}$  is a pure state on $J\cap \tD{L}$ that extends uniquely to a state $\rho'$ on $\tZ{L}{L+1}$. Hence, $\rho'$ uniquely extends $\rho$.\\ 

The pure state extension property already ensures that $\tD{L}$ is maximal abelian and that there is a unique conditional expectation $\tilde{P}:\tZ{L}{L+1} \to \tD{L}$ (e.g.\@ see \cite{Exel}). So it remains to check that $\tilde{P}$ is faithful.
We take $x \in \tZ{L}{L+1}$ non-zero and observe that
$ x^*\theta(1)x$ is positive and non-zero as $J$ is essential. Now, the conditional expectation
\begin{equation*}
C_0((0,1)) \otimes \M_L \otimes \M_{L+1} \cong J \overset{\tilde{P}\vert _J}{\longrightarrow} J \cap \tD{L} \cong C_0((0,1)) \otimes \D_L \otimes \D_{L+1}
\end{equation*}
coincides with the unique such conditional expectation, which is faithful, and hence $\tilde{P}\vert_J$ is faithful. We conclude that $\tilde{P}(x^*x)  \geq \tilde{P}\vert_J(  x^* \theta(1)x ) \neq 0.$
\end{proof}

To show that these diagonals give rise to a $\mathrm{C}^*$-diagonal in $\mathcal{Z}$, we use the following theorem. 
\begin{theo}[1.10 in \cite{Li1}]
\label{theo:XinLi}
 Let $(D_n, A_n)_n$ be a sequence of Cartan pairs with normalisers $N_{A_n}(D_n)$ and faithful conditional expectations $P_n: A_n \to D_n$. Further, let $\Phi_n :A_n \to A_{n+1}$ be injective $^*$-homomorphisms such that 
\begin{equation*}
\Phi_n(D_n) \subset D_{n+1}, \;\;\; 
\Phi_n(N_{A_n}(D_n)) \subset N_{A_{n+1}}(D_{n+1}), \;\;\;\Phi_{n}\vert_{D_n} \circ P_n = P_{n+1} \circ \Phi_n,
\end{equation*}
hold for each $n$. Then, $D:=\varinjlim(D_n, \Phi_n\vert_{D_n})$ is a Cartan subalgebra in $A:=\varinjlim(A_n, \Phi_n)$. If each $D_n $ was a $\mathrm{C}^*$-diagonal in $A_n$, then $D$ is a $\mathrm{C}^*$-diagonal in $A$. 
\end{theo}

Note that $\Phi_n(N_{A_n}(D_n)) \subset N_{A_{n+1}}(D_{n+1})$ already implies $\Phi_n(D_n) \subset D_{n+1}$ since elements of $D_n$ are trivially normalisers and $D_{n+1}$ contains an approximate unit of $A_{n+1}$ (cf.\@ 1.6 in \cite{LiaoTikuisis}). Moreover, compatibility of the conditional expectations and the connecting maps is automatic in the case of unital $\mathrm{C}^*$-diagonals, as stated in the following lemma.
\begin{lemma}
\label{lemma:condexpindlim}
Let $D_i \subset A_i$  be unital $\mathrm{C}^*$-diagonals with faithful conditional expectations $P_i: A_i \to D_i$, $i=1,2$. Further, let $\Phi:A_1 \to A_2$ be a unital injective $^*$-homomorphism with $\Phi(D_1) \subset D_2$. Then, we have $\Phi\vert_{D_1} \circ P_1 = P_2 \circ \Phi$.
\end{lemma}
\begin{proof}
Assume that there is a non-zero, positive $x \in A_1$ such that $\Phi\vert_{D_1}  \circ P_1(x) \neq  P_2 \circ \Phi (x)$. Then, there is a pure state $\rho$ on $D_2$ such that 
\begin{equation*}
\rho \circ \Phi\vert_{D_1}  \circ P_1(x) \neq \rho \circ  P_2 \circ \Phi (x).
\end{equation*}
Since $\Phi$ is unital and $1_{A_i} \in D_i$ for $i=1,2$, we have that $\rho \circ \Phi\vert_{D_1} $ is a pure state on $D_1$. However, then $\rho \circ \Phi\vert_{D_1}  \circ P_1$ and  $\rho \circ  P_2 \circ \Phi$ are different states on $A_1$ both extending $\rho \circ \Phi\vert_{D_1}$ as $\Phi(D_1) \subset D_2$. This contradicts $D_1$ being a $\mathrm{C}^*$-diagonal in $A_1$.
\end{proof} 

It remains to check that in the construction of the Jiang--Su algebra in Theorem \ref{theo:constrindlim}, the normalisers in each finite stage are again mapped to normalisers by the connecting $^*$-homomorphisms. To do that, let us fix $L \geq 2$ and $L'\geq L$ as in (\ref{eq:choiceofL'}) and consider $\Phi_{L, L'}:\tZ{L}{L+1} \to \tZ{L'}{L'+1}$ from Proposition \ref{theo:starhom}. We write 
\begin{equation*}
N:= N_{\tZ{L}{L+1}}(\tilde{D}_{L,L+1}), \;\;N':= N_{\tZ{L'}{L'+1}}(\tilde{D}_{L',L'+1})
\end{equation*}
and $X_L$ for the spectrum of $\tD{L}.$ Further, we write $\bar{c}_1, \dots, \bar{c}_L,\bar{s}$ and $c_1, \dots, c_{L'}, s$ for the generators of $\tZ{L}{L+1}$ respectively $\tZ{L'}{L'+1}$, as in Section \ref{sec:construction}. Let us first consider a well-tractable class of normalisers, namely words in the generators of $\tZ{L}{L+1}$, denoted by $\mathcal{W}(\bar{c}_1, \dots, \bar{c}_L, \bar{s})$. 
\begin{lemma}
\label{lemma:nicenormpreserved}
    Let $w $ be a word in the generators of $\tZ{L}{L+1}$, i.e.\@ $w\in \mathcal{W}(\bar{c}_1, \dots, \bar{c}_L, \bar{s})$. Then, we have $\Phi_{L,L'}(w) \in N'=N_{\tZ{L'}{L'+1}}(\tilde{D}_{L',L'+1})$.
\end{lemma}

\begin{proof}
First, let us establish the following.\\
   \textbf{Claim}:
  Let  $1\leq i,j \leq L'$ and $d,d' \in \tD{L'}$. Then, we have $\piph{1,2}d, \pips{i}{j}d' \in N'$ (where $ \pi_\varphi, \pi_\psi$ are defined as in Section \ref{sec:construction} before (\ref{eq:defi:tildec})), provided for every $\varepsilon >0$ there are $f, f' \in C_0((0,1])$ with $\norm{f(s^*s)d-d} < \varepsilon$ and $  \norm{f'(c_j^*c_j)d'-d'} < \varepsilon.$\\
    Indeed, let $\varepsilon > 0$ and $f \in C_0((0,1])$, $d \in \tD{L'}$ as above. Take $(a_n)_{n}$ a sequence of functions in $C_0((0,1])$ with $a_n  \mathrm{id}_{(0,1]} f \to f$ uniformly. Then, we compute
    \begin{equation*}
    \piph{1,2}d \approx_\varepsilon \piph{1,2} f(s^*s) d = \lim\limits_{n \to \infty}  (ss^*)^{\frac{1}{2}}s a_n(s^*s) f(s^*s)d \in N'
    \end{equation*}
   as $N'$ is a closed set and closed under taking products. Here, $a \approx_\varepsilon b$ abbreviates $\norm{a-b} \leq \varepsilon$. Since $\varepsilon$ was arbitrary, we get $\piph{1,2}d \in N'$. Analogously, one shows $\pips{i}{j}d' \in N'$. \\
   By the definitions in (\ref{eq:defictilde:a}), (\ref{eq:defictilde:b}), (\ref{eq:defi:tildec}), $\Phi_{L,L'}(\bar{c}_l)$ is a finite sum of pairwise orthogonal elements of the form considered in the claim for each $1\leq l \leq L$. Therefore, we have $\Phi_{L,L'}(\bar{c}_l) \in N'$. Analogously, we get $\Phi_{L,L'}(\bar{s}) \in N'$.
   \end{proof}

Now, denote $\mathcal{M}:= \{ w \in N \, \vert \, \Phi_{L,L'}(w) \in N'\} \subset N$.
Further, recall the normaliser excision property (NEP) from Definition \ref{defi:propertyNEP}. Then, $\mathcal{M} $ is a large subfamily of $N$ in the following sense: For each normaliser $v$ and each point $x\in X_L$ (the spectrum of $\tD{L}),$ which is transported to $y\in X_L$ by $v$,  there already is an element of $\mathcal{M}$ also transporting $x$ to $y$. This is made precise in the following lemma. 

\begin{lemma}
\label{lemma:exnicenormaliser}
    Let $v \in N$ and $x,y \in X_L$. Further, let $\rho_1,\rho_2 \in PS(\tD{L})$ be the evaluations in $x,y$ excised by sequences $(k_n)_n, (l_n)_n$ in $\tD{L}$, respectively. Suppose that $\rho_1(k_n)=1,$ $\rho_2(l_n)=1$ for each $n \in \N $ and $l_n v k_n \nrightarrow 0$. Then, there exist $w \in \mathcal{M}$ and $\mu \in (0, \infty)$, $\lambda \in \C \backslash \{0\}$ such that, 
    \begin{itemize}
    \item[\upshape{(i)}]$ \mu \rho_1\circ \Ad{w^*} = \rho_2,$
    \item[\upshape{(ii)}] after possibly passing to subsequences of $(k_n)_n$ and $(l_n)_n$, we have  
    \begin{equation*}
        (l_nvk_n - \lambda l_n wk_n)_n \overset{n}{\longrightarrow} 0.
    \end{equation*}
    \end{itemize}
    
\end{lemma}

\begin{proof}
Let us fix some notation. We write $\mathrm{ev}_i$ for the evaluation in the $i$-th entry of the diagonal matrices $\mathbb{D}_L$, $\mathrm{ev}_j$ on $\mathbb{D}_{L+1}$ likewise, and $\mathrm{ev}_{t,i,j}$ for the pure state defined on $C_0((0,1)) \otimes \mathbb{M}_L \otimes \mathbb{M}_{L+1}$ by
   \begin{equation*}
      \mathrm{ev}_{t,i,j}(f \otimes e_{k,l} \otimes e_{m,n}) = 
      \begin{cases}
        f(t) & \text{if } k=l=i,\, m=n=j, \\
        0 &\text{else},
       \end{cases}
   \end{equation*}
   for $1\leq i \leq L, 0 \leq j \leq L$ and $t \in (0,1)$. \\
    (i): Recall from the proof of Theorem \ref{theocharnorm} that $l_nvk_n \nrightarrow 0$ implies $\alpha_v(x) =y$ and $v^*v(x) >0$ with $\alpha_v$ from (\ref{eq:homeoalphav}).
    With that, we observe for any ideal $I \unlhd \tZ{L}{L+1}$, that $\rho_1$ does not vanish on $I\cap \tD{L}$ if and only if that holds for $\rho_2$. Further recall the ideals $I_1, I_2, J$, the functions $a,b$ and the c.\@p.\@c.\@ maps $\sigma_L, \sigma_{L+1}, \theta$ from Proposition \ref{prop:essideal}. Now, the following three cases may occur.\\
\textbf{Case 1}: $\rho_1, \rho_2$ annihilate $I_1 \cap \tD{L}$. \\
Recall the isomorphism $\beta:=Q_{L+1} \circ \sigma_{L+1}\vert_{\mathbb{D}_{L+1}}$ from (\ref{eq:isosdiagonals}), where $Q_{L+1}$ denotes the quotient map $ \tD{L} \twoheadrightarrow \tD{L} / (I_1 \cap \tD{L})$. For the rest of this proof, let us denote the restriction $\sigma_{L+1}\vert_{\mathbb{D}_{L+1}}$ again by $\sigma_{L+1}$. Now, there are $\rho_1', \rho_2'\in PS(\mathbb{D}_{L+1})$ such that $ \rho_m' \circ  \beta^{-1} \circ Q_{L+1} = \rho_m$ for $m=1,2$. In particular, we have $\rho_m' = \rho_m \circ \sigma_{L+1}$. To show (i) in this case, it now suffices to show that there is $w \in \mathcal{W}(\bar{c}_1, \dots, \bar{c}_L, \bar{s}) \subset \mathcal{M}$ with 
\begin{equation*}
\rho_1 \circ \Ad{w^*}\circ\sigma_{L+1} = \rho_2'.
\end{equation*}
Indeed, suppose we find such a $w$. Let $\bar{\rho}_1$ be the pure state extension of $\rho_1$ to $\tZ{L}{L+1}$. Then, we get with Fact 1.2 in \cite{Pit} that 
    \begin{equation*}
    \rho_2 = \rho_2' \circ \beta^{-1} \circ Q_{L+1} = \rho_1 \circ \Ad{w^*}\circ (\sigma_{L+1} \circ \beta^{-1} \circ Q_{L+1}) = \Ad{\bar{\rho_1}(w^*)} \circ \rho_1    = \rho_1\circ \Ad{w^*}.
    \end{equation*}
    Now, let us first suppose that $\rho_1' = \mathrm{ev}_i, \rho_2' = \mathrm{ev}_j$ for some $1 \leq i,j \leq L$. Then, we compute for $1 \leq k \leq L$
    \begin{equation*}
        \rho_1(\bar{c}_i^*\bar{c}_j \sigma_{L+1}(e_{k,k})\bar{c}_j^*\bar{c}_i) = \delta_{j,k} \rho_1( \sigma_{L+1}(e_{i,i})) = \delta_{j,k} = \rho_2'(e_{k,k})
    \end{equation*}
    and note for $k=0$ that
    \begin{equation*}
        \rho_1(\bar{c}_i^*\bar{c}_j\sigma_{L+1}(e_{0,0})\bar{c}_j^*\bar{c}_i) \overset{(\ref{eq:basicCT})}{=} \rho_1(\bar{c}_i^*\bar{c}_1^2 a(\bar{s}^*\bar{s})\bar{c}_i) = 0 = \rho_2'(e_{0,0})
    \end{equation*}
    as $\bar{c}_1^2 a(\bar{s}^*\bar{s}) \in I_1 \cap \tD{L}$.
    This shows $\rho_1\circ \Ad{\bar{c}_i^*\bar{c}_j} \circ \sigma_{L+1}= \rho_2'$. If $ \rho_1'= \mathrm{ev}_i$ for some $1\leq i \leq L$, $ \rho_2' = \mathrm{ev}_0$, we similarly get $\rho_2' = \rho_1 \circ \Ad{\bar{c}_i^*\bar{s}} \circ \sigma_{L+1}.$ Analogously for $\rho_1'= \mathrm{ev}_0$ and $\rho_2'= \mathrm{ev}_j$ for some $1 \leq j \leq L$, we have $\rho_2' = \rho_1 \circ \Ad{\bar{s}^*\bar{c}_j} \circ \sigma_{L+1}.$\\
    \textbf{Case 2}: $\rho_1, \rho_2$ annihilate $I_2$. 
    This case can be dealt with analogously.\\
    \textbf{Case 3}: $\rho_1(\theta(1)), \rho_2(\theta(1)) >0$.\\
    In particular, $\rho_1, \rho_2$ restrict to pure states on $J \cap \tD{L}$. Recall the definition of the $^*$-isomorphism $\alpha$ from (\ref{eq:definalpha}) that restricts to a $^*$-isomorphism 
    \begin{equation*}
        \alpha':C_0((0,1)) \otimes \mathbb{D}_{L} \otimes \mathbb{D}_{L+1} \to J \cap \tD{L}.
    \end{equation*}
    Define pure states $\rho_1' := \rho_1 \circ \alpha', \rho_2' :=\rho_2 \circ \alpha'$ on $C_0((0,1)) \otimes \mathbb{D}_{L} \otimes \mathbb{D}_{L+1}$.  In particular, we have 
    $\rho_1'= \mathrm{ev}_{t_1,i ,j} , \, \rho_2'= \mathrm{ev}_{t_2,k,l}$
for some $t_1, t_2 \in (0,1), 1 \leq i,k \leq L, 0 \leq j,l \leq L$. If $t_1 \neq t_2$, then $\rho_1$ annihilates the ideal 
\begin{equation*}
    \alpha'(\{ f \otimes x \otimes y \, \vert \, f(t_1)=0 \}) \unlhd J\cap \tD{L} \unlhd \tD{L},
\end{equation*}
but $\rho_2$ does not, a contradiction. Thus, we may assume $t:= t_1 = t_2$. Now, write $\iota:= \mathrm{id}_{(0,1]}$ and note that
\begin{equation*}
    \rho_2 \circ \alpha'= \frac{1}{(t-t^2)^2} \, \rho_1 \circ \alpha' \circ \Ad{(\iota-\iota^2)\otimes e_{i,k} \otimes e_{j,l}} .
\end{equation*}
    With the definition of $\alpha'$, we compute
\begin{equation*}
    \alpha' ((\iota-\iota^2)\otimes e_{k,i} \otimes e_{l,j}) = \hat{\theta} (e_{k,i} \otimes e_{l,j}) = 
    \begin{cases}
\bar{c}_l^*\bar{s}\bar{c}_k^*\bar{c}_i\bar{s}^*\bar{c}_j & \text{ if } 1\leq j,l \leq L,\\
    \bar{c}_l^*s(\bar{s}^*\bar{s})^\frac{1}{2}\bar{c}_k^*\bar{c}_i  &  \text{ if } j=0, 1 \leq l \leq L,\\
    \bar{c}_k^*\bar{c}_i (\bar{s}^*\bar{s})^\frac{1}{2} s^*c_j & \text{ if }  l=0, 1 \leq j \leq L,\\
    \bar{s}^*\bar{s}\bar{c}_k^*\bar{c}_i & \text{ if } j=l= 0.    
    \end{cases}
\end{equation*}
In particular, $\hat{\theta}(e_{k,i} \otimes e_{l,j}) \in \mathcal{M}$ as a product of generators of $\tZ{L}{L+1}$ and elements of $\tD{L}$.
This shows statement (i) of the lemma for $ \mu= \frac{1}{(t-t^2)^2}$ and $w:= \hat{\theta} (e_{k,i} \otimes e_{l,j}).$\\
(ii): We will show that (i) implies (ii). We observe that $x \in \mathrm{dom}(w)$ and
\begin{equation*}
    1 = \rho_2(l_n) = \mu \rho_1(w^*l_nw) = \mu w^* l_n w (x) \overset{(\ref{eq:definalpha})}{=}\mu l_n(\alpha_w(x)) w^*w(x),
\end{equation*}
for any $n \in \N$ where $\alpha_w$ is defined as in (\ref{eq:homeoalphav}). By Lemma \ref{lemma:charexcseq}, we get $\alpha_w(x)=y$. Therefore and since $w$ is a normaliser, there is a sequence $(\lambda_n)_n$ in $  \C \backslash \{{0}\}$ such that $(\lambda_n wk_n^2w^*)_n$ excises $\rho_2$ by Lemma \ref{lemmaexcisingalphavx0}. By possibly passing to suitable subsequences with Lemma \ref{lemma:charexcseq}, we may assume 
\begin{equation} 
l_n (1 - \lambda_n wk_n^2w^*)\overset{n}{\longrightarrow} 0. \label{eq:proof:lemma:exnicenormaliser}
\end{equation}
Take $\varepsilon > 0.$ Recall from Corollary \ref{corexcisingextensions} that the unique extension of $\rho_1$ to $\tZ{L}{L+1}$, denoted by $\bar{\rho}_1$, is still excised by $(k_n)_n$.  Then, we have for sufficiently large $n$
\begin{align*}
   l_nvk_n  
    \approx_\varepsilon & \,l_n ( \lambda_n wk_n^2w^*)vk_n \\ \approx_\varepsilon & \,\lambda_n l_n w \bar{\rho}_1(w^*v)k_n^3 \\
    \approx_\varepsilon \, &\frac{\bar{\rho}_1(w^*v)}{w^*w(x)} l_n (\lambda_n w k_n^2 w^*) w k_n \\
    \approx_\varepsilon \, & \frac{\bar{\rho}_1(w^*v)}{w^*w(x)} l_n wk_n,
\end{align*}
by the excision property and (\ref{eq:proof:lemma:exnicenormaliser}). This implies statement (ii) with $\lambda:=\frac{\bar{\rho}_1(w^*v)}{w^*w(x)}$.
\end{proof}

With this tool at hand, we can show that general normalisers are preserved under $\Phi_{L,L'}.$

\begin{lemma}
\label{lemma:normpreserved}
    Let $v\in N$. Then, we have $\Phi_{L,L'}(v) \in N'.$
\end{lemma}

\begin{proof}
    Write $X_{L'}$ for the spectrum of $\tD{L'}$ and assume that $\Phi_{L,L'}(v) \notin N'$. Then, by Theorem \ref{theocharnorm} there exist $x',y'\neq z' \in X_{L'}$ and sequences $(k_n')_n, (l_n')_n, (m_n')_n$ in $\tD{L'}$ excising the point evaluations $\rho_{x'}, \rho_{y'}, \rho_{z'} \in PS(\tD{L'})$, respectively, such that 
    \begin{equation}
      \label{eq:proof:lemma:normpreserved:1}
    l_n' \Phi_{L,L'}(v) k_n' \nrightarrow 0, \,\,\;m_n' \Phi_{L,L'}(v) k_n' \nrightarrow 0.
    \end{equation}
    We may additionally assume $\rho_{x'}(k_n')=1, \rho_{y'}(l_n')=1, \rho_{z'}(m_n')=1$ for $n \in \N.$
    Precomposing with $\Phi_{L,L'}$ yields pure states 
    \begin{equation*}
        \rho_x:= \rho_{x'} \circ \Phi_{L,L'}, \;\; \rho_y:= \rho_{y'}  \circ \Phi_{L,L'},  \,\, \rho_z:= \rho_{z'}  \circ \Phi_{L,L'},
    \end{equation*}
    on $\tD{L}$ given by point evaluations in some $x,y,z \in X_L$. Take excising sequences $(k_n)_n, (l_n)_n, (m_n)_n$ in $\tD{L}$ for $\rho_x, \rho_y, \rho_z$, respectively, such that $k_n(x) =1$, $l_n(y)=1$, $m_n(z)=1$ for each $n \in \N$. Upon passing to suitable subsequences, we have 
    \begin{equation}
    \label{eq:proof:lemma:normpreserved:2}
        \lim\limits_{n \to \infty} \Phi_{L,L'}(k_n) k_n' -k_n'= 0, \;\; \lim\limits_{n \to \infty}\Phi_{L,L'}(l_n) l_n' -l_n' = 0 ,\;\; \lim\limits_{n \to \infty} \Phi_{L,L'}(m_n) m_n' -m_n'= 0
    \end{equation}
    by the excision property. With (\ref{eq:proof:lemma:normpreserved:1}), this implies $l_n v k_n, m_n v k_n \nrightarrow 0$. By the NEP for $v$ (see Definition \ref{defi:propertyNEP}), we get $y=z$ and thus may assume $(l_n)_n = (m_n)_n$. Moreover, after passing to suitable subsequences with Lemma \ref{lemma:exnicenormaliser}, there are $w \in \mathcal{M}$ and $\lambda \in \C \, \backslash \{ 0 \}$ such that 
    \begin{equation}
    \label{eq:proof:lemma:normpreserved:3}
        \lim\limits_{n \to \infty} l_nvk_n - \lambda l_n w k_n=0 .
    \end{equation}
    Now, take $\varepsilon > 0$ and observe that for $n $ sufficiently large
    \begin{align*}
        \lambda l_n' \Phi_{L,L'}(w) k_n' \overset{(\ref{eq:proof:lemma:normpreserved:2})}{\approx}_\varepsilon \lambda l_n'\Phi_{L,L'}(l_nwk_n)k_n' \overset{(\ref{eq:proof:lemma:normpreserved:3})}{\approx}_\varepsilon l_n' \Phi_{L,L'}(l_nvk_n)k_n' \overset{(\ref{eq:proof:lemma:normpreserved:2})}{\approx}_\varepsilon l_n' \Phi_{L,L'}(v) k_n' \nrightarrow 0.
    \end{align*}
    Thus, we have $l_n'\Phi_{L,L'}(w) k_n' \nrightarrow 0.$ By the same argument with $(m_n')_n$ in place of $(l_n')_n$, we also get $m_n'\Phi_{L,L'}(w) k_n' \nrightarrow 0.$ This is a contradiction as  $\Phi_{L,L'}(w)$ is a normaliser by Lemma \ref{lemma:nicenormpreserved} and hence satisfies the NEP.
\end{proof}
Altogether, we get a $\mathrm{C}^*$-diagonal in the Jiang--Su algebra $\mathcal{Z}$ resulting from our construction of $\mathcal{Z}$ in Theorem \ref{theo:constrindlim}.

\begin{theo}
\label{theo:diaginZ}
    Let $(L_n)_n$ be an increasing sequence of integers as in Theorem \ref{theo:constrindlim}, i.e.\@ such that $\mathcal{Z} \cong \varinjlim (\tZ{L_n}{L_{n+1}}, \Phi_{L_n, L_{n+1}})$ with connecting $^*$-homomorphisms explicitly defined as in Proposition \ref{theo:starhom}. Then, there is a $\mathrm{C}^*$-diagonal in $\mathcal{Z}$ given by  
    \begin{equation*}
    \varinjlim (\tD{L_n}, \Phi_{L_n, L_{n+1}}\vert_{\tD{L_n}} ) \subset \mathcal{Z}.
    \end{equation*}
\end{theo}

\begin{proof}
    This now follows from Theorem \ref{theo:XinLi}, Lemma \ref{lemma:condexpindlim}, and Lemma \ref{lemma:normpreserved}.
\end{proof}

The theorem above shows the existence of a $\mathrm{C}^*$-diagonal in $\mathcal{Z}$, but the $\mathrm{C}^*$-diagonals may or may not depend on the choice of the sequence $(L_n)_n$. At this point, we do not have the tools to decide this type of question.\\ 

However, to understand the constructed $\mathrm{C}^*$-diagonal further, we can consider its diagonal dimension in the sense of \cite{LLW}. In fact, the diagonal dimension here will be $1$. To see this, it suffices to compute the diagonal dimension of $\tD{L} \subset \tZ{L}{L+1}$ for $L \in \N$ since 
\begin{align*}
    &\mathrm{dim}_{\mathrm{diag}} (\tilde{D} \subset \mathcal{Z}) \leq \liminf_{n \in \N}( \mathrm{dim}_{\mathrm{diag}} (\tD{L_n} \subset \tZ{L_n}{L_n+1})), \\
   &  \mathrm{dim}_{\mathrm{diag}} (\tilde{D} \subset \mathcal{Z})\geq  \mathrm{dim}_{\mathrm{\mathrm{nuc}}}\mathcal{Z}  = 1,
\end{align*}
see \cite{LLW}. In prime dimension drop algebras, one can even compute the diagonal dimension with respect to any diagonal. Here, we use the classical presentation of dimension drop algebras from (\ref{eq:defidimdrop}).

\begin{prop}
    Let $D \subset Z_{p,q}$ be a $\mathrm{C}^*$-diagonal for $p,q \in \N$ coprime. Then, we have $\mathrm{dim}_{\mathrm{diag}}(D \subset Z_{p,q})=1.$ In particular, this implies $\mathrm{dim}_{\mathrm{diag}}(\tilde{D} \subset \mathcal{Z})=1$ for the $\mathrm{C}^*$-diagonal $\tilde{D}$ constructed in Theorem \ref{theo:diaginZ}.
\end{prop}

\begin{proof}
 First, note that $\mathrm{dim}_{\mathrm{diag}}(D \subset Z_{p,q})\geq 1$ since $\mathrm{dim}_{\mathrm{nuc}}Z_{p,q}\geq 1$. Thus, it suffices to show $\mathrm{dim}_{\mathrm{diag}}(D \subset Z_{p,q})\leq 1$. By the classification result in \cite{BaRa}, we may assume that $D$ is of the form
 \begin{equation}
 \label{eq:proof:prop:diagdim:1}
     D= \{ f \in Z_{p,q} : f(t) \in u_t \mathbb{D}_p \otimes \mathbb{D}_q u_t^* \text{ for } t \in [\tfrac{1}{3}, \tfrac{2}{3}],\, f(t) \in \mathbb{D}_p \otimes \mathbb{D}_q \text{ else}\}
 \end{equation}
 where $(u_t)_{t \in [\frac{1}{3}, \frac{2}{3}]}$ is a path of unitaries in $\mathbb{M}_p \otimes \mathbb{M}_q$ with $u_{\frac{1}{3}}$ being a permutation matrix and $u_{\frac{2}{3}}=1$. In particular, we have $u_{\frac{1}{3}}\mathbb{D}_p \otimes \mathbb{D}_q u_{\frac{1}{3}}^* \subset \mathbb{D}_p \otimes \mathbb{D}_q$. Extend this path to a unitary $u= (u_t)_{t \in [0,1]} \in C([0,1], \mathbb{M}_p \otimes \mathbb{M}_q) $ with $u_t=u_{\frac{1}{3}}$ on $[0, \tfrac{1}{3}]$ and $u_t= 1$ on $[\tfrac{2}{3}, 1]$. Now, we define functions $h_{t_0, \delta}\in C([0,1])$ for $t_0 \in [0,1]$, $ \delta >0$ by
 \begin{equation*}
     h_{t_0, \delta} (t) = \begin{cases}
         \frac{t-t_0+\delta}{\delta} & \text{ if }t \in [t_0-\delta, t_0]\cap [0,1],\\
         \frac{t_0-t + \delta}{\delta} & \text{ if } t \in [t_0, t_0+ \delta]\cap [0,1],\\
         0  & \text{ else }.
     \end{cases}
 \end{equation*}
With these, we define c.p.c.\@ maps for $n \in \N$ by
\begin{gather*}
    \psi_n:Z_{p,q} \to \mathbb{M}_p \oplus (\bigoplus_{i=1}^{2n-1} \mathbb{M}_p\otimes \mathbb{M}_q ) \oplus \M_q, \\ \psi_n(f) = (f(0), (u_{\frac{1}{2n}}^* f(\tfrac{1}{2n})u_{\frac{1}{2n}}, \ldots, u_{\frac{2n-1}{2n}}^* f(\tfrac{2n-1}{2n})u_{\frac{2n-1}{2n}}), f(1)).
    \end{gather*}
Put
\begin{gather*}
    F_n:= \mathbb{M}_p \oplus (\bigoplus_{i=1}^{2n-1} \mathbb{M}_p\otimes \mathbb{M}_q ) \oplus \M_q,\\ 
    F_n^{(0)}:=\mathbb{M}_p \oplus (\bigoplus_{i=1, \,i \text{ even}}^{2n-1} \mathbb{M}_p\otimes \mathbb{M}_q ) \oplus \M_q, \;\;\;\;  F_n^{(1)}:=\bigoplus_{i=1, \,i \text{ odd}}^{2n-1} \mathbb{M}_p\otimes \mathbb{M}_q,
\end{gather*}
and write $D_{F_n} \subset F_n, D_{F_n^{(j)}}\subset F_n^{(j)},$ $ j =0,1$, for the canonical diagonals. Now, we define maps $\varphi_n :=\varphi_n^{(0)}\oplus \varphi_n^{(1)}: F_n^{(0)} \oplus F_n^{(1)} \to Z_{p,q}$ for $n \in \N$ by
\begin{align*}
     \varphi_n^{(0)}(x, (x_2, \ldots, x_{2n-2}), y) &
     = h_{0, \frac{1}{2n}}\otimes x \otimes 1_q + h_{1, \frac{1}{2n}}\otimes 1_p \otimes y + \sum\limits_{i=1}^{n-1} u(h_{\frac{2i}{2n}, \frac{1}{2n}}\otimes x_{2i}) u^*,\\
     \varphi_n^{(1)}( x_1,\ldots, x_{2n-1})&
     = \sum\limits_{i=1}^{n} u(h_{\frac{2i-1}{2n}, \frac{1}{2n}}\otimes x_{2i-1}) u^*.    
\end{align*}
Note that $\varphi_n$ is c.p.\@ and $\varphi_n^{(0)}, \varphi_n^{(1)}$ are c.p.c.\@ order zero.\\
Let us now check that $(\psi_n, \varphi_n, F_n)_n$ witnesses $\mathrm{dim}_{\mathrm{diag}}  (D \subset Z_{p,q} ) \leq 1.$ Indeed, we have $\psi_n(1) = 1_{D_{F_n}}$ which implies $\psi_n(D) \subset D_{F_n}$ by Remark 2.4 in \cite{LLW}. Next, we observe that for $v_0 \in N_{\M_p}(\mathbb{D}_p), v_1 \in N_{\M_q}(\mathbb{D}_q)$, the elements $h_{0, \frac{1}{2n}} \otimes v_0 \otimes 1_q, h_{1, \frac{1}{2n}}\otimes 1_p \otimes v_{1}$ normalise $D$ for $n \geq 2$, since $D$ is of the form described in (\ref{eq:proof:prop:diagdim:1}). Further, for $d \in D$, $h \in C_0((0,1))$, and $ w \in N_{\M_p\otimes\M_q}(\mathbb{D}_p\otimes \mathbb{D}_q)$ we get
\begin{equation*}
   u(h \otimes w) u^* du(\bar{h} \otimes w^*)u^*(t)= \begin{cases}
       \vert h(t) \vert^2 u_\frac{1}{3} wu_\frac{1}{3}^* d(t) u_\frac{1}{3}w^*u_\frac{1}{3}^* & \text{ if } t \in [0, \tfrac{1}{3}],\\
        \vert h(t) \vert^2 u_t wu_t^* d(t) u_tw^*u_t^* & \text{ if } t \in [\tfrac{1}{3}, \tfrac{2}{3}],\\
         \vert h(t) \vert^2wd(t) w^* & \text{ if } t \in [\tfrac{2}{3},1],
    \end{cases}
\end{equation*}
and an analogous result for $w^*$ in place of $w$. Therefore, with (\ref{eq:proof:prop:diagdim:1}) and since $u_\frac{1}{3}$ is a normaliser, we have  $ u(h \otimes w) u^*\in N_{Z_{p,q}}(D)$, and hence
\begin{equation*}
   \varphi_n^{(0)} (N_{F_n^{(0)}}(D_{F_n^{(0)}})) \subset N_{Z_{p,q}}(D), \;\;\; \varphi_n^{(1)} (N_{F_n^{(1)}}(D_{F_n^{(1)}})) \subset N_{Z_{p,q}}(D), \; \; n \geq 2,
\end{equation*}
as the summands of $\varphi_n^{(0)}, \varphi_n^{(1)}$ are mutually orthogonal.
Finally, since $u$ is uniformly continuous in each matrix unit, we observe for $f \in Z_{p,q}$, $\varepsilon> 0$, and $n$ large enough that
\begin{align*}
     \varphi_n \circ \psi_n(f)&
    = h_{0, \frac{1}{2n}}\otimes f(0) \otimes 1_q + h_{1, \frac{1}{2n}}\otimes 1_p \otimes f(1) + \sum\limits_{i=1}^{2n-1} u(h_{\frac{i}{2n}, \frac{1}{2n}}\otimes u_{\frac{i}{2n}}^*f(\tfrac{i}{2n}) u_{\frac{i}{2n}}) u^*\\
   & \approx_\varepsilon h_{0, \frac{1}{2n}}\otimes f(0) \otimes 1_q + h_{1, \frac{1}{2n}}\otimes 1_p \otimes f(1) + \sum\limits_{i=1}^{2n-1} h_{\frac{i}{2n}, \frac{1}{2n}}\otimes f(\tfrac{i}{2n}) \overset{n }{\longrightarrow} f.
\end{align*}
This shows that $(\psi_n, \varphi_n, F_n)_n$ is a c.p.c.\@ approximation of $Z_{p,q}$, which finishes the proof.
\end{proof}

\section[]{The spectrum }
\label{sec:spectrum}
To further analyse the $\mathrm{C}^*$-diagonal in the Jiang--Su algebra resulting from Theorem \ref{theo:diaginZ} let us investigate its spectrum. This is given by the inverse limit of the spectra throughout the inverse system constructed in Theorem \ref{theo:constrindlim}. So let us first determine these spaces. \\

Let $L \in \N$ with $L \geq 2$ and consider the compact space $\{1, \ldots, L\} \times \{0, \ldots, L\} \times [0,1]$. Write 
\begin{equation*}
X_{i,j}:=\{i \} \times \{j\} \times [0,1], \,\,\,0_{i,j}:= (i,j,0) \in X_{i,j},\,\,\, 1_{i,j}:= (i,j,1) \in X_{i,j}
\end{equation*}
for $1\leq i \leq L$, $ 0 \leq j \leq L$. Moreover, let $ \sim$ denote the equivalence relation generated by 
\begin{align*}
    \begin{cases}
    0_{i,j} \sim 0_{k,l}& \;\text{ if } j =l, 1 \leq j,l \leq L,\\
    0_{i,j} \sim 0_{k,l}& \;\text{ if } l=0, k=j,\\
    1_{i,j} \sim 1_{k,l} &\; \text{ if } j = l.\\
    \end{cases}
\end{align*}

With that, let us define 
\begin{equation*}
    X_L:= \{1, \ldots, L\} \times \{0, \ldots, L\} \times [0,1]/ \sim.
\end{equation*}

\begin{figure}[h!]
\centering

\begin{tikzpicture}[scale=1, transform shape]

\foreach \x in {0,1,2,3} {

\draw[gray] (5.25,1.5) --  (3,\x);
}

\foreach \x in {0,1,2,3} {
\foreach \i in {0,1,2,3} {
       
\draw[gray] (0,\x) to[in=157.5 +\i*15, out=22.5 -\i*15] (3,\x);

}
}

\filldraw[black] (0,3) circle(1pt) node[black, left]{$1_1^4$\hspace{2pt} };
\filldraw[black] (0,2) circle(1pt) node[black, left]{$1_2^4$\hspace{2pt} };
\filldraw[black] (0,1) circle(1pt) node[black, left]{$1_3^4$\hspace{2pt} };
\filldraw[black] (0,0) circle(1pt) node[black, left]{$1_4^4$\hspace{2pt}};

\filldraw[black] (3,3) circle(1pt) node[black, above]{\hspace{10pt}$0_1^4$ };
\filldraw[black] (3,2) circle(1pt) node[black, above]{\hspace{10pt}$0_2^4$ };
\filldraw[black] (3,1) circle(1pt) node[black, below]{\hspace{10pt}$0_3^4$ };
\filldraw[black] (3,0) circle(1pt) node[black, below]{\hspace{10pt}$0_4^4$ };

\filldraw[black] (5.25,1.5) circle(1pt) node[black, right]{ $1_0^4$};

\end{tikzpicture}
\caption{Picture of $X_4$.\label{fig:spectrum}}
\end{figure}
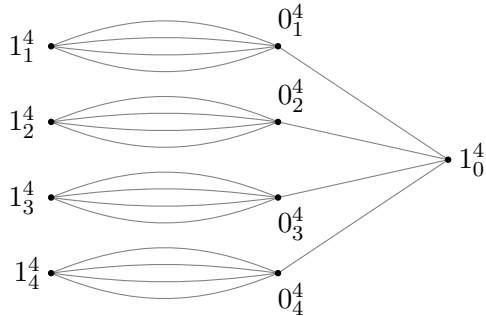
We write 
\begin{eqnarray*}
    & Q^{L} : \coprod_{k,l} X_{k,l}  \twoheadrightarrow X_L, \;\; Q_{i,j}^L : X_{i,j} \hookrightarrow \coprod_{k,l} X_{k,l}  \twoheadrightarrow X_L,\;\;
     1_j^{L}= Q^L(1_{1,j}), \;\; 0_i^{L}= Q^L(0_{i,0}).&
\end{eqnarray*}
If the choice of $L$ is clear from context, we will just write $Q, Q_{i,j}, 1_j$ and $0_i$.
Note that $X_L$ is one-dimensional, compact, metrisable,  path-connected, and locally path-connected, i.e.\@ each point has a neighbourhood base of path-connected sets. In fact, it turns out that this space is (homeomorphic to) the spectrum of $\tD{L}$.
\begin{prop}
\label{prop:spectrumfinitestage}
    Let $L \in \N$ with $L \geq 2$. Then, we have a homeomorphism
    \begin{equation*}
        \beta: PS(\tD{L}) \to X_L,  
    \end{equation*}
    where $PS(\tD{L})$ is equipped with the weak$^*$-topology. That is, the character spectrum of $\tD{L}$ is homeomorphic to $X_L$. 
\end{prop}

\begin{proof}
 In the following, we use the setup and notation from  Proposition \ref{prop:essideal}.
Let us define $\beta: PS(\tD{L}) \to X_L$ by
   \begin{equation*}
       \beta(\rho) = \begin{cases}
           1_j, & \text{if } \rho \circ \sigma_{L+1}\vert_{\mathbb{D}_{L+1}} = \mathrm{ev}_j \text{ for } 0 \leq j \leq L, \\
           0_i, & \text{if } \rho \circ \sigma_{L}\vert_{\mathbb{D}_L} = \mathrm{ev}_i \text{ for } 1 \leq i\leq L, \\
           Q_{i,j}(t) & \text{if } \rho \circ \alpha' = \mathrm{ev}_{t,i,j} \text{ for } t \in (0,1), 1\leq i \leq L, 0 \leq j \leq L.     \end{cases}
   \end{equation*}
Here, $\mathrm{ev}_i, \mathrm{ev}_j, \mathrm{ev}_{t,i,j}$ are the canonical pure states on $\mathbb{D}_L$, $\mathbb{D}_{L+1}$, $C_0((0,1))\otimes \mathbb{D}_L\otimes \mathbb{D}_{L+1}$, respectively. $\beta $ is well-defined since each $\rho \in PS(\tD{L})$ either factorises through $\D_L$ as $\rho\circ \sigma_L \vert_{\D_L}$,  through $\D_{L+1}$ as $\rho \circ \sigma_{L+1}\vert_{\D_{L+1}}$, or restricts to a pure state on $J \cap \tD{L}$.\\
Moreover, with the isomorphisms from (\ref{eq:isosdiagonals}), we get that $\beta$ is bijective. It remains to show that $\beta $ is continuous. Let $(\rho_n)_{ n}$ be a sequence of pure states with $\rho_n \to \rho$ in the weak$^*$-topology for some $\rho \in PS(\tD{L}).$ First, suppose $\rho$ annihilates $I_1$. Then, all but finitely many $\rho_n$ restrict to a pure state on $J \cap \tD{L}$ or satisfy $\rho_n = \rho$. Thus, we may assume that each $\rho_n$ lives on $J \cap \tD{L}$, i.e.\@ we have corresponding pure states $\rho_n \circ \alpha' =\mathrm{ev}_{t_n, i_n ,j_n}.$\\
   Now, let us first suppose that $\rho\circ \sigma_{L+1}\vert_{\D_{L+1}} = \mathrm{ev}_0$ on $\mathbb{D}_{L+1}.$ We note that $\rho(f(\bs^*\bs))=1 $ for every $f \in C_0((0,1])$ with $f(1)=1$ since $a(\bs^*\bs) -f(\bs^*\bs) \in I_1$ and $a(\bs^*\bs)=\sigma_{L+1}(e_{0,0})$.  Take $\varepsilon >0$ and let $f_\varepsilon \in C_0((0,1])$ with $f_\varepsilon(1)=1$ and $\mathrm{supp}(f_\varepsilon)\subset (1- \varepsilon,1]$. Under the isomorphism $\alpha'$ from (\ref{eq:isosdiagonals}), we have 
   \begin{equation*}
       \alpha'(abf_\varepsilon  \otimes 1_L \otimes e_{0,0}) = (a f_\varepsilon)(\hat{\sigma}_{L+1}(1)) b(\bs^*\bs) = (af_\varepsilon) (\bs^*\bs) \theta(1).
   \end{equation*}
   This implies $$\mathrm{ev}_{t_n, i_n, j_n} (abf_\varepsilon  \otimes 1_L \otimes e_{0,0}) = \rho_n  ((af_\varepsilon)(\bs^*\bs) \theta(1)) >0$$
   for $n$ large enough since $\rho_n(\theta(1)) > 0$ for all $n$, and $\rho_n ((af_\varepsilon)(\bs^*\bs)) \to \rho((af_\varepsilon)(\bs^*\bs) ) =1$. In particular, $j_n = 0$ and $t_n \in (1-\varepsilon,1]$ for large enough $n$. Since $\varepsilon$ was chosen arbitrarily, we get $t_n \to 1$ and conclude that
   \begin{equation*}
       \beta(\rho_n) =Q_{i_n,j_n}(t_n) \overset{n}{\longrightarrow} 1_{0} = \beta(\rho).
   \end{equation*}
   If $\rho$ corresponds to $\mathrm{ev}_i$ for some $1 \leq i \leq  L$, we can argue analogously.\\
   Next, consider the case that $\rho$ annihilates $I_2$, i.e.\@ $\rho \circ \sigma_L\vert_{\D_{L}} =\mathrm{ev}_i$ on $\mathbb{D}_L$ for some $1\leq i \leq L$. As above, we may assume that each $\rho_n$ restricts to a pure state on $J \cap \tD{L}$, i.e.\@ we have  pure states $\rho_n \circ \alpha' =\mathrm{ev}_{t_n,i_n,j_n}$, $n \in \N$. Note that $\rho(f(\bc_i^*(1-\bs\bs^*)\bc_i))=1$ for all $f \in C_0((0,1])$ with $f(1)=1$ as $(a-f)(\bc_i^*(1-\bs\bs^*)\bc_i) \in I_2$ and $\sigma_L(e_{i,i})=a(\bc_i^*(1-\bs\bs^*)\bc_i)$. Let $\varepsilon > 0$ and $f_\varepsilon \in C_0((0,1])$ as above. Write $\tilde{f}_\varepsilon:=f_\varepsilon\circ(1- \mathrm{id})$ and  $\tilde{a}:= a \circ (1 - \mathrm{id})$.
Similarly to the previous case, we then get
   \begin{equation*}
       \mathrm{ev}_{t_n, i_n, j_n} ( \tilde{a}b\tilde{f}_\varepsilon  \otimes 1_L \otimes e_{i,i} + \tilde{a}b \tilde{f}_\varepsilon  \otimes e_{i,i} \otimes e_{0,0}) = \rho_n( (af_\varepsilon)(\bc_i^*(1-\bs\bs^*)\bc_i ) \theta(1)) >0
   \end{equation*}
   for $n$ sufficiently large. Thus, we have $t_n \in (0, \varepsilon)$ and $(i_n, j_n) \in \{ 1, \dots , L\} \times \{i\} \cup \{(i,0)\}$ for large $n$. As $\varepsilon$ was arbitrarily small, this shows $t_n \to 0$ and hence $\beta(\rho_n) \to 0_i = \beta (\rho)$.\\
   If now $\rho$ does not vanish on $J \cap \tD{L}$, 
   then the same holds for all but finitely many $\rho_n$. With this and the correspondence to pure states on $C_0((0,1)) \otimes \mathbb{D}_L \otimes \mathbb{D}_{L+1}$ via $\alpha'$, one checks that $\beta(\rho_n) \to \beta(\rho).$
   \end{proof}
The next step is to understand the maps induced on spectral level by the connecting maps $\Phi_{L,L'}$ from Proposition \ref{theo:starhom} with $L'\geq L$ as in (\ref{eq:choiceofL'}). One way to do that is to analyse the behaviour of pure states after composing with $\Phi_{L,L'}$. To that end, let us first describe sequences that excise the prominent pure states on $\tD{L}$. For $n \in \N$, define functions on $[0,1]$ by
\begin{equation}
\label{eq:defilnmnknt}
    l_n(t) = \begin{cases}
        0& \text{ if } t \in [0, 1-\tfrac{1}{n}],\\
        n(t-(1-\frac{1}{n})) & \text{ if } t \in [1 -\tfrac{1}{n}, 1],
    \end{cases}  
     \,\,\,\,\,\,\,
      m_n(t) = \begin{cases}
        nt & \text{ if } t \in [0, \tfrac{1}{n}],\\
        1 & \text{ if } t \in [\tfrac{1}{n}, 1].
    \end{cases}  
\end{equation}
With these, we can write out excising sequences for the pure states on $\tD{L}$ evaluating in the endpoints of the interval segments in $X_L$ (cf.\@ Figure \ref{fig:spectrum}).
\begin{lemma}
\label{lemma:5excseq}
    Let $L \geq 2 $. Then, we have
    \begin{spacing}{1.3}
    \begin{itemize}
        \item[\upshape{(i)}] $(l_n(\bs^*\bs))_n$ excises $\beta^{-1}(1_0)$,
        \item[\upshape{(ii)}] $(\bc_i^* l_n(\bs\bs^*)\bc_i)_n $ excises $\beta^{-1}(1_i)$ for  $1 \leq i \leq L$, 
        \item[\upshape{(iii)}]  $(l_n(\bc_j^*\bc_j)( 1 -  \bc_j^*m_n(\bs\bs^*)\bc_j))_n$ excises $\beta^{-1}(0_j)$ for $1 \leq j \leq L$,
    \end{itemize}
    \end{spacing}
    where  $\bc_1, \ldots, \bc_L, \bs$ denote the generators of $\tZ{L}{L+1}$ and $\beta$ is the homeomorphism from \ref{prop:spectrumfinitestage}. 
\end{lemma}

\begin{proof}
We again use the notation from Proposition  \ref{prop:essideal} and write $\iota:= \mathrm{id}\in C_0((0,1])$. \\
(i): $\beta^{-1}(1_0) $ annihilates $I_1 \cap \tD{L}$ and factorises as $\beta^{-1}(1_0) = \mathrm{ev}_{0} \circ Q_{L+1}$ where $Q_{L+1}$ denotes the quotient map $ \tD{L} \twoheadrightarrow \tD{L} /( I_1 \cap \tD{L})$. We may view $Q_{L+1}$ as mapping onto $\mathbb{D}_{L+1}$ by (\ref{eq:isosdiagonals}). 
To show that $(l_n(\bs^*\bs)_n)$ excises $\beta^{-1}(1_0)$, it suffices, by Proposition 2.2 in \cite{AAP}, to check that $(1-l_n(\bs^*\bs))_n$ is an approximate unit for $\ker(\beta^{-1}(1_0))$.
So take $y \in \ker(\beta^{-1}(1_0))$ and first suppose $y \in I_1$. Then, there exists $N \in \N$ with $l_N(\bs^*\bs)y \in J$ as  $l_n(\bs^*\bs) \sigma_{L}(1) = 0$ for sufficiently large $n$. 
Further, observe that $l_n(\bs^*\bs) \theta(1)\to 0$, whence $l_n(\bs^*\bs) x \to 0$ for any $x \in J$. Combining these two observations, we get 
\begin{equation*}
 \lim\limits_{n \to \infty}(1-l_n(\bs^*\bs)) y = y - \lim\limits_{n \to \infty}l_{n}(\bs^*\bs)l_N(\bs^*\bs)y= y.
\end{equation*}
Now, let $y \notin I_1$, i.e.\@ $Q_{L+1}(y) \neq 0$ but $\beta^{-1}(1_0)(y) = \mathrm{ev}_{0}(Q_{L+1}(y)) =0$. Then, we have 
\begin{equation*}
Q_{L+1}(y) = Q_{L+1}( \sum\limits_{i =1}^L \lambda_i \sigma_{L+1}(e_{i,i})),
\end{equation*}
for some $\lambda_i \in \C$. Hence, there is $z \in I_1\cap \tD{L}$ with $y = \sum_{i =1}^L \lambda_i \sigma_{L+1}(e_{i,i}) +z$. We get $(1-l_n(\bs^*\bs))y \to y$, as $l_n(\bs^*\bs) \sigma_{L+1}(e_{i,i}) \to 0 $  for $i \geq 1$ and we already checked $l_n(\bs^*\bs) z \to 0$.\\
(ii): $\beta^{-1}(1_i)$ for $1\leq i \leq L$ annihilates $I_1\cap \tD{L}$ and $\beta^{-1}(1_i) =\mathrm{ev}_i \circ Q_{L+1}$ on $\mathbb{D}_{L+1}$. We observe
$\bc_i^*l_n(\bs\bs^*)\bc_i \theta(1) \to 0$, and $\bc_i^*l_N(\bs\bs^*)\bc_i \sigma_{L}(1) = 0$
for $N \in \N$ sufficiently large. As above, this implies that $(1-\bc_i^*l_n(\bs\bs^*)\bc_i)_n$ is an approximate unit of $\ker(\beta^{-1}(1_i))$ which shows the assertion.\\
(iii): $\beta^{-1}(0_j)$ for $1\leq j \leq L$ annihilates $I_2 \cap \tD{L}$ and $\beta^{-1}(0_j)= \mathrm{ev}_j \circ Q_L$ on $\mathbb{D}_L$ where $Q_L$ denotes the quotient map onto $\D_{L}$.
Note that 
\begin{equation*}
l_N(\bc_j^*\bc_j)( 1 -  \bc_j^*m_N(\bs\bs^*)\bc_j) \sigma_{L+1}(1) =0
\end{equation*} 
for $N \in \N$ sufficiently large. Further, we have 
\begin{equation*}
    l_n(\bc_j^*\bc_j)b(\bs^*\bs)  \overset{n}{\longrightarrow}0, \;\;
    l_n(\bc_j^*\bc_j)( 1 -  \bc_j^*m_n(\bs\bs^*)\bc_j) b(\bc_j^*\bs\bs^*\bc_j) \overset{n}{\longrightarrow} 0.
\end{equation*}
This implies $l_n(\bc_j^*\bc_j)( 1 -  \bc_j^*m_n(\bs\bs^*)\bc_j) \theta(1) \to 0$. With these observations, one shows as in (i) and (ii) that $(1-l_n(\bc_j^*\bc_j)( 1 -  \bc_j^*m_n(\bs\bs^*)\bc_j))_n $ is an approximate unit on $\ker(\beta^{-1}(0_j))$.
\end{proof}
To see where the pure states considered in the latter lemma are mapped by the spectral connecting maps, it now suffices to consider the images of the respective excising sequences. 
\begin{lemma}
\label{lemma:toolconnectingspectralmaps}
    Let $L \geq 2$ and $\Phi_{L,L'}: \tZ{L}{L+1} \to \tZ{L'}{L'+1}$ as in Proposition \ref{theo:starhom} with $L'$ chosen as in (\ref{eq:choiceofL'}). Further, let $\rho_x \in PS(\tD{L'})$ be the pure state in some $x \in X_{L'}$ and $\psi_y \in PS(\tD{L})$ in some $y \in X_{L}$. Suppose there are sequences $(d_n')_n$ in $\tD{L'}$ excising $\rho_x$ and $(d_n)_n$ in $\tD{L}$ excising $\psi_y$ such that
    $$\lim\limits_{n \to \infty} d_n' - \Phi_{L,L'}(d_n) d_n'  = \lim\limits_{n \to \infty}(1- \Phi_{L,L'}(d_n))d_n' =0.$$
    Then, we have $\rho_x \circ \Phi_{L,L'} = \psi_y.$
\end{lemma}

\begin{proof}
    We prove this via contraposition. So assume that $\rho_x \circ \Phi_{L,L'} \neq \psi_y$. Take $(\bar{d}_n)_n$ in $\tD{L}$ excising $\rho_x \circ \Phi_{L,L'}$ with $\rho_x \circ \Phi_{L,L'}(\bar{d}_n)=1$ for each $n \in \N$ and let $(d_n)_n$ in $\tD{L}$, $(d_n')_n$ in $\tD{L'}$ be any sequences excising $\psi_y, \rho_x$, respectively. Then, we have 
    \begin{equation*}
   \lim\limits_{n \to \infty} \norm{d_n' \Phi_{L,L'}(\bar{d}_i) d_n'- (d_n')^2} = \lim\limits_{n \to \infty} \norm{d_n' \Phi_{L,L'}(\bar{d}_i) d_n'- \rho_x \circ \Phi_{L,L'}(\bar{d}_i) (d_n')^2} = 0
    \end{equation*}
    for any $i \in \N$ as $(d_n')_n$ excises $\rho_x$. Thus, there is a subsequence $(d'_{n_i})_n$ with 
    \begin{equation}
    \label{eq:proof:lemma:homspectra}
    \lim\limits_{i \to \infty} \norm{\Phi_{L,L'}(\bar{d}_{i}) d_{n_i}'  - d_{n_i}'}= 0.
    \end{equation}
    Further, by Lemma \ref{lemma:charexcseq}, we have $d_{n_i} \bar{d}_{i} \to 0$ as they excise different pure states. We get
    \begin{equation*}
        \liminf\limits_{i \in \N} \norm{\Phi_{L,L'}(d_{n_i}) d_{n_i}'-d_{n_i}'} \overset{(\ref{eq:proof:lemma:homspectra})} =  \liminf\limits_{i \in \N} \norm{\Phi_{L,L'}(d_{n_i} \bar{d}_{i} )d_{n_i}' -d_{n_i}'} = 
        \liminf\limits_{i \in \N} \norm{d_{n_i}'} = 1.
    \end{equation*}
    This implies $\Phi_{L,L'}(d_{n}) d_n'-d_n' \nrightarrow 0$.
\end{proof}

The latter two lemmas now allow us to check where some of the prominent points in the spectrum $X_L$ are mapped.

\begin{lemma}
\label{lemma:wherearepeakfctsmapped}
   Let $L,K,M \geq 2$ and $L', K'$ as in (\ref{eq:choiceofL'}). Further, let $\Phi_{L,L'}: \tZ{L}{L+1} \to \tZ{L'}{L'+1}$ as in Proposition \ref{theo:starhom} and $f: X_{L'} \to X_L$ be the induced map between the spectra of the respective diagonals. Then, we have
   \begin{equation*}
   f(1_0^{L'}) = 1_0^{L}, \;\; f(0_1^{L'} ) = 0_1^L= f(1^{L'}_{K'})  .
   \end{equation*}
\end{lemma}

\begin{proof}
   In the following, we will use the notation established in Section \ref{sec:construction}.
   In particular, we write $\bar{c}_1, \ldots, \bar{c}_L, \bar{s}$ for the generators of $\tZ{L}{L+1}$ and $c_1, \ldots, c_{L'}, s$ for the generators of $\tZ{L'}{L'+1}$.
   Further, we write $\beta_L, \beta_{L'}$ for the homeomorphisms from the character spectra of $\tD{L},$ $ \tD{L'}$ to $X_L, X_{L'}$, respectively. Moreover, we 
   write $\iota:= \mathrm{id} \in C_0((0,1]).$\\
   Now, by Lemma \ref{lemma:5excseq}, $\beta_L^{-1}(1_0^L) \in PS(\tD{L})$ is excised by the sequence $(l_n(\bs^* \bs))_n$ where the functions $l_n$ are defined as in (\ref{eq:defilnmnknt}). Recall $\Phi_{L,L'}(\bar{s})=\tilde{s}$ from (\ref{eq:defi:tildes}) and observe that
    \begin{equation*}
        \Phi_{L,L'}(\bs^*\bs) = h(s^*s) + \sum\limits_{l=1}^L \sum\limits_{i=1}^M \frac{i}{M} h (\cstarc{lK'+(l-1)M+i})
    \end{equation*}
    with pairwise orthogonal summands. Now, choose  a subsequence $(l_n')_n$ of $(l_n)_n$ with 
     \begin{equation*}
       \lim\limits_{n \to \infty} (1-l_n \circ h)  l_{n}' =0.
     \end{equation*}
     Then, the sequence $(l'_n(s^*s))_{n}$ excises $\beta_{L'}^{-1}(1_0^{L'})$ by Lemma \ref{lemma:5excseq}. We compute
    \begin{equation*}
      \lim\limits_{n \to \infty} \big(1- \Phi_{L,L'}(l_n(\bs^*\bs))\big) l_n'(s^*s) \overset{(\ref{eq:polynomialtrick3})}{=} \lim\limits_{n \to \infty}(1-l_n\circ h(s^*s)) l_n'(s^*s) = 0.
    \end{equation*}
    This implies $f(1_0^{L'})= 1_0^L$ with Lemma  \ref{lemma:toolconnectingspectralmaps}.    
    For $f(0_1^{L'})=0_1^L$, we note with (\ref{eq:defi:tildec}), (\ref{eq:defi:tildes}) that
    \begin{align*}
        l_n(\Phi_{L,L'}(\bc_1^2)) &= \sum\limits_{i=1}^{K'}l_n\circ g(c_i^*c_i) + \sum\limits_{i=1}^M l_n\circ(g - \frac{i}{M} h)(\cstarc{K'+i}), \\
        m_n(\Phi_{L,L'}(\bs\bs^*)) & = m_n\circ h(ss^*) + \sum\limits_{l=1}^L \sum\limits_{i=1}^M m_n \circ (\frac{i}{M} h) (\cstarc{(l-1)M+i+1}),
    \end{align*}
    where the functions $l_n, m_n$ are defined as in (\ref{eq:defilnmnknt}). We have for sufficiently large $n \in \N$ that 
    \begin{equation*}
    (1-l_n\circ g) l_n = 0,\;\; (m_n\circ h )  m_n=m_n \circ h.
    \end{equation*}
    This implies
    \begin{align*}
        \big(1 - \Phi_{L,L'}(l_n(\bc_1^2))\big) l_n(c_1^2) \overset{(\ref{eq:polynomialtrick1})}{=} (1 - l_n\circ g(c_1^2)) l_n(c_1^2) \overset{n}{\longrightarrow} 0,
    \end{align*}
    and, using $ss^*c_1=ss^*$,
    \begin{align*}
        \Phi_{L,L'}(m_n(\bs\bs^*)) l_n(c_1^2)(1- m_n(ss^*))  \overset{(\ref{eq:polynomialtrick1})}&{=} m_n\circ h (ss^*) - m_n \circ h(ss^*) m_n(ss^*) = 0.
    \end{align*}
    Combining these yields 
    \begin{equation*}
       \lim\limits_{n \to \infty} \big(1-\Phi_{L,L'}(l_n(\bc_1^2)(1-m_n(\bs\bs^*))) \big) l_n(c_1^2)(1-m_n(ss^*)) = 0
    \end{equation*}
    and hence $f(0_1^{L'})=0_1^L$ by Lemmas \ref{lemma:5excseq}, \ref{lemma:toolconnectingspectralmaps}. 
    Finally, to show $f(1_{K'}^{L'})= 0_1^L$, we compute
    \begin{equation*}
      \Phi_{L,L'}(l_n(\bc_1^2)) l_n(c_{K'}^* ss^* c_{K'})  = l_n\circ g( \cstarc{K'}) c_{K'}^* l_n(ss^*) c_{K'} = c_{K'}^* l_n(ss^*) c_{K'},  
    \end{equation*}
     using the relations of $\tZ{L'}{L'+1}$. Similarly, recall $m_n(\Phi_{L,L'}(\bs \bs^*))$ from above and observe
    \begin{equation*}
    \Phi_{L,L'}(m_n(\bs\bs^*)) l_n(c_{K'}^* ss^* c_{K'})= 0
    \end{equation*}
    for $n \in \N$. This yields
    \begin{equation*}
    \big(1 - \Phi_{L,L'}(l_n(\bc_1^2)(1-m_n(\bs\bs^*)))\big)   l_n(c_{K'}^* ss^* c_{K'})  = 0,
    \end{equation*}
    for $n \in \N,$ and hence $f(1_{K'}^{L'}) = 0_1^{L}$ by Lemmas \ref{lemma:5excseq}, \ref{lemma:toolconnectingspectralmaps}. 
    \end{proof}

With these preparations, we can now deduce some structural results about the spectrum of the $\mathrm{C}^*$-diagonal in the Jiang--Su algebra from Theorem \ref{theo:diaginZ}. To do that, fix one choice of a sequence $(L_n)_n$ such that $\varinjlim \tZ{L_n}{L_n+1} \cong \mathcal{Z}$ as in Theorem \ref{theo:constrindlim} and denote the $\mathrm{C}^*$-diagonal $\varinjlim \tD{L_n} \subset \mathcal{Z}$ from Theorem \ref{theo:diaginZ}  by $\tilde{D}$. 
Further, let $X_{L_{n}} $ denote the spectrum  of the diagonal $\tD{L_n}$ from Proposition \ref{prop:tD}. The injective $^*$-homomorphisms $\Phi_{L_n,L_{n+1}}\vert_{\tD{L_n}}$ induce surjective, continuous maps $f_{n+1,n}:X_{L_{n+1}} \to X_{L_n}$ and the spectrum  of $\tilde{D}$  is homeomorphic to the inverse limit 
\begin{equation*}
     X:= \varprojlim (X_{L_n}, f_{n+1,n}) \subset \prod_{n \in \N } X_{L_n}.
\end{equation*}
We denote the projection maps from the inverse limit to the finite stages by $f_n: X \to X_{L_n}$. These are surjective since all the connecting maps are.\\
We immediately observe that $X$ is compact, connected, metrisable, and one-dimensional, as this holds for all $X_{L_n}$. Moreover, $X$ is not locally connected, i.e.\@ there is a point in $X$ that does not admit a neighbourhood base of connected sets.

\begin{prop}
    The spectrum $X$ of $\tilde{D}$ is not locally connected.
\end{prop}

\begin{proof}
    Define $x = (x_n)_{n } = (0_1^{L_1}, 0_1^{L_2}, 0_1^{L_3}, 0_1^{L_4} \dots ) \in \prod_{n} X_{L_n}$. This is an element of $X$ by Lemma \ref{lemma:wherearepeakfctsmapped}. Take $U=X_{L_1} \backslash \{ 1_0^{L_1} \} \subset {X_{L_1}}$ and $V \subset f^{-1}_1(U) \subset X $ an open neighbourhood of $x$. We show by contradiction that $V$ is not path-connected, so let us assume the contrary. By 2.1.9 in \cite{Macias2018}, there is $n \in \N$ and $V_n \subset X_{L_n}$ an open neighbourhood of $x_n=0_1^{L_n}$ with $f^{-1}_n(V_n) \subset V.$  Observe that 
    \begin{equation*}
       f_{n+1,n}^{-1} (V_n)= f_{n+1}(f_n^{-1} (V_n)) \subset f_{n+1} (V)
    \end{equation*}
    since $f_n = f_{n+1,n} \circ f_{n+1}$ and the projection maps are surjective. In particular, we have $1^{L_{n+1}}_k, 0^{L_{n+1}}_1 \in f_{n+1}(V)$ for some $k \neq 1$ by Lemma \ref{lemma:wherearepeakfctsmapped}. Since $V$ was assumed to be path-connected, there is a path $\gamma: [0,1] \to f_{n+1}(V) \subset f_{n+1,1}^{-1}(U)$ connecting  $1^{L_{n+1}}_k$ and $0^{L_{n+1}}_1$. Since $X_{L_{n+1}} \backslash \{ 1_0^{L_{n+1}} \}$ is not connected and $1^{L_{n+1}}_k$ and $0^{L_{n+1}}_1$ lie in different connected components, we have $1_0^{L_{n+1}} \in \gamma([0,1])\subset  f_{n+1,1}^{-1}(U) $. This implies $1^{L_{1}}_0 \in U$ by Lemma \ref{lemma:wherearepeakfctsmapped}, a contradiction. Thus, there is no path-connected neighbourhood of $x$ contained in $f_1^{-1}(U)$ whence $X$ is not locally connected by 8.25 in \cite{Nadler}.
    \end{proof}
        
An interesting consequence of the previous proposition is that $X$ is not a Peano continuum, i.e.\@ a continuum which is locally connected in each point. In particular, the constructed $\mathrm{C}^*$-diagonal $\tilde{D} \subset \mathcal{Z}$ does not coincide with the examples with spectrum homeomorphic to a Peano continuum from Section 8 in \cite{Li1} and also does not fall in the scope of the machinery developed in \cite{Li2} to produce diagonals with Menger manifold spectrum.

\bibliographystyle{plain}
\bibliography{Mybib}

\begin{thebibliography}{10}

\bibitem{AAP}
Charles~A. Akemann, Joel Anderson, and Gert~K. Pedersen.
\newblock Excising states of {$\mathrm{C}^*$}-algebras.
\newblock {\em Canad. J. Math.}, 38(5):1239--1260, 1986.

\bibitem{AuMi}
Kyle Austin and Atish Mitra.
\newblock Groupoid models of {$\mathrm{C}^*$}-algebras and the {G}elfand
  functor.
\newblock {\em New York J. Math.}, 27:740--775, 2021.

\bibitem{BaLi1}
Sel\c{c}uk Barlak and Xin Li.
\newblock Cartan subalgebras and the {UCT} problem.
\newblock {\em Adv. Math.}, 316:748--769, 2017.

\bibitem{BaLi2}
Sel\c{c}uk Barlak and Xin Li.
\newblock Cartan subalgebras and the {UCT} problem, {II}.
\newblock {\em Math. Ann.}, 378(1-2):255--287, 2020.

\bibitem{BaRa}
Sel\c{c}uk Barlak and Sven Raum.
\newblock Cartan subalgebras in dimension drop algebras.
\newblock {\em J. Inst. Math. Jussieu}, 20(3):725--755, 2021.

\bibitem{CGSTW}
José~R. Carrión, James Gabe, Christopher Schafhauser, Aaron Tikuisis, and
  Stuart White.
\newblock Classifying $^*$-homomorphisms {I}: Unital simple nuclear
  $\mathrm{C}^*$-algebras, {A}r{X}iv preprint,
  \href{https://arxiv.org/abs/2307.06480}{arXiv:2307.06480}, 2023.

\bibitem{CFaH}
Lisa~Orloff Clark, James Fletcher, and Astrid an~Huef.
\newblock All classifiable {K}irchberg algebras are {$\mathrm{C}^*$}-algebras
  of ample groupoids.
\newblock {\em Expo. Math.}, 38(4):559--565, 2020.

\bibitem{DePuStr}
Robin~J. Deeley, Ian~F. Putnam, and Karen~R. Strung.
\newblock Constructing minimal homeomorphisms on point-like spaces and a
  dynamical presentation of the {J}iang--{S}u algebra.
\newblock {\em J. Reine Angew. Math.}, 742:241--261, 2018.

\bibitem{Elliott-Gong-Lin-Niu}
George~A. Elliott, Guihua Gong, Huaxin Lin, and Zhuang Niu.
\newblock On the classification of simple amenable {$\mathrm{C}^*$}-algebras
  with finite decomposition rank, {II}.
\newblock {\em J. Noncommut. Geom.}, 19(1):73--104, 2025.

\bibitem{EvSi2}
Samuel Evington and Philipp Sibbel.
\newblock {Principal groupoid models for stable UCT Kirchberg algebras, ArXiv
  preprint, \href{https://arxiv.org/abs/2605.30147v2}{arXiv:2605.30147}}, 2026.

\bibitem{EvSi1}
Samuel Evington and Philipp Sibbel.
\newblock {${\rm C}^*$}-diagonals with {C}antor spectrum in {C}untz algebras.
\newblock {\em J. Funct. Anal.}, 290(12):Paper No. 111418, 2026.

\bibitem{Exel}
Ruy Exel.
\newblock {On Kumjian's $\mathrm{C}^*$-diagonal and the opaque ideal, ArXiv
  preprint, \href{https://arxiv.org/abs/2110.09445}{arXiv:2110.09445}}, 2021.

\bibitem{RF1}
Jacob Feldman and Calvin~C. Moore.
\newblock Ergodic equivalence relations, cohomology, and von {N}eumann
  algebras.
\newblock {\em Bull. Amer. Math. Soc.}, 81(5):921--924, 1975.

\bibitem{RF2}
Jacob Feldman and Calvin~C. Moore.
\newblock Ergodic equivalence relations, cohomology, and von {N}eumann
  algebras. {I}.
\newblock {\em Trans. Amer. Math. Soc.}, 234(2):289--324, 1977.

\bibitem{RF3}
Jacob Feldman and Calvin~C. Moore.
\newblock Ergodic equivalence relations, cohomology, and von {N}eumann
  algebras. {II}.
\newblock {\em Trans. Amer. Math. Soc.}, 234(2):325--359, 1977.

\bibitem{GongLinNiu}
Guihua Gong, Huaxin Lin, and Zhuang Niu.
\newblock A classification of finite simple amenable {$\mathcal{Z}$}-stable
  {$\mathrm{C}^*$}-algebras, {II}: {$\mathrm{C}^*$}-algebras with rational
  generalized tracial rank one.
\newblock {\em C. R. Math. Acad. Sci. Soc. R. Can.}, 42(4):451--539, 2020.

\bibitem{JaWi}
Bhishan Jacelon and Wilhelm Winter.
\newblock {$\mathcal{Z}$} is universal.
\newblock {\em J. Noncommut. Geom.}, 8(4):1023--1042, 2014.

\bibitem{JiSu}
Xinhui Jiang and Hongbing Su.
\newblock On a simple unital projectionless {$\mathrm{C}^*$}-algebra.
\newblock {\em Amer. J. Math.}, 121(2):359--413, 1999.

\bibitem{RKir}
Eberhard Kirchberg and Mikael R{\o}rdam.
\newblock Infinite non-simple {$\mathrm{C}^*$}-algebras: Absorbing the {C}untz
  algebras {$\mathcal{ O}_\infty$}.
\newblock {\em Adv. Math.}, 167(2):195--264, 2002.

\bibitem{KoWi}
Grigoris Kopsacheilis and Wilhelm Winter.
\newblock {Paper-folding models for the CAR algebra, to appear in
  \textit{Ergod.\@ Theory Dyn.\@ Syst.}, ArXiv preprint,
  \href{https://arxiv.org/abs/2508.04837}{arXiv:2508.04837}}, 2025.

\bibitem{Kum}
Alexander Kumjian.
\newblock On {$\mathrm{C}^*$}-diagonals.
\newblock {\em Canad. J. Math.}, 38(4):969--1008, 1986.

\bibitem{LLW}
Kang Li, Hung-Chang Liao, and Wilhelm Winter.
\newblock {The diagonal dimension of sub-$\mathrm{C}^*$-algebras, ArXiv
  preprint,
  \href{https://doi.org/10.48550/arXiv.2303.16762}{arXiv:2303.16762}}, 2023.

\bibitem{Li1}
Xin Li.
\newblock Every classifiable simple {$\rm \mathrm{C}^*$}-algebra has a {C}artan
  subalgebra.
\newblock {\em Invent. Math.}, 219(2):653--699, 2020.

\bibitem{Li2}
Xin Li.
\newblock Constructing {M}enger manifold {$\mathrm{C}^*$}-diagonals in
  classifiable {$\mathrm{C}^*$}-algebras.
\newblock {\em Int. Math. Res. Not. IMRN}, (23):18992--19053, 2022.

\bibitem{LiRa}
Xin Li and Ali~I. Raad.
\newblock Constructing {$\rm \mathrm{C}^*$}-diagonals in {AH}-algebras.
\newblock {\em Trans. Amer. Math. Soc.}, 376(12):8857--8875, 2023.

\bibitem{LiRe}
Xin Li and Jean Renault.
\newblock Cartan subalgebras in {${\rm C}^*$}-algebras. {E}xistence and
  uniqueness.
\newblock {\em Trans. Amer. Math. Soc.}, 372(3):1985--2010, 2019.

\bibitem{LiaoTikuisis}
Hung-Chang Liao and Aaron Tikuisis.
\newblock {Almost finiteness, comparison, and tracial
  {$\mathcal{Z}$}-stability}.
\newblock {\em J. Funct. Anal.}, 282(3):Paper No. 109309, 2022.

\bibitem{Lor}
Terry~A. Loring.
\newblock {$\mathrm{C}^*$}-algebras generated by stable relations.
\newblock {\em J. Funct. Anal.}, 112(1):159--203, 1993.

\bibitem{Macias2018}
Sergio Mac\'{\i}as.
\newblock {\em Topics on continua}.
\newblock Springer, Cham, second edition, 2018.

\bibitem{Nadler}
Sam~B. Nadler, Jr.
\newblock {\em Continuum theory}, volume 158 of {\em Monographs and Textbooks
  in Pure and Applied Mathematics}.
\newblock Marcel Dekker, Inc., New York, 1992.

\bibitem{Oy}
Dolapo Oyetunbi.
\newblock {$\mathrm{C}^*$-diagonal of inductive limit of $1$-dimensional NCCW
  complexes, ArXiv preprint,
  \href{https://arxiv.org/abs/2505.04011}{arXiv:2505.04011}}, 2025.

\bibitem{Pit}
David~R. Pitts.
\newblock Normalizers and approximate units for inclusions of
  {$\mathrm{C}^*$}-algebras.
\newblock {\em Indiana Univ. Math. J.}, 72(5):1849--1866, 2023.

\bibitem{Ren1}
Jean Renault.
\newblock Cartan subalgebras in {$\mathrm{C}^*$}-algebras.
\newblock {\em Irish Math. Soc. Bull.}, 61:29--63, 2008.

\bibitem{RWi}
Mikael R{\o}rdam and Wilhelm Winter.
\newblock The {J}iang--{S}u algebra revisited.
\newblock {\em J. Reine Angew. Math.}, 642:129--155, 2010.

\bibitem{Rud}
Walter Rudin.
\newblock {\em Real and complex analysis}.
\newblock McGraw-Hill Book Co., New York, third edition, 1987.

\bibitem{Sato}
Yasuhiko Sato.
\newblock The {R}ohlin property for automorphisms of the {J}iang--{S}u algebra.
\newblock {\em J. Funct. Anal.}, 259(2):453--476, 2010.

\bibitem{Sch}
Andr\'e Schemaitat.
\newblock The {J}iang--{S}u algebra is strongly self-absorbing revisited.
\newblock {\em J. Funct. Anal.}, 282(6):Paper No. 109347, 2022.

\bibitem{SiWi}
Philipp Sibbel and Wilhelm Winter.
\newblock A {C}antor spectrum diagonal in {$\mathcal{O}_2$}.
\newblock {\em Proc. Amer. Math. Soc. Ser. B}, 12:210--217, 2025.

\bibitem{Tu}
Jean-Louis Tu.
\newblock La conjecture de {B}aum--{C}onnes pour les feuilletages moyennables.
\newblock {\em $K$-Theory}, 17(3):215--264, 1999.

\bibitem{WiZa}
Wilhelm Winter and Joachim Zacharias.
\newblock Completely positive maps of order zero.
\newblock {\em M\"unster J. Math.}, 2:311--324, 2009.

\end{thebibliography}

\end{document}